\documentclass{birkau}

\usepackage{amsmath,amssymb,url,amsfonts}
\usepackage{mathptmx}
\usepackage{stmaryrd}
\usepackage{enumitem} 

\usepackage{hyperref}
\usepackage{tikz}
\usetikzlibrary{cd}

\usepackage{mathdots}
\usepackage{yhmath}
\usepackage{cancel}
\usepackage{color}
\usepackage{siunitx}
\usepackage{array}
\usepackage{multirow}
\usepackage{gensymb}
\usepackage{tabularx}
\usepackage{booktabs}
\usetikzlibrary{fadings}

\DeclareFontFamily{U}{mathx}{\hyphenchar\font45}
\DeclareFontShape{U}{mathx}{m}{n}{
      <5> <6> <7> <8> <9> <10>
      <10.95> <12> <14.4> <17.28> <20.74> <24.88>
      mathx10
      }{}
\DeclareSymbolFont{mathx}{U}{mathx}{m}{n}
\DeclareFontSubstitution{U}{mathx}{m}{n}
\DeclareMathAccent{\widecheck}{0}{mathx}{"71}
\DeclareMathAccent{\wideparen}{0}{mathx}{"75}

\numberwithin{equation}{section}

\theoremstyle{plain}
\newtheorem{theorem}{Theorem}[section]
\newtheorem{lemma}[theorem]{Lemma}

\theoremstyle{definition}
\newtheorem{definition}[theorem]{Definition}

\newtheorem{remark}[theorem]{Remark}

\newcommand{\sqcoversubset}{\raisebox{1.3ex}{\,\rotatebox{180}{$\sqsupseteq$}}\,}
\newcommand{\existunique}{\exists \textit{\bf !}}

\begin{document}


\title[A new perspective on the tensor product of semi-lattices]{A new perspective on the tensor product of semi-lattices}

\author[Eric Buffenoir]{Eric Buffenoir}
\address{Universit\'e de la C\^ote d'Azur, CNRS, InPhyNi, FRANCE}
\email{eric.buffenoir@cnrs.fr}





\begin{abstract}
We adopt a new perspective on the tensor product of arbitrary semi-lattices. Our basic construction exploits a description of semi-lattices in terms of bi-extensional Chu spaces associated to a target space defined to be the "boolean domain".  The comparison between our tensor products and the canonical tensor product, introduced by G.A. Fraser, is made in the distributive case and in the general case.  Some properties of our tensor products are also given. \textbf{Latest update: March 2024.}
\end{abstract}

\maketitle


\section{Inf semi-lattices and States/Effects Chu spaces}\label{sectiongeneral}

The set ${ \mathfrak{B}}:=\{\textit{\bf Y},\textit{\bf N},\bot\}$ will be equipped with the following poset structure : 
\begin{eqnarray}
\forall u,v\in { \mathfrak{B} },&& (u\leq v)\; :\Leftrightarrow\; (u={ \bot}\;\;\textit{\rm or}\;\; u=v).
\end{eqnarray} 
$({ \mathfrak{B}},\leq)$ is also an Inf semi-lattice which infima will be denoted $\bigwedge$. We have
\begin{eqnarray}
\forall x,y\in { \mathfrak{B} },&&  x \wedge y = \left\{\begin{array}{ll} x & \textit{\rm if}\;\;\; x=y\\
\bot & \textit{\rm if}\;\;\; x\not= y\end{array}\right.
\end{eqnarray}
We will also introduce a commutative monoid law denoted $\bullet$ and defined by
\begin{eqnarray}
\forall x\in { \mathfrak{B}},&& x \bullet \textit{\bf Y}=x, \;\;\;\;\; x \bullet \textit{\bf N}=\textit{\bf N}, \;\;\;\;\; \bot \bullet \bot = \bot. \label{expressionbullet}
\end{eqnarray}
This product law verifies the following properties
\begin{eqnarray}
\forall x\in { \mathfrak{B}},\forall B\subseteq { \mathfrak{B}} && x \bullet \bigwedge B=\bigwedge{}_{b\in B}(x\bullet b),\label{distributivitybullet}\\
\forall x\in { \mathfrak{B}},\forall C\subseteq_{Chain} { \mathfrak{B}} && x \bullet \bigvee C =\bigvee{}_{c\in C}(x\bullet c).\label{distributivitybullet2}
\end{eqnarray}
$({ \mathfrak{B} },\leq)$ will be also equipped with the following involution map :
\begin{eqnarray}
\overline{{ \bot}}:={ \bot} \;\;\;\;\;\;\;\; \overline{ \rm \bf Y}:={ \rm \bf N}\;\;\;\;\;\;\;\; \overline{ \rm \bf N}:={ \rm \bf Y}.
\end{eqnarray} 
$({ \mathfrak{B} },\leq)$ will be called {\em the boolean domain}.

\subsection{States/Effects Chu spaces}

We will say that a triple $({ \mathfrak{S}},{ \mathfrak{E}},{ \epsilon}^{ \mathfrak{S}})$ is {\em a States/Effects Chu space} iff 
\begin{itemize}
\item the set ${ \mathfrak{S}}$, called {\em space of states},  is a down-complete Inf semi-lattice (i.e.  $\forall S\subseteq {\mathfrak{S}}$ the infimum $(\bigsqcap{}^{{}^{ \mathfrak{S}}} S)$ exists in ${\mathfrak{S}}$), which admits a bottom element denoted $\bot_{{}_{ \mathfrak{S}}}$;
\item the set ${ \mathfrak{E}}$, called {\em space of effects},  is a down-complete Inf semi-lattice (i.e.  $\forall E\subseteq {\mathfrak{E}}$ the infimum $(\bigsqcap{}^{{}^{ \mathfrak{E}}} E)$ exists in ${\mathfrak{E}}$);
\item ${ \epsilon}^{ \mathfrak{S}}$ is a map from ${\mathfrak{E}}$ to ${\mathfrak{B}}^{\mathfrak{S}}$, called {\em evaluation map}, and satisfying
\begin{eqnarray}
\forall { \mathfrak{l}}\in { \mathfrak{E}},\forall S\subseteq { \mathfrak{S}}, && { \epsilon}^{ \mathfrak{S}}_{ { \mathfrak{l}}}(\bigsqcap{}^{{}^{ \mathfrak{S}}} S)=\bigwedge{}_{{}_{\sigma\in S}}\; { \epsilon}^{ \mathfrak{S}}_{ { \mathfrak{l}}}(\sigma),
 \label{axiomsigmainfsemilattice}\\
 \forall \sigma\in { \mathfrak{S}},\forall E\subseteq { \mathfrak{E}},&&{ \epsilon}^{ \mathfrak{S}}_{\bigsqcap{}^{{}^{ \mathfrak{E}}} E}(\sigma)=\bigwedge{}_{{}_{{ \mathfrak{l}}\in E}}\; { \epsilon}^{ \mathfrak{S}}_{ { \mathfrak{l}}}(\sigma),\label{axiomEinfsemilattice}
\end{eqnarray}
and 
\begin{eqnarray}
&& \forall { \mathfrak{l}},{ \mathfrak{l}}'\in { \mathfrak{E}},\;\;\;\;\;\;\;\;\;\;\;\;(\, \forall \sigma\in { \mathfrak{S}},\; { \epsilon}^{ \mathfrak{S}}_{ { \mathfrak{l}}}(\sigma)={ \epsilon}^{ \mathfrak{S}}_{ { \mathfrak{l}}'}(\sigma) \,) \Leftrightarrow  (\, { \mathfrak{l}}= { \mathfrak{l}}' \,),\label{Chuextensional}\\
&& \forall \sigma,\sigma'\in { \mathfrak{S}},\;\;\;\;\;\;\;\;\;\;\;\;(\, \forall { \mathfrak{l}}\in { \mathfrak{E}},\; { \epsilon}^{ \mathfrak{S}}_{ { \mathfrak{l}}}(\sigma)={ \epsilon}^{ \mathfrak{S}}_{ { \mathfrak{l}}}(\sigma') \,) \Leftrightarrow  (\, \sigma= \sigma' \,).\label{Chuseparated}
\end{eqnarray}
\end{itemize}
The partial order on ${ \mathfrak{S}}$ (resp. on ${ \mathfrak{E}}$) will be denoted by $\sqsubseteq_{{}_{{ \mathfrak{S}}}}$ (resp.  $\sqsubseteq_{{}_{{ \mathfrak{E}}}}$).\\

We will say that {\em the space of states admits a description in terms of pure states} iff we have moreover
\begin{itemize}
\item the set of complely meet-irreducible elements of ${ \mathfrak{S}}$, denoted ${ \mathfrak{S}}^{{}^{pure}}$ and called {\em set of pure states}, is equal to the set of maximal elements $Max({ \mathfrak{S}})$ and it 
 is a generating set for ${ \mathfrak{S}}$, i.e.  
\begin{eqnarray}
&&\forall \sigma \in { \mathfrak{S}}, \;\; \sigma= \bigsqcap{}^{{}^{ { \mathfrak{S}}}}  \underline{\sigma}_{{}_{ { \mathfrak{S}}}} ,\;\;\textit{\rm where}\;\;
\underline{\sigma}_{{}_{ { \mathfrak{S}}}}:=
({ \mathfrak{S}}^{{}^{pure}} \cap (\uparrow^{{}^{ { \mathfrak{S}}}}\!\!\!\! \sigma) ) \;\;\textit{\rm and}\;\; { \mathfrak{S}}^{{}^{pure}}=Max({ \mathfrak{S}}).\;\;\;\;\;\;\;\;\;\;\;\;
\label{completemeetirreducibleordergenerating}
\end{eqnarray}
\end{itemize}
We will introduce the following notations :
\begin{eqnarray}
\forall { \mathfrak{l}}\in { \mathfrak{E}},\exists \overline{\; \mathfrak{l}\;}\in { \mathfrak{E}} &\vert & \forall \sigma\in { \mathfrak{S}}, { \epsilon}^{ \mathfrak{S}}_{ \overline{\; \mathfrak{l}\;}}(\sigma)= \overline{{ \epsilon}^{ \mathfrak{S}}_{ { \mathfrak{l}}}(\sigma)},\label{etbar}\\
\exists { \mathfrak{Y}}_{ \mathfrak{E}}\in { \mathfrak{E}} &\vert & \forall \sigma\in { \mathfrak{S}}, { \epsilon}^{ \mathfrak{S}}_{{ \mathfrak{Y}}_{ \mathfrak{E}}}(\sigma)= \textit{\bf Y},\label{ety}\\
\exists \bot_{ \mathfrak{E}}\in { \mathfrak{E}} &\vert & \forall \sigma\in { \mathfrak{S}}, { \epsilon}^{ \mathfrak{S}}_{\bot_{ \mathfrak{E}}}(\sigma)= \bot.
\end{eqnarray}
Here and in the following, $\widehat{\Sigma\Sigma'}{}^{{}^{ \mathfrak{S}}}$ means that $\Sigma$ and $\Sigma'$ have a common upper-bound in ${ \mathfrak{S}}$, and $\neg \widehat{\Sigma\Sigma'}{}^{{}^{ \mathfrak{S}}}$ means they have none.  We will also adopt the following notations :   $\uparrow^{{}^{ \mathfrak{S}}}\!\!\Sigma$ for the upper subset $\{ \sigma\in { \mathfrak{S}} \;\vert\; \sigma\sqsupseteq_{{}_{{ \mathfrak{S}}}}\Sigma\}$, and $\sigma \parallel_{{}_{ \mathfrak{S}}}\sigma'$ for ($\sigma \not\sqsubseteq_{{}_{ \mathfrak{S}}}\sigma'$ and $\sigma' \not\sqsubseteq_{{}_{ \mathfrak{S}}}\sigma$).\\

Now, we intend to describe a natural States/Effects Chu space defined once is given the space of states ${ \mathfrak{S}}$.\\
The {\em natural space of effects}, denoted ${ \mathfrak{E}}_{ \mathfrak{S}}$ is defined to be 
\begin{eqnarray}
{ \mathfrak{E}}_{ \mathfrak{S}}:=\{ { \mathfrak{l}}_{{}_{(\sigma,\sigma')}}\;\vert\; \sigma,\sigma'\in { \mathfrak{S}},\;\neg \widehat{\sigma\sigma'}{}^{{}^{ \mathfrak{S}}}\;\}\cup \{ { \mathfrak{l}}_{{}_{(\sigma,\cdot)}}\;\vert\; \sigma\in { \mathfrak{S}}\;\}\cup \{ { \mathfrak{l}}_{{}_{(\cdot,\sigma)}}\;\vert\; \sigma\in { \mathfrak{S}}\;\}\cup \{\, { \mathfrak{l}}_{{}_{(\cdot,\cdot)}}\,\} &&\;\;\;\;\;\;\;\;
\end{eqnarray}
as a set, with the following Inf semi-lattice law
\begin{eqnarray}
{ \mathfrak{l}}_{{}_{(\sigma_1,\sigma'_1)}} \sqcap_{{}_{{ \mathfrak{E}}_{ \mathfrak{S}}}} { \mathfrak{l}}_{{}_{(\sigma_2,\sigma'_2)}} &=& \left\{ \begin{array}{lcl}
{ \mathfrak{l}}_{{}_{(\sigma_1\sqcup_{{}_{{ \mathfrak{S}}}} \sigma_2,\sigma'_1 \sqcup_{{}_{{ \mathfrak{S}}}} \sigma'_2)}} & \textit{\rm if} & \widehat{\sigma_1\sigma_2}{}^{{}^{ \mathfrak{S}}}\;\textit{\rm and}\; \widehat{\sigma'_1\sigma'_2}{}^{{}^{ \mathfrak{S}}}\\
{ \mathfrak{l}}_{{}_{(\;\cdot\;,\sigma'_1 \sqcup_{{}_{{ \mathfrak{S}}}} \sigma'_2)}} & \textit{\rm if} & \neg \widehat{\sigma_1\sigma_2}{}^{{}^{ \mathfrak{S}}}\;\textit{\rm and}\; \widehat{\sigma'_1\sigma'_2}{}^{{}^{ \mathfrak{S}}}\\
{ \mathfrak{l}}_{{}_{(\sigma_1\sqcup_{{}_{{ \mathfrak{S}}}} \sigma_2,\;\cdot\;)}} & \textit{\rm if} & \widehat{\sigma_1\sigma_2}{}^{{}^{ \mathfrak{S}}}\;\textit{\rm and}\; \neg\widehat{\sigma'_1\sigma'_2}{}^{{}^{ \mathfrak{S}}}\\
{ \mathfrak{l}}_{{}_{(\cdot,\cdot)}} & \textit{\rm if} & \neg\widehat{\sigma_1\sigma_2}{}^{{}^{ \mathfrak{S}}}\;\textit{\rm and}\; \neg\widehat{\sigma'_1\sigma'_2}{}^{{}^{ \mathfrak{S}}}
\end{array}\right.\label{defcapES}
\end{eqnarray}
This expression is naturally extended to the whole set of effects (i.e. including the elements of the form ${ \mathfrak{l}}_{{}_{(\sigma,\cdot)}}$, ${ \mathfrak{l}}_{{}_{(\cdot,\sigma)}}$ and ${ \mathfrak{l}}_{{}_{(\cdot,\cdot)}}$) as long as we adopt the convention defining 
\begin{eqnarray}
\forall \sigma\in { \mathfrak{S}},\;\;\;\;\neg\widehat{\;\cdot\;\sigma}{}^{{}^{ \mathfrak{S}}}:= \textit{\rm TRUE}\;\;\;\;\;\;\; &\textit{\rm and}& \;\;\;\;\;\;\; \neg\widehat{\;\cdot\;\cdot\;}{}^{{}^{ \mathfrak{S}}}:= \textit{\rm TRUE}.
\end{eqnarray}

Here and in the following, we adopt the following notations 
\begin{eqnarray}
(\uparrow^{{}^{ { \mathfrak{E}_{ \mathfrak{S}}}}}\!\!\! { \mathfrak{l}}):=\{\, { \mathfrak{l}}'\in { \mathfrak{E}}_{ \mathfrak{S}}\;\vert\; { \mathfrak{l}}\sqsubseteq_{{}_{ \mathfrak{E}_{ \mathfrak{S}}}}{ \mathfrak{l}}' \,\}& \textit{\rm and} &(\downarrow_{{}_{ { \mathfrak{E}_{ \mathfrak{S}}}}}\!\!\! { \mathfrak{l}}):=\{\, { \mathfrak{l}}'\in { \mathfrak{E}}_{ \mathfrak{S}}\;\vert\; { \mathfrak{l}}\sqsupseteq_{{}_{ \mathfrak{E}_{ \mathfrak{S}}}}{ \mathfrak{l}}' \,\}.\label{downarrowE}
\end{eqnarray}

The evaluation map $\epsilon^{ \mathfrak{S}}$ is defined by
\begin{eqnarray}
\epsilon^{ \mathfrak{S}}_{{ \mathfrak{l}}_{{}_{(\sigma,\sigma')}}}(\sigma'') &:= & \left\{ \begin{array}{lcl} \textit{\bf Y} & \textit{\rm if} & \sigma\sqsubseteq_{{}_{ \mathfrak{S}}} \sigma''\\\textit{\bf N} & \textit{\rm if} & \sigma'\sqsubseteq_{{}_{ \mathfrak{S}}} \sigma''\\ \bot & \textit{\rm otherwise } &\end{array}\right.
\end{eqnarray}
this expression is naturally extended to the whole set of effects by adopting the following convention
\begin{eqnarray}
(\;\cdot\;\sqsubseteq_{{}_{{ \mathfrak{S}}}} \sigma) & :=&\textit{\rm FALSE}.
\end{eqnarray}

We note that $\bot_{{}_{{ \mathfrak{E}}_{ \mathfrak{S}}}} = { \mathfrak{l}}_{{}_{(\cdot,\cdot)}}$ and ${ \mathfrak{Y}}_{{}_{{ \mathfrak{E}}_{ \mathfrak{S}}}} = { \mathfrak{l}}_{{}_{(\bot_{{}_{{ \mathfrak{S}}}},\cdot)}}$.\\

To conclude, we have obtained a well-defined States/Effects Chu space $({ \mathfrak{S}},{ \mathfrak{E}}_{ \mathfrak{S}},\epsilon^{ \mathfrak{S}})$.

\begin{definition}\label{DefinreductionE}
In the following, we will always consider that ${ \mathfrak{E}}$ is a subset of the natural space of effects ${ \mathfrak{E}}_{ \mathfrak{S}}$, satisfying 
\begin{eqnarray}
\forall \sigma\in { \mathfrak{S}}\smallsetminus \{\bot_{{}_{ \mathfrak{S}}}\},\exists \sigma'\in { \mathfrak{S}}\smallsetminus \{\bot_{{}_{ \mathfrak{S}}}\} & \vert & { \mathfrak{l}}_{(\sigma,\sigma')}\in { \mathfrak{E}},\label{axiomreduc1}\\
\forall L\subseteq {{ \mathfrak{E}}},&&( \bigsqcap{}^{{}^{{ \mathfrak{E}}_{ \mathfrak{S}}}}L) \in {{ \mathfrak{E}}},\label{axiomreduc2}\\
\forall { \mathfrak{l}}\in { \mathfrak{E}},&&\overline{\,{ \mathfrak{l}}\,} \in { \mathfrak{E}},\label{axiomreduc3}\\
{ \mathfrak{Y}}_{{}_{{{ \mathfrak{E}}}_{ \mathfrak{S}}}} \in { \mathfrak{E}} &\textit{\rm and} & \bot_{{}_{{{ \mathfrak{E}}}_{ \mathfrak{S}}}} \in { \mathfrak{E}}.\label{axiomreduc4}
\end{eqnarray}
\end{definition}

\begin{theorem}\label{TheoremreductionE}
If the space of effects ${ \mathfrak{E}}$ satisfies the conditions of Definition \ref{DefinreductionE}, then $({ \mathfrak{S}}, { \mathfrak{E}}, \epsilon^{ \mathfrak{S}})$ is a well-defined States/Effects Chu space.
\end{theorem}
\begin{proof}
Trivial. It suffices to check the validity of equation (\ref{Chuseparated}), which is straightforward. The other properties are tautological.
\end{proof}


Let us note the two following theorems that will reveal to be useful in the rest of this paper.

\begin{theorem}  \label{aepsilonsigma}
Let us consider a map $(A : { \mathfrak{S}}\longrightarrow { \mathfrak{B}},  \sigma \mapsto { \mathfrak{a}}_{\sigma})$ satisfying
\begin{eqnarray}
\forall \sigma,\sigma'\in { \mathfrak{S}},&& (\sigma \sqsubseteq_{{}_{{ \mathfrak{S}}}} \sigma')\Rightarrow ({ \mathfrak{a}}_{\sigma} \leq { \mathfrak{a}}_{\sigma'}),\label{theorema1}\\
\forall \{\sigma_i\;\vert\; i\in I\}\subseteq { \mathfrak{S}},&&  { \mathfrak{a}}_{{\bigsqcap{}^{{}^{ \mathfrak{S}}}_{{}_{i\in i}}\sigma_i}} = \bigwedge{}_{{i\in I}} \;{ \mathfrak{a}}_{{\sigma_i}},\label{theorema2}
\end{eqnarray}
Then, we have
\begin{eqnarray}
\existunique \;{ \mathfrak{l}}\in { \mathfrak{E}}_{ \mathfrak{S}} & \vert & \forall \sigma\in { \mathfrak{S}}, \; { \epsilon}^{ \mathfrak{S}}_{{ \mathfrak{l}}}(\sigma)={ \mathfrak{a}}_{\sigma}.
\end{eqnarray}
\end{theorem}
\begin{proof}
Straightforward. If $\{\, \sigma\;\vert\; { \mathfrak{a}}_{\sigma}=\textit{\bf Y}\,\}$ and $\{\, \sigma\;\vert\; { \mathfrak{a}}_{\sigma}=\textit{\bf N}\,\}$ are not empty, it suffices to define $\Sigma_A:=\bigsqcap^{{}^{{ \mathfrak{S}}}}\{\, \sigma\;\vert\; { \mathfrak{a}}_{\sigma}=\textit{\bf Y}\,\}$,  $\Sigma_A':=\bigsqcap^{{}^{{ \mathfrak{S}}}}\{\, \sigma\;\vert\; { \mathfrak{a}}_{\sigma}=\textit{\bf N}\,\}$ and ${ \mathfrak{l}}:={ \mathfrak{l}}_{(\Sigma_A,\Sigma_A')} $ (the case where some or all of these subsets are empty is treated immediately).
\end{proof}

\begin{theorem}  \label{blepsilonsigma}
Let us consider a map $(B : { \mathfrak{E}} \longrightarrow { \mathfrak{B}},{ \mathfrak{l}}\mapsto { \mathfrak{b}}_{{ \mathfrak{l}}})$ satisfying
\begin{eqnarray}
\forall { \mathfrak{l}}, { \mathfrak{l}}'\in { \mathfrak{E}},&& ({ \mathfrak{l}} \sqsubseteq_{{}_{{ \mathfrak{E}}}} { \mathfrak{l}}')\Rightarrow ({ \mathfrak{b}}_{{ \mathfrak{l}}} \leq { \mathfrak{b}}_{{ \mathfrak{l}}'}),\label{theorembl1}\\
\forall \{{ \mathfrak{l}}_i\;\vert\; i\in I\}\subseteq { \mathfrak{E}},&&  { \mathfrak{b}}_{{\bigsqcap{}^{{}^{ \mathfrak{E}}}_{{}_{i\in i}}{ \mathfrak{l}}_i}} = \bigwedge{}_{{i\in I}} \;{ \mathfrak{b}}_{{{ \mathfrak{l}}_i}},\label{theorembl2}\\
\forall { \mathfrak{l}}\in { \mathfrak{E}},&& { \mathfrak{b}}_{\overline{\; \mathfrak{l}\;}}=\overline{{ \mathfrak{b}}_{{ \mathfrak{l}}}},\label{theorembl3}\\
& & { \mathfrak{b}}_{{ \mathfrak{Y}}_{ \mathfrak{E}}}=\textit{\bf Y}.\label{theorembl4}
\end{eqnarray}
Then, we have
\begin{eqnarray}
\existunique \; \sigma\in { \mathfrak{S}} & \vert & \forall { \mathfrak{l}}\in { \mathfrak{E}}, \; { \epsilon}^{ \mathfrak{S}}_{{ \mathfrak{l}}}(\sigma)={ \mathfrak{b}}_{{ \mathfrak{l}}}.
\end{eqnarray}
\end{theorem}
\begin{proof}
Let us consider ${ \mathfrak{l}}_B:=\bigsqcap^{{}^{{ \mathfrak{E}}}}\{\, { \mathfrak{l}}\in { \mathfrak{E}} \;\vert\; { \mathfrak{b}}_{{ \mathfrak{l}}}=\textit{\bf Y}\,\}$.  Note that ${ \mathfrak{l}}_B$ exists in ${ \mathfrak{E}}$ because ${ \mathfrak{E}}$ is a down-complete Inf semi-lattice and the subset $\{\, { \mathfrak{l}}\in { \mathfrak{E}} \;\vert\; { \mathfrak{b}}_{{ \mathfrak{l}}}=\textit{\bf Y}\,\}$ contains at least the element ${ \mathfrak{Y}}_{ \mathfrak{E}}$.  Moreover, ${ \mathfrak{b}}_{{ \mathfrak{l}}_B}=\textit{\bf Y}$ because of the relation (\ref{theorembl2}). Note also that ${ \mathfrak{l}} \sqsupseteq_{{}_{{ \mathfrak{E}}}} { \mathfrak{l}}_B$ implies ${ \mathfrak{b}}_{{ \mathfrak{l}}}=\textit{\bf Y}$ because of the relation (\ref{theorembl1}), and conversely ${ \mathfrak{b}}_{{ \mathfrak{l}}}=\textit{\bf Y}$ implies ${ \mathfrak{l}} \sqsupseteq_{{}_{{ \mathfrak{E}}}} { \mathfrak{l}}_B$ due to the definition of ${ \mathfrak{l}}_B$. Let us now introduce $\Sigma_{{}_{{ \mathfrak{l}}_B}}=\bigsqcap{}^{{}^{ \mathfrak{S}}}({\epsilon}^{ \mathfrak{S}}_{{ \mathfrak{l}}_B})^{\;-1}(\textit{\bf Y})$. For any ${ \mathfrak{l}}$ such that ${ \mathfrak{l}} \sqsupseteq_{{}_{{ \mathfrak{E}}}} { \mathfrak{l}}_B$, we have $\epsilon^{ \mathfrak{S}}_{{ \mathfrak{l}}}(\Sigma_{{}_{{ \mathfrak{l}}_B}})\geq \epsilon^{ \mathfrak{S}}_{{ \mathfrak{l}}_B}(\Sigma_{{}_{{ \mathfrak{l}}_B}}) = \textit{\bf Y}$, i.e. $\epsilon^{ \mathfrak{S}}_{{ \mathfrak{l}}}(\Sigma_{{}_{{ \mathfrak{l}}_B}}) = \textit{\bf Y}$.  We could suppose that ${ \mathfrak{l}}_B={ \mathfrak{l}}_{(\Sigma_{{ \mathfrak{l}}_B},\Sigma'_{{ \mathfrak{l}}_B})}$ for a certain $\Sigma'_{{ \mathfrak{l}}_B}\in {\mathfrak{S}}$. However, we note that, because of (\ref{theorembl2}) and (\ref{theorembl4}), we have ${ \mathfrak{l}}_{(\Sigma_{{ \mathfrak{l}}_B},\cdot)}\sqsubset_{{}_{{ \mathfrak{E}}}} { \mathfrak{l}}_{(\Sigma_{{ \mathfrak{l}}_B},\Sigma'_{{ \mathfrak{l}}_B})}$ and ${ \mathfrak{b}}_{{ \mathfrak{l}}_{(\Sigma_{{ \mathfrak{l}}_B},\cdot)}}={ \mathfrak{b}}_{{ \mathfrak{l}}_{(\Sigma_{{ \mathfrak{l}}_B},\Sigma'_{{ \mathfrak{l}}_B})}\sqcap_{{}_{{ \mathfrak{E}}}} { \mathfrak{Y}}_{ \mathfrak{E}} }= { \mathfrak{b}}_{{ \mathfrak{l}}_{(\Sigma_{{ \mathfrak{l}}_B},\Sigma'_{{ \mathfrak{l}}_B})}}\wedge { \mathfrak{b}}_{{ \mathfrak{Y}}_{ \mathfrak{E}} }=\textit{\bf Y}$ which would contradict the definition of ${ \mathfrak{l}}_B$.  Hence, we have to accept that ${ \mathfrak{l}}_B={ \mathfrak{l}}_{(\Sigma_{{ \mathfrak{l}}_B},\cdot)}$.  Thus, we note that, for any ${ \mathfrak{l}}_{(\Sigma,\Sigma')}$, the property ${ \mathfrak{l}}_{(\Sigma,\Sigma')} \not\sqsupseteq_{{}_{{ \mathfrak{E}}}} { \mathfrak{l}}_B$ is then equivalent to the property $\Sigma\not\sqsupseteq_{{}_{{ \mathfrak{S}}}} \Sigma_{{ \mathfrak{l}}_B}$. Then, if ${ \mathfrak{l}}_{(\Sigma,\Sigma')} \not\sqsupseteq_{{}_{{ \mathfrak{E}}}} { \mathfrak{l}}_B$ we cannot have $\epsilon^{ \mathfrak{S}}_{{ \mathfrak{l}}_{(\Sigma,\Sigma')}}(\Sigma_{{}_{{ \mathfrak{l}}_B}})=\textit{\bf Y}$.  We then conclude that the property $\epsilon^{ \mathfrak{S}}_{{ \mathfrak{l}}}(\Sigma_{{}_{{ \mathfrak{l}}_B}}) = \textit{\bf Y}$ is equivalent to the property ${ \mathfrak{l}} \sqsupseteq_{{}_{{ \mathfrak{E}}}} { \mathfrak{l}}_B$, or in other words $\epsilon^{ \mathfrak{S}}_{{ \mathfrak{l}}}(\Sigma_{{}_{{ \mathfrak{l}}_B}}) = \textit{\bf Y}$ is equivalent to ${ \mathfrak{b}}_{{ \mathfrak{l}}}=\textit{\bf Y}$.  Using (\ref{theorembl3}) and (\ref{etbar}),  we deduce that $(\epsilon^{ \mathfrak{S}}_{{ \mathfrak{l}}}(\Sigma_{{}_{{ \mathfrak{l}}_B}}) = \textit{\bf N})\Leftrightarrow (\epsilon^{ \mathfrak{S}}_{\overline{ \mathfrak{l}}}(\Sigma_{{}_{{ \mathfrak{l}}_B}}) = \textit{\bf Y})\Leftrightarrow ({ \mathfrak{b}}_{\overline{ \mathfrak{l}}}=\textit{\bf Y}) \Leftrightarrow ({ \mathfrak{b}}_{{ \mathfrak{l}}}=\textit{\bf N})$. As a final conclusion, we have for any ${ \mathfrak{l}}\in { \mathfrak{E}}$ the equality $\epsilon^{ \mathfrak{S}}_{{ \mathfrak{l}}}(\Sigma_B)={ \mathfrak{b}}_{{ \mathfrak{l}}}$. \\
The fact that every elements of the form ${ \mathfrak{l}}_{(\sigma,\cdot)}$ must be in ${ \mathfrak{E}}$ (consequence of (\ref{axiomreduc1}) (\ref{axiomreduc2}) and (\ref{axiomreduc4})) allows to deduce the uniqueness character of the state $\sigma$.
\end{proof}

\begin{theorem}
\begin{eqnarray}
\forall \{\sigma_i\;\vert\; i\in I\} \subseteq_{Chain} { \mathfrak{S}},\;\;\exists \sigma \in { \mathfrak{S}} & \vert & \forall { \mathfrak{l}}\in { \mathfrak{E}}, \; \epsilon^{ \mathfrak{S}}_{{ \mathfrak{l}}}(\sigma)=\bigvee{}_{\!\!\! i\in I}\; \epsilon^{ \mathfrak{S}}_{{ \mathfrak{l}}}(\sigma_i),\label{demochaincontinuous1}\\
\sigma & = & \bigsqcup{}^{{}^{{ \mathfrak{S}}}}_{{i\in I}}\;\sigma_i.
\end{eqnarray} 
As a consequence, using Zorn's Lemma, we deduce that 
\begin{eqnarray}
\forall \sigma\in { \mathfrak{S}}, && \exists \sigma'\in Max({ \mathfrak{S}})\;\vert\; \sigma \sqsubseteq_{{}_{{ \mathfrak{S}}}}\sigma'.\label{existsmaxabove}
\end{eqnarray}
\end{theorem}
\begin{proof}
First of all, we note that $\{\sigma_i\;\vert\; i\in I\} \subseteq_{Chain} { \mathfrak{S}}$ and the monotonicity property of $\epsilon^{ \mathfrak{S}}$ implies that $\{\epsilon^{ \mathfrak{S}}_{{ \mathfrak{l}}}(\sigma_i)\;\vert\; i\in I\} \subseteq_{Chain} { \mathfrak{B}}$ for any ${ \mathfrak{l}}\in { \mathfrak{E}}$ and then $\bigvee{}_{\!\!\! i\in I}\; \epsilon^{ \mathfrak{S}}_{{ \mathfrak{l}}}(\sigma_i)$ exists for any ${ \mathfrak{l}}\in { \mathfrak{E}}$ due to the chain-completeness of ${ \mathfrak{B}}$.\\
Using the properties (\ref{axiomEinfsemilattice})(\ref{etbar})(\ref{ety}) of the map $\epsilon$ and the complete-distributivity properties satisfied by ${ \mathfrak{B}}$, we can check easily that the map ${ \mathfrak{l}}\mapsto \bigvee{}_{\!\!\! i\in I}\; \epsilon^{ \mathfrak{S}}_{{ \mathfrak{l}}}(\sigma_i)$ satisfies properties (\ref{theorembl1}) (\ref{theorembl2}) (\ref{theorembl3}) (\ref{theorembl4}). As a consequence, the property (\ref{demochaincontinuous1}) is a direct consequence of Theorem \ref{blepsilonsigma}. \\
By definition of the poset structure on ${ \mathfrak{S}}$, we deduce, from the property $(\, \forall { \mathfrak{l}}\in { \mathfrak{E}}, \;\epsilon^{ \mathfrak{S}}_{{ \mathfrak{l}}}(\sigma)=\bigvee{}_{\!\!\! i\in I}\; \epsilon^{ \mathfrak{S}}_{{ \mathfrak{l}}}(\sigma_i) \,)$, that $\sigma \sqsupseteq_{{}_{{ \mathfrak{S}}}} \sigma_i,\; \forall i\in I$ and $(\sigma'\sqsupseteq_{{}_{{ \mathfrak{S}}}} \sigma_i,\; \forall i\in I)\Rightarrow (\sigma \sqsupseteq_{{}_{{ \mathfrak{S}}}} \sigma')$. In other words, $\sigma  =  \bigsqcup{}^{{}^{{ \mathfrak{S}}}}_{{}_{i\in I}}\sigma_i$.
\end{proof}

\begin{remark}\label{remarkspaceeffectspurestates}
We now observe that, if ${ \mathfrak{S}}$ has a description in terms of pure states, then the natural space of effects ${ \mathfrak{E}}_{ \mathfrak{S}}$ is such that its subset of meet-irreducible elements, denoted $Irr({ \mathfrak{E}}_{ \mathfrak{S}})$, is a generating subset for ${ \mathfrak{E}}_{ \mathfrak{S}}$.  \\
To check this point, we first note that the space of effects has a bottom element denoted $\bot_{{}_{ \mathfrak{E}_{ \mathfrak{S}}}}$ and defined by
\begin{eqnarray}
\forall \sigma\in { \mathfrak{S}}, && \epsilon_{\bot_{{}_{ \mathfrak{E}_{ \mathfrak{S}}}}}(\sigma) = \bot. 
\end{eqnarray}
Secondly, ${ \mathfrak{E}}_{ \mathfrak{S}}$ appears to be an algebraic domain.\\
To prove this point, we firstly observe that ${ \mathfrak{E}}_{ \mathfrak{S}}$ satisfies the following chain-completeness property
\begin{eqnarray}
\forall \{{ \mathfrak{l}}_i\;\vert\; i\in I\} \subseteq_{Chain} { \mathfrak{E}}_{ \mathfrak{S}},\;\;\exists { \mathfrak{l}} \in { \mathfrak{E}}_{ \mathfrak{S}} & \vert & \forall \sigma\in { \mathfrak{S}}, \; \epsilon_{{ \mathfrak{l}}}(\sigma)=\bigvee{}_{\!\!\! i\in I}\; \epsilon_{{ \mathfrak{l}}_i}(\sigma),\label{demochaincontinuous2}\\
{ \mathfrak{l}} & = & \bigsqcup{}^{{}^{{ \mathfrak{E}_{ \mathfrak{S}}}}}_{{i\in I}}\;{ \mathfrak{l}}_i.
\end{eqnarray} 
This is an immediate consequence of the down-completeness of ${ \mathfrak{S}}$. \\
We secondly observe that ${ \mathfrak{E}}_{ \mathfrak{S}}$ is atomistic, i.e. 
\begin{eqnarray}
\exists { \mathcal{A}}_{{ \mathfrak{E}_{ \mathfrak{S}}}}\subseteq { \mathfrak{E}}_{ \mathfrak{S}} & \vert &
\left\{
\begin{array}{l}\forall { \mathfrak{l}}\in { \mathcal{A}}_{{ \mathfrak{E}_{ \mathfrak{S}}}},\;\; \bot_{{}_{ \mathfrak{E}_{ \mathfrak{S}}}} \sqcoversubset_{{}_{{ \mathfrak{E}_{ \mathfrak{S}}}}} { \mathfrak{l}} \\
\forall { \mathfrak{l}} \in { \mathfrak{E}_{ \mathfrak{S}}}\smallsetminus \{\bot_{{}_{ \mathfrak{E}_{ \mathfrak{S}}}}\}, \;\; \exists { \mathfrak{l}}'\in { \mathcal{A}}_{{ \mathfrak{E}_{ \mathfrak{S}}}}\;\vert\; { \mathfrak{l}}' \sqsubseteq_{{}_{{ \mathfrak{E}_{ \mathfrak{S}}}}} { \mathfrak{l}}
\end{array}\right.\label{defatomeffects}\\
{ \mathcal{A}}_{{ \mathfrak{E}_{ \mathfrak{S}}}} &=& \{{ \mathfrak{l}}_{{}_{(\Sigma,\cdot)}}\vert \Sigma\in { \mathfrak{S}}^{{}^{pure}}\}\cup \{{ \mathfrak{l}}_{{}_{(\cdot,\Sigma)}}\vert \Sigma\in { \mathfrak{S}}^{{}^{pure}}\}\label{expressionatomeffects}
\end{eqnarray}
and
\begin{eqnarray}
\forall { \mathfrak{l}}\in { \mathfrak{E}}_{ \mathfrak{S}}, && { \mathfrak{l}}= \bigsqcup{}^{{}^{ \mathfrak{E}_{ \mathfrak{S}}}} \{\, { \mathfrak{l}}'\in { \mathcal{A}}_{{ \mathfrak{E_{ \mathfrak{S}}}}} \;\vert\; { \mathfrak{l}}' \sqsubseteq_{{}_{ \mathfrak{E}_{ \mathfrak{S}}}}{ \mathfrak{l}}\}.\label{atomisticeffects}
\end{eqnarray}
Here, we have adopted the notation $\forall { \mathfrak{l}}\in {{ \mathfrak{E}_{ \mathfrak{S}}}},  \bot_{{}_{ \mathfrak{E}}} \sqcoversubset_{{}_{{ \mathfrak{E}_{ \mathfrak{S}}}}} { \mathfrak{l}} \;\Leftrightarrow\; (\, \bot_{{}_{ \mathfrak{E}_{ \mathfrak{S}}}} \sqsubset_{{}_{{ \mathfrak{E}_{ \mathfrak{S}}}}} { \mathfrak{l}} \;\;\textit{\rm and}\;\; (\forall { \mathfrak{l}}'\in {{ \mathfrak{E}_{ \mathfrak{S}}}}, \bot_{{}_{ \mathfrak{E}_{ \mathfrak{S}}}} \sqsubseteq_{{}_{{ \mathfrak{E_{ \mathfrak{S}}}}}} { \mathfrak{l}}'\sqsubseteq_{{}_{{ \mathfrak{E}_{ \mathfrak{S}}}}} { \mathfrak{l}}\;\Rightarrow\; (\bot_{{}_{ \mathfrak{E}_{ \mathfrak{S}}}} = { \mathfrak{l}}' \;\textit{\rm or}\; { \mathfrak{l}}'={ \mathfrak{l}}))\,)$.

The property (\ref{existsmaxabove}) implies directly the second condition of (\ref{defatomeffects}). The first condition of (\ref{defatomeffects}) is easy to check using the expression of the order on effects. The property (\ref{atomisticeffects}) is a direct consequence of property (\ref{completemeetirreducibleordergenerating}).\\
Endly, the compacity of atoms is trivial. \\ 
The algebraicity of ${ \mathfrak{E}}_{ \mathfrak{S}}$ follows.\\
As a conclusion of this analysis, we have the following property :
\begin{eqnarray}
&&\forall { \mathfrak{l}} \in { \mathfrak{E}}_{ \mathfrak{S}}, \;\; { \mathfrak{l}}= \bigsqcap{}^{{}^{ { \mathfrak{E}_{ \mathfrak{S}}}}}  \underline{\;{ \mathfrak{l}}\;}_{{}_{ { \mathfrak{E}_{ \mathfrak{S}}}}},\;\;\textit{\rm where}\;\;
\underline{\;{ \mathfrak{l}}\;}_{{}_{ { \mathfrak{E}_{ \mathfrak{S}}}}}=
(Irr({ \mathfrak{E}}_{ \mathfrak{S}}) \cap (\uparrow^{{}^{ { \mathfrak{E}_{ \mathfrak{S}}}}}\!\!\!\! { \mathfrak{l}}) ).\label{completemeetirreducibleordergeneratingeffects}
\end{eqnarray}
Indeed, from previous results  ${ \mathfrak{E}}_{ \mathfrak{S}}$ is a bounded-complete algebraic domain. The property (\ref{completemeetirreducibleordergeneratingeffects}) is then a direct consequence of \cite[Theorem I-4.26 p.126]{gierz_hofmann_keimel_lawson_mislove_scott_2003}. \\
However, we have to insist on the fact that $Irr({ \mathfrak{E}}_{ \mathfrak{S}})$ is not necessarily equal to $Max({ \mathfrak{E}}_{ \mathfrak{S}})$  (even when $Irr({ \mathfrak{S}})=Max({ \mathfrak{S}})$). We have nevertheless  always $Irr({ \mathfrak{E}}_{ \mathfrak{S}})\supseteq Max({ \mathfrak{E}}_{ \mathfrak{S}})$.  \\
We will conclude by characterizing the elements of $Max({ \mathfrak{E}_{ \mathfrak{S}}})$ :
\begin{eqnarray}
Max({ \mathfrak{E}}_{ \mathfrak{S}}) &=& \{\, { \mathfrak{l}}_{{}_{(\Sigma,\Sigma')}}\;\vert\; (\Sigma,\Sigma')\in { \mathfrak{S}}^{\times 2},\;\Sigma\, \widecheck{\bowtie}_{{}_{ \mathfrak{S}}}\Sigma'\,\} \cup \{{ \mathfrak{Y}}_{{}_{ \mathfrak{E_{ \mathfrak{S}}}}} \}\cup \{\overline{{ \mathfrak{Y}}_{{}_{ \mathfrak{E}_{ \mathfrak{S}}}}} \},\label{Epure}\;\;\;\;\;\;\;\;\;\;\;\;\;\;\;\;\;\;
\end{eqnarray}
where we have introduced the following binary relation, denoted by $\widecheck{\bowtie}_{{}_{ \mathfrak{S}}}$  and defined on ${ \mathfrak{S}}$ by
\begin{eqnarray}
&&\hspace{-0.5cm} \forall (\sigma,\sigma')\in { \mathfrak{S}}^{\times 2},\;\;\;\;\;\sigma\, \widecheck{\bowtie}_{{}_{ \mathfrak{S}}}\sigma'  :\Leftrightarrow \nonumber\\
&&\hspace{1cm}  (\,(\forall \sigma'' \sqsubset_{{}_{ \mathfrak{S}}}\sigma', \widehat{\sigma\sigma''}{}^{{}^{ \mathfrak{S}}})\;\;\textit{\rm and}\;\; (\forall \sigma'' \sqsubset_{{}_{ \mathfrak{S}}}\sigma, \widehat{\sigma'\sigma''}{}^{{}^{ \mathfrak{S}}})\;\textit{\rm and}\; (\neg\; \widehat{\sigma\sigma'}{}^{{}^{ \mathfrak{S}}})\,).\;\;\;\;\;\;\;\;\;\;\;\;\;\;\;\;\label{quasiantipodal}
\end{eqnarray}
\end{remark}

\subsection{Particular spaces of states}\label{particularspacesofstates}

\begin{definition}
A space of states will be said to be {\em a simplex space of states} (or {\em a classical space of states}) iff
\begin{eqnarray}
\forall \sigma \in { \mathfrak{S}},\existunique \;U_\sigma\subseteq { \mathfrak{S}}^{{}^{pure}} & \vert & \sigma=\bigsqcap{}^{{}^{ \mathfrak{S}}}U_\sigma.\label{simplexdecompunique}
\end{eqnarray}
We note that $U_\sigma=\underline{\sigma}_{{}_{{ \mathfrak{S}}}}$.
\end{definition}

\begin{lemma}\label{particularpsi}
If ${ \mathfrak{S}}$ is a simplex space of states, then for any $\sigma\in { \mathfrak{S}}$ there exists a unique homomorphism $\psi_\sigma$ from ${ \mathfrak{S}}$ to ${ \mathfrak{B}}$ satisfying $\psi_\sigma ({ \mathfrak{S}}^{{}^{pure}})\subseteq \{\textit{\bf Y},\textit{\bf N}\}$ and $\psi_\sigma (\sigma)=\textit{\bf Y}$.  The reciprocal implication is also true.
\end{lemma}
\begin{proof}
Let us assume that ${ \mathfrak{S}}$ is a simplex space of states. For any $\sigma\in { \mathfrak{S}}$, we choose $\sigma':=\bigsqcap{}^{{}^{{ \mathfrak{S}}}}_{\alpha\in { \mathfrak{S}}^{{}^{pure}}\smallsetminus \underline{\sigma}_{{}_{{ \mathfrak{S}}}}}\alpha$.  We then define, for any $\eta\in { \mathfrak{S}}$,  $\psi_\sigma (\eta):=\textit{\bf Y}$ iff $\sigma\sqsubseteq_{{}_{ \mathfrak{S}}}\eta$, and $\psi_\sigma (\eta):=\textit{\bf N}$ iff $\sigma'\sqsubseteq_{{}_{ \mathfrak{S}}}\eta$, and $\psi_\sigma (\eta):=\bot$ otherwise. The map $\psi_\sigma$ satisfies the requirements. The unicity of $\psi_\sigma$ is clear.\\ 
Concerning the reciprocal implication, it suffices to define $U_\sigma:=\psi_\sigma^{-1}(\textit{\bf Y})\cap { \mathfrak{S}}^{{}^{pure}}$. The unicity of $U_\sigma$ is trivial to check from the unicity of $\psi_\sigma$.
\end{proof}

\begin{definition}
According to \cite[definition p.117]{Gratzer1971}, an Inf semi-lattice ${ \mathfrak{S}}$ is said to be {\em distributive} iff 
\begin{eqnarray}
&&\hspace{-2cm} \forall \sigma,\sigma_1,\sigma_2\in { \mathfrak{S}}\;\vert\; \sigma\not= \sigma_1,\sigma_2 ,\;\;\;\;\;\; (\sigma_1 \sqcap_{{}_{ \mathfrak{S}}} \sigma_2) \sqsubseteq_{{}_{ \mathfrak{S}}}  \sigma  \Rightarrow \nonumber\\ 
&&\exists \sigma'_1,\sigma'_2\in { \mathfrak{S}}\;\vert\; (\,\sigma_1 \sqsubseteq_{{}_{ \mathfrak{S}}} \sigma'_1,\;\;\; \sigma_2 \sqsubseteq_{{}_{ \mathfrak{S}}} \sigma'_2\;\;\;\textit{\rm and}\;\;\; \sigma = \sigma'_1 \sqcap_{{}_{ \mathfrak{S}}} \sigma'_2\,).\;\;\;\;\;\;\;\;\;\;\;\;
\end{eqnarray} 
\end{definition}

\begin{lemma}
A simplex space of states is necessarily distributive as an Inf semi-lattice. Conversely, a distributive Inf semi-lattice is also necessarily a simplex space of states.
\end{lemma}
\begin{proof}
Let us first assume that ${\mathfrak{S}}$ is a simplex. 
Let us consider $\sigma,\sigma_1,\sigma_2\in { \mathfrak{S}}$ such that $\sigma\not= \sigma_1,\sigma_2 $ and $ (\sigma_1 \sqcap_{{}_{ \mathfrak{S}}} \sigma_2) \sqsubseteq_{{}_{ \mathfrak{S}}}  \sigma$. We have then $\underline{\sigma}_{{}_{{ \mathfrak{S}}}} \subseteq (\underline{\sigma_1}_{{}_{{ \mathfrak{S}}}}\cup \underline{\sigma_2}_{{}_{{ \mathfrak{S}}}})$. If we define $\sigma'_1:=\bigsqcap{}^{{}^{\mathfrak{S}}}(\underline{\sigma_1}_{{}_{{ \mathfrak{S}}}}\cap \underline{\sigma}_{{}_{{ \mathfrak{S}}}})$ and $\sigma'_2:=\bigsqcap{}^{{}^{\mathfrak{S}}}(\underline{\sigma_2}_{{}_{{ \mathfrak{S}}}}\cap \underline{\sigma}_{{}_{{ \mathfrak{S}}}})$ we check immediately that $\sigma_1 \sqsubseteq_{{}_{ \mathfrak{S}}} \sigma'_1,\;\;\; \sigma_2 \sqsubseteq_{{}_{ \mathfrak{S}}} \sigma'_2$ and $\sigma = \sigma'_1 \sqcap_{{}_{ \mathfrak{S}}} \sigma'_2$. As a result, ${\mathfrak{S}}$ is distributive.\\
Let us now suppose that ${\mathfrak{S}}$ is distributive. Let us imagine that ${\mathfrak{S}}$ is not a simplex.  It would exist $\sigma\in {\mathfrak{S}}$ and $U\varsubsetneq \underline{\sigma}_{{}_{{\mathfrak{S}}}}$ such that $\sigma=\bigsqcap{}^{{}^{{\mathfrak{S}}}} U$. Let us suppose that $U$ is a minimal subset for inclusion among subsets $V$ satisfying $\sigma=\bigsqcap{}^{{}^{{\mathfrak{S}}}} V$. We then introduce $\sigma_3\in \underline{\sigma}_{{}_{{\mathfrak{S}}}} \smallsetminus U$ and $U_1,U_2\not=\varnothing$ such that $U_1\cap U_2=\varnothing$ and $U_1\cup U_2=U$, we denote $\sigma_1:=\bigsqcap{}^{{}^{{\mathfrak{S}}}} U_1$ and $\sigma_2:=\bigsqcap{}^{{}^{{\mathfrak{S}}}} U_2$. We have $\sigma \sqsupseteq_{{}_{{\mathfrak{S}}}} (\sigma_1 \sqcap_{{}_{{\mathfrak{S}}}}\sigma_2)$. Because of the distributivity of ${\mathfrak{S}}$, it would exist $\sigma'_1$ and $\sigma'_2$ such that $\sigma'_1 \sqsupseteq_{{}_{{\mathfrak{S}}}} \sigma_1$, $\sigma'_2 \sqsupseteq_{{}_{{\mathfrak{S}}}} \sigma_2$ and 
$\sigma_3 = (\sigma'_1 \sqcap_{{}_{{\mathfrak{S}}}}\sigma'_2)$. However, $\sigma_3$ being meet-irreducible we must have $\sigma_3 = \sigma'_1$ or $\sigma_3 = \sigma'_2$, and then necessarily $\sigma_3\in U$. We have then obtained the announced contradiction.The space of states ${\mathfrak{S}}$ is then a simplex. 
\end{proof}

\begin{lemma}\label{lemmathreestatesnotsimplex}
If ${ \mathfrak{S}}$ is a non-simplex space of states, then we have
\begin{eqnarray}
\exists \sigma_1,\sigma_2,\sigma_3\in { \mathfrak{S}} & \vert &\sigma_3 \sqsupseteq_{{}_{{ \mathfrak{S}}}} (\sigma_1\sqcap_{{}_{{ \mathfrak{S}}}}\sigma_2)\;\;\textit{\rm and} \;\; \neg \widehat{\sigma_3\sigma_1}{}^{{}^{{ \mathfrak{S}}}}
\;\;\textit{\rm and}\;\;  \neg \widehat{\sigma_3\sigma_2}{}^{{}^{{ \mathfrak{S}}}}\;\;\textit{\rm and}\;\;  \sigma_1\parallel_{{}_{{ \mathfrak{S}}}}\sigma_2.\;\;\;\;\;\;\;\;\;\;\;\;\;\;\label{threestatesnotsimplex}
\end{eqnarray}
The reciprocal implication is also true.
\end{lemma}
\begin{proof}
If ${ \mathfrak{S}}$ does not satisfy (\ref{simplexdecompunique}), then there exists a subset $U$ of $\underline{\sigma}_{{}_{{ \mathfrak{S}}}}$ which is minimal for the inclusion among the subsets of $\underline{\sigma}_{{}_{{ \mathfrak{S}}}}$ satisfying $\sigma=\bigsqcap{}^{{}^{\mathfrak{S}}} U$, and there exists $\sigma'$ element of $\underline{\sigma}_{{}_{{ \mathfrak{S}}}} \smallsetminus U$.\\ Let us then introduce $U_1$ and $U_2$ such that $U_1\cap U_2=\varnothing$ and $U_1\cup U_2=U$. Now, we define $\sigma_1:=\bigsqcap{}^{{}^{\mathfrak{S}}} U_1$, $\sigma_2:=\bigsqcap{}^{{}^{\mathfrak{S}}} U_2$ and $\sigma_3:=\sigma'$. They satisfy (\ref{threestatesnotsimplex}).\\
Conversely, let us assume that the property (\ref{threestatesnotsimplex}) is satisfied. Then, the state $\sigma:=(\sigma_1\sqcap_{{}_{{ \mathfrak{S}}}}\sigma_2)$ satisfies $\sigma=\bigsqcap{}^{{}^{\mathfrak{S}}}(\underline{\sigma_1}_{{}_{{ \mathfrak{S}}}}\cup \underline{\sigma_2}_{{}_{{ \mathfrak{S}}}})$ and $\sigma=\bigsqcap{}^{{}^{\mathfrak{S}}}(\underline{\sigma_1}_{{}_{{ \mathfrak{S}}}}\cup \underline{\sigma_2}_{{}_{{ \mathfrak{S}}}}\cup \underline{\sigma_3}_{{}_{{ \mathfrak{S}}}})$ with $(\underline{\sigma_1}_{{}_{{ \mathfrak{S}}}}\cup \underline{\sigma_2}_{{}_{{ \mathfrak{S}}}})\cap \underline{\sigma_3}_{{}_{{ \mathfrak{S}}}}=\varnothing$ because $\neg \widehat{\sigma_1\sigma_3}{}^{{}^{{ \mathfrak{S}}}}$ and $\neg \widehat{\sigma_2\sigma_3}{}^{{}^{{ \mathfrak{S}}}}$. As a conclusion, ${ \mathfrak{S}}$ does not satisfy (\ref{simplexdecompunique}).
\end{proof}

\begin{definition}
The space of states ${ \mathfrak{S}}$ is said to be {\em orthocomplemented} iff there exists a map $\star : { \mathfrak{S}}\smallsetminus \{\bot_{{}_{ \mathfrak{S}}}\} \rightarrow { \mathfrak{S}}\smallsetminus \{\bot_{{}_{ \mathfrak{S}}}\}$ such that
\begin{eqnarray}
\forall \sigma\in { \mathfrak{S}}\smallsetminus \{\bot_{{}_{ \mathfrak{S}}}\},&& \sigma^{\star\star}=\sigma \label{involutive}\\
\forall \sigma_1,\sigma_2\in { \mathfrak{S}}\smallsetminus \{\bot_{{}_{ \mathfrak{S}}}\},&& \sigma_1\sqsubseteq_{{}_{{ \mathfrak{S}}}}\sigma_2\;\;\Rightarrow\;\; \sigma_2^\star \sqsubseteq_{{}_{{ \mathfrak{S}}}} \sigma_1^\star \label{orderreversing}\\
\forall \sigma\in { \mathfrak{S}}\smallsetminus \{\bot_{{}_{ \mathfrak{S}}}\},&&\sigma\, \widecheck{\bowtie}_{{}_{ \mathfrak{S}}}\sigma^{\star},\label{starcomplement}
\end{eqnarray}
(see (\ref{quasiantipodal}) for the definition of $\widecheck{\bowtie}$).
\end{definition}

By the way, we note the following fact

\begin{lemma} 
A simplex space of states ${ \mathfrak{S}}$ is always orthocomplemented, as long as we fix for any $\sigma\in { \mathfrak{S}}$ : $\sigma^\star := \bigsqcap{}^{{}^{{ \mathfrak{S}}}}_{\alpha\in { \mathfrak{S}}^{{}^{pure}}\!\!\!\!\smallsetminus \underline{\sigma}_{{}_{{ \mathfrak{S}}}}} \alpha$.
\end{lemma}

\subsection{Reduction of the space of effects}\label{subsectionreduction}

In the rest of this paper, we will adopt the following generic reduction of the space of effects, as soon as the space of states is orthocomplemented :

\begin{definition}
\label{definreductionquantum} 
Let ${ \mathfrak{S}}$ be a space of states, ${ \mathfrak{E}}_{ \mathfrak{S}}$ and $\epsilon^{ \mathfrak{S}}$ be respectively the natural space of effects and the evaluation map introduced in previous subsection.\\
If the space of states ${ \mathfrak{S}}$ is orthocomplemented, we can define a {\em reduction} of the natural space of effects ${ \mathfrak{E}}_{ \mathfrak{S}}$, denoted by $\overline{{ \mathfrak{E}}}_{ \mathfrak{S}}$, as follows :  
\begin{eqnarray}
\hspace{-2cm}\overline{{ \mathfrak{E}}}_{ \mathfrak{S}} & := &
\{\,{ \mathfrak{l}}_{(\sigma,\sigma')}\;\vert\; \sigma,\sigma'\in { \mathfrak{S}}\smallsetminus \{\bot_{{}_{{ \mathfrak{S}}}}\}, \sigma'\sqsupseteq_{{}_{ \mathfrak{S}}} \sigma^\star\,\} \cup \nonumber\\
&&\;\;\;\;\;\;\;\;\;\;\;\;\;\; \{ { \mathfrak{l}}_{{}_{(\sigma,\cdot)}}\;\vert\; \sigma\in { \mathfrak{S}}\;\}\cup \{ { \mathfrak{l}}_{{}_{(\cdot,\sigma)}}\;\vert\; \sigma\in { \mathfrak{S}}\;\}\cup \{\, { \mathfrak{l}}_{{}_{(\cdot,\cdot)}}\,\} \;\textit{\rm as a set} \;\;\;\;\;\;\;\;\;\;\;\;\;
\end{eqnarray} 
and the Inf semi-lattice structure on $\overline{{ \mathfrak{E}}}_{ \mathfrak{S}}$ is induced from that defined on ${ \mathfrak{E}}_{ \mathfrak{S}}$, i.e. defined by (\ref{defcapES}). The evaluation map is the restriction of $\epsilon^{ \mathfrak{S}}$ to $\overline{{ \mathfrak{E}}}_{ \mathfrak{S}}$.
\end{definition}

\begin{theorem}
According to Theorem \ref{TheoremreductionE}, the reduced space of effects $\overline{{ \mathfrak{E}}}_{ \mathfrak{S}}$ is a well defined space of effects associated to ${ \mathfrak{S}}$. In other words, $({ \mathfrak{S}}, \overline{{ \mathfrak{E}}}_{ \mathfrak{S}}, \epsilon^{ \mathfrak{S}})$ is a well defined States/Effects Chu space. 
\end{theorem}

\begin{theorem}\label{decomppurestatesreducedeffectsspace}
The reduced space of effects $\overline{{ \mathfrak{E}}}_{ \mathfrak{S}}$ is defined in terms of pure effects :
\begin{eqnarray}
&&\forall { \mathfrak{l}} \in \overline{ \mathfrak{E}}_{ \mathfrak{S}}, \;\; { \mathfrak{l}}= \bigsqcap{}^{{}^{ { \overline{{ \mathfrak{E}}}_{ \mathfrak{S}}}}}  \underline{\;{ \mathfrak{l}}\;}_{{}_{ \overline{ \mathfrak{E}}_{ \mathfrak{S}}}},\;\;\textit{\rm where}\;\;
\underline{\;{ \mathfrak{l}}\;}_{{}_{ {\overline{ \mathfrak{E}}_{ \mathfrak{S}}}}}=
(\overline{ \mathfrak{E}}{}_{ \mathfrak{S}}^{{}^{pure}} \cap (\uparrow^{{}^{ \overline{ \mathfrak{E}}_{ \mathfrak{S}}}}\!\!\! { \mathfrak{l}}) )
\end{eqnarray}
with 
\begin{eqnarray}
\overline{{ \mathfrak{E}}}_{ \mathfrak{S}}^{{}^{pure}}=Max(\overline{{ \mathfrak{E}}}_{ \mathfrak{S}})=
\{\,{ \mathfrak{l}}_{(\sigma,\sigma^\star)}\;\vert\; \sigma\in { \mathfrak{S}}\smallsetminus \{\bot_{{}_{{ \mathfrak{S}}}}\}\,\} \cup \{\, { \mathfrak{Y}}_{{}_{{{ \mathfrak{E}}}_{ \mathfrak{S}}}}\,\}\cup \{\, \overline{{ \mathfrak{Y}}_{{}_{{{ \mathfrak{E}}}_{ \mathfrak{S}}}}}\,\}.
\end{eqnarray}
\end{theorem}
\begin{proof}
It suffices to exploit the remark \ref{remarkspaceeffectspurestates} and to note that
\begin{eqnarray}
\forall \sigma,\sigma'\in { \mathfrak{S}}\;\vert\; \sigma^\star\sqsubseteq_{{}_{ \mathfrak{S}}}\sigma', && { \mathfrak{l}}_{(\sigma,\sigma')}={ \mathfrak{l}}_{(\sigma,\sigma^\star)} \sqcap_{{}_{{ \mathfrak{E}}_{ \mathfrak{S}}}} { \mathfrak{l}}_{(\sigma'{}^\star,\sigma')}.
\end{eqnarray}
The set of maximal elements of $\overline{{ \mathfrak{E}}}_{ \mathfrak{S}}$ being a generating subset of $\overline{{ \mathfrak{E}}}_{ \mathfrak{S}}$, this subset is then equal to the set of meet-irreducible elements of $\overline{{ \mathfrak{E}}}_{ \mathfrak{S}}$.
\end{proof}

\subsection{Morphisms}\label{subsectionchannels}

We turn the collection of States/Effects Chu spaces into a category by defining the following morphisms.

\begin{definition}
We will consider the morphisms from a States/Effects Chu space $({ \mathfrak{S}}_{A},{ \mathfrak{E}}_{A},\epsilon^{{ \mathfrak{S}}_A})$ to another States/Effects Chu space $({ \mathfrak{S}}_{B},{ \mathfrak{E}}_{B},\epsilon^{{ \mathfrak{S}}_B})$, i.e.  pairs of maps $f : { \mathfrak{S}}_{A} \rightarrow { \mathfrak{S}}_{B}$ and $f^{\ast} : { \mathfrak{E}}_{B}\rightarrow { \mathfrak{E}}_{A}$ satisfying the following properties (see \cite{Pratt1999}) 
\begin{eqnarray} 
\forall \sigma_A\in { \mathfrak{S}}_{A}, \forall { \mathfrak{l}}_B \in { \mathfrak{E}}_{B}&&\epsilon^{{ \mathfrak{S}}_B}_{{ \mathfrak{l}}_B}(f(\sigma_A))=\epsilon^{{ \mathfrak{S}}_A}_{f^{\ast}({ \mathfrak{l}}_B)}(\sigma_A).\label{defchumorphism}
\end{eqnarray}
\end{definition}

\begin{remark}
Note that, the eventual surjectivity of $f^{\ast}$ implies the injectivity of $f$.  This point uses the property (\ref{Chuseparated}).  Explicitly,
\begin{eqnarray}
\forall \sigma_A,\sigma'_A\in { \mathfrak{S}}_{A},\;\; f(\sigma_A)=f(\sigma'_A) &\Leftrightarrow & (\,\forall { \mathfrak{l}}_B\in { \mathfrak{E}}_{B},\;\epsilon^{{ \mathfrak{S}}_B}_{{ \mathfrak{l}}_B}(f(\sigma_A))=\epsilon^{{ \mathfrak{S}}_B}_{{ \mathfrak{l}}_B}(f(\sigma'_A))\,)\nonumber\\
&\Leftrightarrow & (\,\forall { \mathfrak{l}}_B\in { \mathfrak{E}}_{B},\;\epsilon^{{ \mathfrak{S}}_B}_{f^{\ast}({ \mathfrak{l}}_B)}(\sigma_A)=\epsilon^{{ \mathfrak{S}}_B}_{f^{\ast}({ \mathfrak{l}}_B)}(\sigma'_A)\,)\nonumber\\
&\Leftrightarrow & (\,\forall { \mathfrak{l}}'_B\in { \mathfrak{E}}_{B},\;\epsilon^{{ \mathfrak{S}}_B}_{{ \mathfrak{l}}'_B}(\sigma_A)=\epsilon^{{ \mathfrak{S}}_B}_{{ \mathfrak{l}}'_B}(\sigma'_A)\,)\nonumber\\
&\Leftrightarrow & (\, \sigma_A=\sigma'_A\,).
\end{eqnarray}
In the same way,  using the properties (\ref{Chuextensional}) and the surjectivity of $f$, we can deduce the injectivity of $f^{\ast}$. 
\end{remark}

The duality property (\ref{defchumorphism}) suffices to deduce the following properties.
\begin{theorem} The left-component $f$ of a Chu morphism from $({ \mathfrak{S}}_A,{ \mathfrak{E}}_A,\epsilon^{{ \mathfrak{S}}_A})$ to $({ \mathfrak{S}}_B,{ \mathfrak{E}}_B,\epsilon^{{ \mathfrak{S}}_B})$ satisfies
\begin{eqnarray}
\forall S\subseteq { \mathfrak{S}}_A,&& f(\bigsqcap{}^{{}^{{ \mathfrak{S}}_A}} \! S) = \bigsqcap{}^{{}^{{ \mathfrak{S}}_B}}_{{}_{\sigma\in S}} \; f(\sigma)\label{f12cap}\\
\forall { \mathfrak{C}}\subseteq_{Chain} { \mathfrak{S}}_A,&& f(\bigsqcup{}^{{}^{{ \mathfrak{S}}_A}} \! { \mathfrak{C}}) = \bigsqcup{}^{{}^{{ \mathfrak{S}}_B}}_{{}_{\sigma\in { \mathfrak{C}}}} \; f(\sigma).\label{f12cupchain}
\end{eqnarray}
As a consequence of (\ref{f12cupchain}), $f$ is in particular order-preserving \\
The right-component $f^{\ast}$ of a Chu morphism from $({ \mathfrak{S}}_A,{ \mathfrak{E}}_A,\epsilon^{{ \mathfrak{S}}_A})$ to $({ \mathfrak{S}}_B,{ \mathfrak{E}}_B,\epsilon^{{ \mathfrak{S}}_B})$ satisfies
\begin{eqnarray}
\forall E\subseteq { \mathfrak{E}}_B,&& f^{\ast}(\bigsqcap{}^{{}^{{ \mathfrak{E}}_B}} \! E) = \bigsqcap{}^{{}^{{ \mathfrak{E}}_A}}_{{}_{{ \mathfrak{l}}\in E}} \; f^{\ast}({ \mathfrak{l}})\label{f21cap}\\
\forall C\subseteq_{Chain} { \mathfrak{E}}_B,&& f^{\ast}(\bigsqcup{}^{{}^{{ \mathfrak{E}}_B}} \! C) = \bigsqcup{}^{{}^{{ \mathfrak{E}}_A}}_{{}_{{ \mathfrak{l}}\in C}} \; f^{\ast}({ \mathfrak{l}})\label{f21cupchain}\\
\forall { \mathfrak{l}} \in { \mathfrak{E}}_B,&& f^{\ast}(\; \overline{\,{ \mathfrak{l}}\,} \;) = \overline{f^{\ast}({ \mathfrak{l}})}\label{f21bar}\\
&& f^{\ast}(\;{ \mathfrak{Y}}_{{ \mathfrak{E}}_B}\;) = { \mathfrak{Y}}_{{ \mathfrak{E}}_A}.\label{f21Y}
\end{eqnarray}
In particular, $f^{\ast}$ is order-preserving.
\end{theorem}
\begin{proof}
All proofs follow the same trick based on the duality relation (\ref{defchumorphism}) and the separation property (\ref{Chuseparated}).  For example, for any $S\subseteq { \mathfrak{S}}_A$ and any ${ \mathfrak{l}} \in { \mathfrak{E}}_B$, we have, using (\ref{defchumorphism}) and (\ref{axiomEinfsemilattice}) :
\begin{eqnarray}
\epsilon^{{ \mathfrak{S}}_B}_{{ \mathfrak{l}}}(f(\bigsqcap{}^{{}^{{ \mathfrak{S}}_A}} \!\! S)) &=& \epsilon^{{ \mathfrak{S}}_A}_{f^{\ast}({ \mathfrak{l}})}(\bigsqcap{}^{{}^{{ \mathfrak{S}}_A}} \!\! S)\nonumber\\
&=& \bigwedge{}_{\sigma\in S} \; \epsilon^{{ \mathfrak{S}}_A}_{f^{\ast}({ \mathfrak{l}})}(\sigma)\nonumber\\
&=&  \bigwedge{}_{\sigma\in S} \; \epsilon^{{ \mathfrak{S}}_B}_{{ \mathfrak{l}}}(f(\sigma))\nonumber\\
&=&  \epsilon^{{ \mathfrak{S}}_B}_{{ \mathfrak{l}}}(\bigsqcap{}^{{}^{{ \mathfrak{S}}_B}}_{\sigma\in S} f(\sigma))
\end{eqnarray}
We now use the property (\ref{Chuseparated}) to deduce (\ref{f12cap}).
\end{proof}

\begin{theorem}
For any map $f:{ \mathfrak{S}} \longrightarrow { \mathfrak{S}}'$ satisfying $\forall \{\sigma_i\;\vert\; i\in I\}\subseteq { \mathfrak{S}}, \; f({\bigsqcap{}^{{}^{ \mathfrak{S}}}_{{}_{i\in i}}\sigma_i}) = \bigsqcap{}^{{}^{{ \mathfrak{S}}'}}_{{i\in I}} \;f(\sigma_i)$, there exists a unique map $f^\ast: { \mathfrak{E}}' \longrightarrow { \mathfrak{E}}$ such that 
\begin{eqnarray}
&&\forall { \mathfrak{l}} \in { \mathfrak{E}},\; \epsilon^{{ \mathfrak{S}}}_{f^\ast({ \mathfrak{l}})}(\sigma)=\epsilon^{{ \mathfrak{S}}'}_{ \mathfrak{l}}(f(\sigma)).
\end{eqnarray}
\end{theorem}
\begin{proof}
Direct consequence of Theorem \ref{aepsilonsigma}.
\end{proof}

As a consequence of this theorem, the couple of maps $(f,f^{\ast})$ defining a morphism from $({ \mathfrak{S}}_A,{ \mathfrak{E}}_A,\epsilon^{{ \mathfrak{S}}_A})$ to $({ \mathfrak{S}}_B,{ \mathfrak{E}}_B,\epsilon^{{ \mathfrak{S}}_B})$ can then be reduced to the single data $f$. We will then speak shortly of "the morphism $f$ from the space of states ${ \mathfrak{S}}_{A}$ to the space of states ${ \mathfrak{S}}_{B}$" rather than "the morphism from the states/effects Chu space $({ \mathfrak{S}}_A,{ \mathfrak{E}}_A,\epsilon^{{ \mathfrak{S}}_A})$ to the states/effects Chu space $({ \mathfrak{S}}_B,{ \mathfrak{E}}_B,\epsilon^{{ \mathfrak{S}}_B})$".
\begin{definition}
The space of morphisms from the space of states ${ \mathfrak{S}}_A$ to the space of states ${ \mathfrak{S}}_B$ will be denoted ${ \mathfrak{C}}({{ \mathfrak{S}}_A},{{ \mathfrak{S}}_B})$. It is the space of maps from ${ \mathfrak{S}}_A$ to ${ \mathfrak{S}}_B$ that is order-preserving and satisfies the homomorphic property (\ref{f12cap}).
\end{definition}

\begin{theorem}
The composition of a morphism $(f ,f^{\ast})$ from $({ \mathfrak{S}}_A,{ \mathfrak{E}}_A,\epsilon^{{ \mathfrak{S}}_A})$ to $({ \mathfrak{S}}_B,{ \mathfrak{E}}_B,\epsilon^{{ \mathfrak{S}}_B})$ by another morphism $(g ,g^{\ast})$ defined from $({ \mathfrak{S}}_B,{ \mathfrak{E}}_B,\epsilon^{{ \mathfrak{S}}_B})$ to $({ \mathfrak{S}}_C,{ \mathfrak{E}}_C,\epsilon^{{ \mathfrak{S}}_C})$ is given by 
$( g\circ f ,f^{\ast} \circ g^{\ast})$ defining a valid morphism from $({ \mathfrak{S}}_A,{ \mathfrak{E}}_A,\epsilon^{{ \mathfrak{S}}_A})$ to $({ \mathfrak{S}}_C,{ \mathfrak{E}}_C,\epsilon^{{ \mathfrak{S}}_C})$.
\end{theorem}
\begin{proof}Using two times the duality property, we obtain
\begin{eqnarray} 
\epsilon^{{ \mathfrak{S}}_C}_{{ \mathfrak{l}}_C}(g\circ f (\sigma_A))&=&
\epsilon^{{ \mathfrak{S}}_B}_{g^{\ast}({ \mathfrak{l}}_C)}(f (\sigma_A))=
\epsilon^{{ \mathfrak{S}}_A}_{f^{\ast} \circ g^{\ast}({ \mathfrak{l}}_C)}(\sigma_A).
\end{eqnarray}
\end{proof}

\begin{definition}
We define the infimum of two maps $f$ and $g$ satisfying (\ref{f12cap}) (resp. two maps $f^\ast$ and $g^\ast$ satisfying (\ref{f21cap})) by $\forall \sigma \in { \mathfrak{S}}_A, (f\sqcap g)(\sigma):= f(\sigma) \sqcap_{{}_{{ \mathfrak{S}}_B}} g(\sigma)$ (resp. $\forall { \mathfrak{l}} \in { \mathfrak{E}}_B, (f^\ast\sqcap g^\ast)({ \mathfrak{l}}):= f^\ast({ \mathfrak{l}}) \sqcap_{{}_{{ \mathfrak{E}}_A}} g^\ast({ \mathfrak{l}})$).
\end{definition}

\begin{theorem}
The infimum of a morphism $(f ,f^{\ast})$ from $({ \mathfrak{S}}_A,{ \mathfrak{E}}_A,\epsilon^{{ \mathfrak{S}}_A})$ to $({ \mathfrak{S}}_B,{ \mathfrak{E}}_B,\epsilon^{{ \mathfrak{S}}_B})$ with another morphism $(g ,g^{\ast})$ defined from $({ \mathfrak{S}}_A,{ \mathfrak{E}}_A,\epsilon^{{ \mathfrak{S}}_A})$ to $({ \mathfrak{S}}_B,{ \mathfrak{E}}_B,\epsilon^{{ \mathfrak{S}}_B})$ is given by 
$( g\sqcap f ,f^{\ast} \sqcap g^{\ast})$ defining a valid morphism from $({ \mathfrak{S}}_A,{ \mathfrak{E}}_A,\epsilon^{{ \mathfrak{S}}_A})$ to $({ \mathfrak{S}}_B,{ \mathfrak{E}}_B,\epsilon^{{ \mathfrak{S}}_B})$.
\end{theorem}
\begin{proof}Using the duality property and the homomorphic property, we obtain
\begin{eqnarray} 
\epsilon^{{ \mathfrak{S}}_B}_{{ \mathfrak{l}}_B}((f\sqcap g) (\sigma_A))&=&
\epsilon^{{ \mathfrak{S}}_B}_{{ \mathfrak{l}}_B}(f (\sigma_A))\wedge \epsilon^{{ \mathfrak{S}}_B}_{{ \mathfrak{l}}_B}(g (\sigma_A))\nonumber\\
&=&
\epsilon^{{ \mathfrak{S}}_A}_{f^{\ast}({ \mathfrak{l}}_B)}(\sigma_A) \wedge \epsilon^{{ \mathfrak{S}}_A}_{g^{\ast}({ \mathfrak{l}}_B)}(\sigma_A)\nonumber\\
&=&
\epsilon^{{ \mathfrak{S}}_A}_{(f^{\ast}\sqcap g^{\ast})({ \mathfrak{l}}_B)}(\sigma_A).\;\;\;\;\;\;\;\;\;\;
\end{eqnarray}
\end{proof}

\section{A new perspective on the construction of the tensor product of semi-lattices}\label{sectionnewperspective}

\subsection{The canonical tensor product construction}\label{subsectionbasic}

Let us first introduce the classical construction of G.A. Fraser for the tensor product of semi-lattices \cite{Fraser1976,Fraser1978}. As it will be clarified in the next subsection a new proposal can be made for the tensor product of semi-lattices.

\begin{definition} Let $A, B$ and $C$ be semi-lattices. 
A function $f:A \times B \longrightarrow C$ is a bi-homomorphism if the functions $g_a:B \longrightarrow C$ defined by $g_a(b) = f(a, b)$ and $h_b:A\longrightarrow C$ defined by $h_b(a)=f(a, b)$ are homomorphisms for all $a\in A$ and $b\in B$.
\end{definition}

\begin{theorem}{\bf \cite[Definition 2.2 and Theorem 2.3]{Fraser1976}}\label{theoremtensorbasic}\\
The tensor product ${S}_{AB}:={ \mathfrak{S}}_{A} \otimes { \mathfrak{S}}_{B}$ of the two Inf semi-lattices ${ \mathfrak{S}}_{A}$ and ${ \mathfrak{S}}_{B}$  is obtained as a solution of the following universal problem : there exists a bi-homomorphism, denoted $\iota$ from ${ \mathfrak{S}}_{A} \times { \mathfrak{S}}_{B}$ to ${S}_{AB}$, such that, for any Inf semi-lattice ${ \mathfrak{S}}$ and any bi-homomorphism $f$ from ${ \mathfrak{S}}_{A} \times { \mathfrak{S}}_{B}$ to ${ \mathfrak{S}}$,  there is a unique homomorphism $g$ from ${S}_{AB}$ to ${ \mathfrak{S}}$ with $f = g \circ \iota$. We denote $\iota(\sigma,\sigma')=\sigma \otimes \sigma'$ for any $\sigma\in { \mathfrak{S}}_{A}$ and $\sigma'\in { \mathfrak{S}}_{B}$.  \\
The tensor product ${S}_{AB}$ exists and is unique up to isomorphism, it is built as the homomorphic image of the free $\sqcap$ semi-lattice generated by the set ${ \mathfrak{S}}_{A} \times { \mathfrak{S}}_{B}$ under the congruence relation determined by identifying $(\sigma_1 \sqcap_{{}_{{ \mathfrak{S}}_{A}}}\sigma_2, \sigma')$ with $(\sigma_1,  \sigma')\sqcap (\sigma_2, \sigma')$ for all $\sigma_1,\sigma_2\in { \mathfrak{S}}_{A}, \sigma'\in { \mathfrak{S}}_{B}$ and identifying  $(\sigma, \sigma'_1 \sqcap_{{}_{{ \mathfrak{S}}_{B}}}\sigma'_2)$ with $(\sigma,  \sigma'_1)\sqcap (\sigma, \sigma'_2)$ for all $\sigma\in { \mathfrak{S}}_{A},\sigma'_1,\sigma'_2\in { \mathfrak{S}}_{B}$. \\
From now on, ${S}_{AB}$ is the Inf semi-lattice (the infimum of $S\subseteq {S}_{AB}$ will be denoted $\bigsqcap{}^{{}^{{S}_{AB}}} S$) generated by the elements $\sigma_A \otimes \sigma_B$ with $\sigma_A \in { \mathfrak{S}}_{A}, \sigma_B \in { \mathfrak{S}}_{B}$ and subject to the conditions 
\begin{eqnarray}
&&\left\{\begin{array}{l}
(\sigma_A\sqcap_{{}_{{ \mathfrak{S}}_{A}}}\sigma'_A)\otimes \sigma_B = (\sigma_A\otimes \sigma_B)\sqcap_{{}_{{S}_{AB}}}(\sigma'_A\otimes \sigma_B),\\
 \sigma_A \otimes (\sigma_B\sqcap_{{}_{{ \mathfrak{S}}_{B}}}\sigma'_B) = (\sigma_A\otimes \sigma_B)\sqcap_{{}_{{S}_{AB}}}(\sigma_A\otimes \sigma'_B). 
\end{array}\right.
\end{eqnarray}
\end{theorem}

\begin{definition}
The space ${S}_{AB}={ \mathfrak{S}}_{A} \otimes { \mathfrak{S}}_{B}$ is turned into a partially ordered set with the following binary relation
\begin{eqnarray}
\forall \sigma_{AB},\sigma'_{AB} \in {S}_{AB},\;\;\;\; 
(\,\sigma_{AB} \sqsubseteq_{{}_{{S}_{AB}}} \sigma'_{AB}\,) & :\Leftrightarrow &
(\,\sigma_{AB} \sqcap_{{}_{{S}_{AB}}} \sigma'_{AB} = \sigma_{AB}\,).\;\;\;\;\;\;\;\;\;\;\;\;
\end{eqnarray}
\end{definition}

\begin{definition}
A non-empty subset ${ \mathfrak{R}}$ of ${ \mathfrak{S}}_{A} \times { \mathfrak{S}}_{B}$ is called a bi-filter of ${ \mathfrak{S}}_{A} \times { \mathfrak{S}}_{B}$ iff 
\begin{eqnarray}
&&\forall \sigma_A,\sigma_{1,A},\sigma_{2,A}\in { \mathfrak{S}}_{A},\forall \sigma_B,\sigma_{1,B},\sigma_{2,B}\in { \mathfrak{S}}_{B},\nonumber\\
&&(\,  (\sigma_{1,A},\sigma_{1,B})\leq (\sigma_{2,A},\sigma_{2,B}) \;\;\textit{\rm and}\;\; (\sigma_{1,A},\sigma_{1,B})\in { \mathfrak{R}} \,)\;\;\Rightarrow\;\; (\sigma_{2,A},\sigma_{2,B})\in { \mathfrak{R}},\;\;\;\;\;\;\;\;\;\;\label{defbifilter1}\\
&&(\sigma_{1,A},\sigma_{B}),(\sigma_{2,A},\sigma_{B})\in { \mathfrak{R}}\;\;\Rightarrow\;\; (\sigma_{1,A}\sqcap_{{}_{{ \mathfrak{S}}_A}}\sigma_{2,A},\sigma_{B})\in { \mathfrak{R}},\label{defbifilter2}\\
&&(\sigma_{A},\sigma_{1,B}),(\sigma_{A},\sigma_{2,B})\in { \mathfrak{R}}\;\;\Rightarrow\;\; (\sigma_{A},\sigma_{1,B}\sqcap_{{}_{{ \mathfrak{S}}_B}}\sigma_{2,B})\in { \mathfrak{R}}.\label{defbifilter3}
\end{eqnarray}
\end{definition}
\begin{definition}
If $\{(\sigma_{1,A},\sigma_{1,B}),\cdots,(\sigma_{n,A},\sigma_{n,B})\}$ is a non-empty finite subset of ${ \mathfrak{S}}_{A} \times { \mathfrak{S}}_{B}$, then the intersection of all bi-filters of ${ \mathfrak{S}}_{A} \times { \mathfrak{S}}_{B}$ which contain $(\sigma_{1,A},\sigma_{1,B})$, $\cdots$, $(\sigma_{n,A},\sigma_{n,B})$ is a bi-filter, which we denote by ${ \mathfrak{F}}\{(\sigma_{1,A},\sigma_{1,B}),\cdots,(\sigma_{n,A},\sigma_{n,B})\}$.
\end{definition}

\begin{lemma}\label{FalphaF}
If $F$ is a filter of ${S}_{AB}={ \mathfrak{S}}_{A} \otimes { \mathfrak{S}}_{B}$ then the set $\alpha(F):=\{\,(\sigma_{A},\sigma_{B})\in { \mathfrak{S}}_{A} \times { \mathfrak{S}}_{B}\;\vert\; \sigma_{A}\otimes \sigma_{B} \in F\;\}$ is a bi-filter of ${ \mathfrak{S}}_{A} \times { \mathfrak{S}}_{B}$.
\end{lemma}

\begin{lemma}{\bf \cite[Lemma 1]{Fraser1978}}\label{sigmainFinequality}
Let us choose $\sigma_A,\sigma_{1,A},\cdots,\sigma_{n,A}\in { \mathfrak{S}}_{A}$ and $\sigma_B,\sigma_{1,B},\cdots,\sigma_{n,B}\in { \mathfrak{S}}_{B}$. Then, 
\begin{eqnarray}
\hspace{-1cm}(\sigma_{A},\sigma_{B})\in { \mathfrak{F}}\{(\sigma_{1,A},\sigma_{1,B}),\cdots,(\sigma_{n,A},\sigma_{n,B})\} & \Leftrightarrow & \left( \bigsqcap{}^{{}^{{S}_{AB}}}_{1\leq i\leq n}\; \sigma_{i,A}\otimes \sigma_{i,B}\right) \sqsubseteq_{{}_{{S}_{AB}}} \sigma_{A}\otimes \sigma_{B}.\;\;\;\;\;\;\;\;\;\;\;\;\;\;
\end{eqnarray}
\end{lemma}

\begin{lemma}{\bf \cite[Theorem 1]{Fraser1978}}\label{orderimpliespolynomial}\\
Let us choose $\sigma_A,\sigma_{1,A},\cdots,\sigma_{n,A}\in { \mathfrak{S}}_{A}$ and $\sigma_B,\sigma_{1,B},\cdots,\sigma_{n,B}\in { \mathfrak{S}}_{B}$. Then, 
\begin{eqnarray}
&&\hspace{-2cm} \left( \bigsqcap{}^{{}^{{S}_{AB}}}_{1\leq i\leq n}\; \sigma_{i,A}\otimes \sigma_{i,B}\right) \sqsubseteq_{{}_{{S}_{AB}}} \sigma_{A}\otimes \sigma_{B}  \Leftrightarrow \nonumber\\  
&& (\,\textit{\rm there exists a n$-$ary lattice polynomial $p$}\;\vert\; \sigma_{A}\sqsupseteq_{{}_{{ \mathfrak{S}}_A}} p(\sigma_{1,A},\cdots,\sigma_{n,A})\nonumber\\
&&\textit{\rm and}\;\;  \sigma_{B}\sqsupseteq_{{}_{{ \mathfrak{S}}_B}} p^\ast(\sigma_{1,B},\cdots,\sigma_{n,B}) \,).
\end{eqnarray}
where $p^\ast$ denotes the lattice polynomial obtained from $p$ by dualizing the lattice operations.
\end{lemma}

\subsection{The maximal tensor product}\label{subsectionmaximaltensorproduct}

Let us now consider a radically different approach of tensor product, exploiting the notion of States/Effects Chu spaces. From now on, $({ \mathfrak{S}}_A,{ \mathfrak{E}}_{A},{ \epsilon}^{{ \mathfrak{S}}_A})$ and $({ \mathfrak{S}}_B,{ \mathfrak{E}}_{B},{ \epsilon}^{{ \mathfrak{S}}_B})$ are States/Effects Chu spaces. 

\begin{definition}
We will denote by $\widecheck{S}_{AB}$ (or equivalently by ${ \mathfrak{S}}_A\widecheck{\otimes}{ \mathfrak{S}}_B$) the set of maps $\Phi$ from ${ \mathfrak{E}}_{A}\times { \mathfrak{E}}_{B}$ to ${ \mathfrak{E}}_\bot\cong { \mathfrak{B}}$ satisfying
\begin{eqnarray}
\forall \{{ \mathfrak{l}}_{i,A}\;\vert\; i\in I\}\subseteq { \mathfrak{E}}_{A},\forall { \mathfrak{l}}_{B} \in { \mathfrak{E}}_{B}, && \Phi(\bigsqcap{}^{{}_{{ \mathfrak{E}}_{A}}}_{i\in I} { \mathfrak{l}}_{i,A},{ \mathfrak{l}}_{B})= \bigwedge{}_{i\in I} \;\Phi( { \mathfrak{l}}_{i,A},{ \mathfrak{l}}_{B})\label{bilinear1}\\
\forall \{{ \mathfrak{l}}_{j,B}\;\vert\; j\in J\}\subseteq { \mathfrak{E}}_{B},\forall { \mathfrak{l}}_{A} \in { \mathfrak{E}}_{A}, && \Phi({ \mathfrak{l}}_{A},\bigsqcap{}^{{}_{{ \mathfrak{E}}_{B}}}_{j\in J} { \mathfrak{l}}_{j,B})= \bigwedge{}_{j\in J} \;\Phi({ \mathfrak{l}}_{A}, { \mathfrak{l}}_{j,B}),\;\;\;\;\;\;\;\;\;\;\;\;\;\;\label{bilinear2}\\
\textit{\rm and}&&\nonumber\\
\forall { \mathfrak{l}}_A\in { \mathfrak{E}}_{A} && \Phi (\overline{{ \mathfrak{l}}_A}, { \mathfrak{Y}}_{{ \mathfrak{E}}_{B}})= \overline{\Phi({ \mathfrak{l}}_A,{ \mathfrak{Y}}_{{ \mathfrak{E}}_{B}})},\label{furtherquotient1}\\ 
\forall { \mathfrak{l}}_B\in { \mathfrak{E}}_{B},&&  \Phi ({ \mathfrak{Y}}_{{ \mathfrak{E}}_{A}},\overline{{ \mathfrak{l}}_B})= \overline{\Phi({ \mathfrak{Y}}_{{ \mathfrak{E}}_{A}},{ \mathfrak{l}}_B)},\label{furtherquotient2}\\
&&\Phi( { \mathfrak{Y}}_{{ \mathfrak{E}}_{A}},{ \mathfrak{Y}}_{{ \mathfrak{E}}_{B}}) = \textit{\bf Y}.\label{furtherquotientYY}
\end{eqnarray}
$\widecheck{S}_{AB}$ is called {\em the maximal tensor product} of ${ \mathfrak{S}}_A$ and ${ \mathfrak{S}}_B$.
\end{definition}

\begin{theorem}
$\widecheck{S}_{AB}$ is equipped with the pointwise partial order. It is a down-complete Inf semi-lattice with 
\begin{eqnarray}
\hspace{-1cm}\forall \{\Phi_i\;\vert\; i\in I\}\subseteq \widecheck{S}_{AB},\forall ({ \mathfrak{l}}_{A},{ \mathfrak{l}}_{B}) \in { \mathfrak{E}}_{A} \times { \mathfrak{E}}_{B},&& (\bigsqcap{}^{{}^{\widecheck{S}_{AB}}}_{i\in I} \Phi_i)({ \mathfrak{l}}_{A},{ \mathfrak{l}}_{B}):=\bigwedge{}_{i\in I} \;\Phi_i({ \mathfrak{l}}_{A},{ \mathfrak{l}}_{B}).\;\;\;\;\;\;\;\;\;\;\;\;\;\;\;\;
\end{eqnarray}
\end{theorem}

\begin{theorem} The following maps are homomorphisms
\begin{eqnarray}
\begin{array}{rcrclccrcrcl}
\eta &: & { \widecheck{S}}_{AB} & \longrightarrow & { \mathfrak{E}}^\ast_{{A}}  &&  & \lambda &: & { \widecheck{S}}_{AB} & \longrightarrow & { \mathfrak{E}}^\ast_{{B}}  \\
& & \Phi & \mapsto & \Phi(\cdot,{ \mathfrak{Y}}_{{ \mathfrak{E}}_{{B}}}) & & & & & \Phi & \mapsto & \Phi({ \mathfrak{Y}}_{{ \mathfrak{E}}_{{A}}},\cdot)
\end{array}
\end{eqnarray}
with 
\begin{eqnarray}
{ \mathfrak{E}}^\ast &:= &\{\, \psi \in { \mathfrak{C}}( { \mathfrak{E}},{ \mathfrak{B}})\;\vert\; \forall { \mathfrak{l}} \in { \mathfrak{E}}, \;\psi(\; \overline{\,{ \mathfrak{l}}\,} \;) = \overline{\psi({ \mathfrak{l}})}\;\textit{\rm and}\;  \psi({ \mathfrak{Y}}_{{ \mathfrak{E}}}) = \textit{\bf Y}\,\}.\;\;\;\;\;\;\;\;\;\;
\end{eqnarray}
Moreover, we have, for any space of states ${ \mathfrak{S}}$, the following isomorphism :
\begin{eqnarray}
{ \mathfrak{E}}^\ast \cong { \mathfrak{S}}
\end{eqnarray}
\end{theorem}
\begin{proof}
Let $\Phi$ be an element of ${ \widecheck{S}}_{AB}$.\\
The map $\psi$ from ${ \mathfrak{E}}_{A}$ to ${ \mathfrak{B}}$ defined by $\psi({ \mathfrak{l}}_A):= \Phi({ \mathfrak{l}}_A,{ \mathfrak{Y}}_{{ \mathfrak{E}}_{{B}}})$ is an element of ${ \mathfrak{C}}( { \mathfrak{E}}_{A},{ \mathfrak{B}})$ because of relation (\ref{bilinear1}).  We have $ \forall { \mathfrak{l}} \in { \mathfrak{E}}, \;\psi(\; \overline{\,{ \mathfrak{l}}\,} \;) = \overline{\psi({ \mathfrak{l}})}$ because of relation  (\ref{furtherquotient1}) and $\psi({ \mathfrak{Y}}_{{ \mathfrak{E}}}) = \textit{\bf Y}$ because of relation (\ref{furtherquotientYY}).  As a result, $\psi$ is an element of ${ \mathfrak{E}}^\ast_{{A}}$.\\
In the same way, the map $\psi'$ from ${ \mathfrak{E}}_{B}$ to ${ \mathfrak{B}}$ defined by $\psi'({ \mathfrak{l}}_B):= \Phi({ \mathfrak{Y}}_{{ \mathfrak{E}}_{{B}}},{ \mathfrak{l}}_B)$ is an element of ${ \mathfrak{E}}^\ast_{{B}}$ because of relations (\ref{bilinear2})(\ref{furtherquotient2}) and (\ref{furtherquotientYY}). \\
Secondly, we note the following isomorphism of Inf semi-lattices :
\begin{eqnarray}
\begin{array}{rcrcl} 
\rho &:& { \mathfrak{S}} &{\buildrel \cong \over \longrightarrow} & { \mathfrak{E}}^\ast
\\
& & \sigma & \mapsto & \rho(\sigma) \;\vert\; \rho(\sigma)({ \mathfrak{l}}):=\epsilon^{ \mathfrak{S}}_{{ \mathfrak{l}}}(\sigma),\;\forall { \mathfrak{l}} \in { \mathfrak{E}}.
\end{array}\;\;\;\;\;\;\;\;\;\;\label{isomorphismSpsi}
\end{eqnarray} 
Indeed, for any $\sigma\in { \mathfrak{S}}$, we can define a map $\varphi$ from ${ \mathfrak{E}}$ to ${ \mathfrak{B}}$ by $\varphi ({{ \mathfrak{l}}}):=\epsilon^{ \mathfrak{S}}_{{ \mathfrak{l}}}(\sigma)$. Using the properties (\ref{etbar})(\ref{axiomEinfsemilattice})(\ref{ety}) of $\epsilon^{ \mathfrak{S}}$, we deduce that $\varphi \in { \mathfrak{E}}^\ast$.\\
Reciprocally, using Theorem \ref{blepsilonsigma}, we know that
\begin{eqnarray}
 \forall \varphi \in { \mathfrak{E}}^\ast, && \existunique \; \sigma\in { \mathfrak{S}} \;\; \vert \;\;  \forall { \mathfrak{l}}\in { \mathfrak{E}}, \; \epsilon^{ \mathfrak{S}}_{{ \mathfrak{l}}}(\sigma)=\varphi ({{ \mathfrak{l}}}).
\end{eqnarray}
The bijective character of the map $\rho$ is then established. We have also trivially, for any $\{\sigma_i\;\vert\; i\in I\}\subseteq { \mathfrak{S}}$ the homomorphic property 
$\rho(\bigsqcap{}^{ \mathfrak{S}}_{i\in I}\sigma_i) = \bigsqcap{}_{i\in I}\rho(\sigma_i)$ due to the property (\ref{axiomsigmainfsemilattice}).
\end{proof}

\begin{theorem}
The inclusion of pure tensors in $\widecheck{S}_{AB}$ is realized as follows :
\begin{eqnarray}
\begin{array}{rcrcl}
\iota^{\widecheck{S}_{AB}} &:& {{ \mathfrak{S}}_A}\times {{ \mathfrak{S}}_B} & \hookrightarrow & \widecheck{S}_{AB}\\
& & (\sigma_A,\sigma_B) & \mapsto & \iota^{\widecheck{S}_{AB}}(\sigma_A,\sigma_B)\;\vert\; \forall ({ \mathfrak{l}}_A,{ \mathfrak{l}}_B)\in { \mathfrak{E}}_{A}\times { \mathfrak{E}}_{B},\; \\
& & & & \iota^{\widecheck{S}_{AB}}(\sigma_A,\sigma_B)({ \mathfrak{l}}_A,{ \mathfrak{l}}_B):= \epsilon^{{ \mathfrak{S}}_A}_{{ \mathfrak{l}}_A}(\sigma_A) \bullet \epsilon^{{ \mathfrak{S}}_B}_{{ \mathfrak{l}}_B}(\sigma_B)\;\;\in { \mathfrak{B}}.
\end{array}\label{inclusionpuretensorsScheck}
\end{eqnarray}
\end{theorem}
\begin{proof}
The properties (\ref{bilinear1}) and (\ref{bilinear2}) are direct consequences of the properties (\ref{axiomEinfsemilattice}) and (\ref{distributivitybullet}). The properties (\ref{furtherquotient1}) (\ref{furtherquotient2}) (\ref{furtherquotientYY}) and (\ref{furtherquotientNN}) are direct consequences of the properties (\ref{etbar}) (\ref{ety}) (\ref{expressionbullet}). As a conclusion, $\iota^{\widecheck{S}_{AB}}(\sigma_A,\sigma_B)\in \widecheck{S}_{AB}$ for any $(\sigma_A,\sigma_B)$ in ${{ \mathfrak{S}}_A}\times {{ \mathfrak{S}}_B}$.\\
Let us now consider $(\sigma_A,\sigma_B)$ and $(\sigma'_A,\sigma'_B)$ in ${{ \mathfrak{S}}_A}\times {{ \mathfrak{S}}_B}$ such that $ \iota^{\widecheck{S}_{AB}}(\sigma_A,\sigma_B)= \iota^{\widecheck{S}_{AB}}(\sigma'_A,\sigma'_B)$. We choose first of all ${ \mathfrak{l}}_A:={ \mathfrak{l}}_{(\sigma_A,\cdot)}$ and ${ \mathfrak{l}}_B:={ \mathfrak{l}}_{(\sigma_B,\cdot)}$.  We have 
\begin{eqnarray}
\iota^{\widecheck{S}_{AB}}(\sigma_A,\sigma_B)({ \mathfrak{l}}_A,{ \mathfrak{l}}_B)=\textit{\bf Y}\bullet \textit{\bf Y}=\textit{\bf Y}
\end{eqnarray} 
and then must have $\epsilon^{{ \mathfrak{S}}_A}_{{ \mathfrak{l}}_A}(\sigma'_A)=\textit{\bf Y}$ and $\epsilon^{{ \mathfrak{S}}_B}_{{ \mathfrak{l}}_B}(\sigma'_B)=\textit{\bf Y}$, i.e.  $\sigma_A\sqsubseteq_{{}_{{ \mathfrak{S}}_A}} \sigma'_A$ and $\sigma_B\sqsubseteq_{{}_{{ \mathfrak{S}}_B}} \sigma'_B$. Choosing now ${ \mathfrak{l}}_A:={ \mathfrak{l}}_{(\sigma'_A,\cdot)}$ and ${ \mathfrak{l}}_B:={ \mathfrak{l}}_{(\sigma'_B,\cdot)}$, we justify also $\sigma_A\sqsupseteq_{{}_{{ \mathfrak{S}}_A}} \sigma'_A$ and $\sigma_B\sqsupseteq_{{}_{{ \mathfrak{S}}_B}} \sigma'_B$. The map $\iota^{\widecheck{S}_{AB}}$ is then injective.
\end{proof}

\begin{theorem}\label{bifilterPcheck}
We have the following relations
\begin{eqnarray}
\hspace{-1cm}\forall \{\sigma_{i,A}\;\vert\; i\in I\}\subseteq { \mathfrak{S}}_{A},\forall \sigma_B\in { \mathfrak{S}}_{B},&& \iota^{\widecheck{S}_{AB}}(\bigsqcap{}^{{}^{{ \mathfrak{S}}_{A}}}_{i\in I}\,\sigma_{i,A},\sigma_B)= \bigsqcap{}^{{}^{\widecheck{S}_{AB}}}_{i\in I}\, \iota^{\widecheck{S}_{AB}}(\sigma_{i,A},\sigma_B),\;\;\;\;\;\;\;\;\;\;\;\;\;\;\;\;\;\;\;\;\label{pitensor=tensorpi1check}\\
\forall \{\sigma_{i,B}\;\vert\; i\in I\}\subseteq { \mathfrak{S}}_{B},\forall \sigma_A\in { \mathfrak{S}}_{A},&& \iota^{\widecheck{S}_{AB}}(\sigma_A, \bigsqcap{}^{{}^{{ \mathfrak{S}}_{A}}}_{i\in I}\,\sigma_{i,B})= \bigsqcap{}^{{}^{\widecheck{S}_{AB}}}_{i\in I}\, \iota^{\widecheck{S}_{AB}}(\sigma_{A},\sigma_{i,B}).\label{pitensor=tensorpi2check}
\end{eqnarray}
\end{theorem}
\begin{proof}
This is a direct consequence of properties (\ref{axiomsigmainfsemilattice}) and (\ref{distributivitybullet}). More explicitly,
\begin{eqnarray}
\hspace{-1cm}\forall ({ \mathfrak{l}}_A,{ \mathfrak{l}}_B)\in { \mathfrak{E}}_{A}\times { \mathfrak{E}}_{B},\;\;\;\; \iota^{\widecheck{S}_{AB}}(\bigsqcap{}^{{}^{{ \mathfrak{S}}_{A}}}_{i\in I}\,\sigma_{i,A},\sigma_B)({ \mathfrak{l}}_A,{ \mathfrak{l}}_B) &=& \epsilon^{{ \mathfrak{S}}_A}_{{ \mathfrak{l}}_A}(\bigsqcap{}^{{}^{{ \mathfrak{S}}_{A}}}_{i\in I}\,\sigma_{i,A}) \bullet \epsilon^{{ \mathfrak{S}}_B}_{{ \mathfrak{l}}_B}(\sigma_B)\nonumber\\
&=& (\bigwedge{}_{i\in I}\,\epsilon^{{ \mathfrak{S}}_A}_{{ \mathfrak{l}}_A}(\sigma_{i,A})) \bullet \epsilon^{{ \mathfrak{S}}_B}_{{ \mathfrak{l}}_B}(\sigma_B)\nonumber\\
&=&\bigwedge{}_{i\in I}\, (\epsilon^{{ \mathfrak{S}}_A}_{{ \mathfrak{l}}_A}(\sigma_{i,A}) \bullet \epsilon^{{ \mathfrak{S}}_B}_{{ \mathfrak{l}}_B}(\sigma_B))\nonumber\\
&=&\bigwedge{}_{i\in I}\, \iota^{\widecheck{S}_{AB}}(\sigma_{i,A},\sigma_B)({ \mathfrak{l}}_A,{ \mathfrak{l}}_B)\nonumber\\
&=&(\bigsqcap{}^{{}^{\widecheck{S}_{AB}}}_{i\in I}\, \iota^{\widecheck{S}_{AB}}(\sigma_{i,A},\sigma_B))({ \mathfrak{l}}_A,{ \mathfrak{l}}_B).\;\;\;\;\;\;\;\;\;\;\;\;\;\;\;\;\;\;\;\;
\end{eqnarray}
In other words, $\iota^{\widecheck{S}_{AB}}(\bigsqcap{}^{{}^{{ \mathfrak{S}}_{A}}}_{i\in I}\,\sigma_{i,A},\sigma_B)= \bigsqcap{}^{{}^{\widecheck{S}_{AB}}}_{i\in I}\, \iota^{\widecheck{S}_{AB}}(\sigma_{i,A},\sigma_B)$.
\end{proof}

\begin{definition}
We define ${\widetilde{S}_{AB}}$ to be the sub Inf semi-lattice of ${\widecheck{S}_{AB}}$ generated by the elements $\iota^{\widecheck{S}_{AB}}(\sigma_A,\sigma_B)$ for any $(\sigma_A,\sigma_B)\in { \mathfrak{S}}_{A}\times { \mathfrak{S}}_{B}$.\\
${\widetilde{S}_{AB}}$ will be equivalently denoted ${ \mathfrak{S}}_A\widetilde{\otimes}{ \mathfrak{S}}_B$ and called {\em the minimal tensor product} of ${ \mathfrak{S}}_A$ and ${ \mathfrak{S}}_B$
\end{definition}

\subsection{The minimal tensor product}\label{subsectionminimaltensorproduct}

In the following, the set ${ \mathcal{P}}({ \mathfrak{S}}_{A} \times { \mathfrak{S}}_{B})$ will be equipped with the Inf semi-lattice structure $\cup$.

\begin{definition}
${ \mathcal{P}}({ \mathfrak{S}}_{A} \times { \mathfrak{S}}_{B})$ is equipped with a congruence relation denoted $\approx$ and defined between any two elements $\{\,(\sigma_{i,A},\sigma_{i,B})\;\vert\; i\in I\,\}$ and $\{\,(\sigma'_{j,A},\sigma'_{j,B})\;\vert\; j\in J\,\}$ of ${ \mathcal{P}}({ \mathfrak{S}}_{A} \times { \mathfrak{S}}_{B})$ by
\begin{eqnarray}
&&\hspace{-1cm}(\,\{\,(\sigma_{i,A},\sigma_{i,B})\;\vert\; i\in I\,\} \approx \{\,(\sigma'_{j,A},\sigma'_{j,B})\;\vert\; j\in J\,\}\,) \;\; :\Leftrightarrow \;\;\nonumber\\
&&\hspace{-1cm}(\,\forall {\mathfrak{l}}_A\in { \mathfrak{E}}_{A},\forall {\mathfrak{l}}_B\in { \mathfrak{E}}_{B},\;\;
\bigwedge{}_{i\in I}\; {\epsilon}\,{}^{{ \mathfrak{S}}_{A}}_{{ \mathfrak{l}}_A}(\sigma_{i,A}) \bullet {\epsilon}\,{}^{{ \mathfrak{S}}_{B}}_{{ \mathfrak{l}}_B}(\sigma_{i,B})
=
\bigwedge{}_{j\in J}\; {\epsilon}\,{}^{{ \mathfrak{S}}_{A}}_{{ \mathfrak{l}}_A}(\sigma'_{j,A}) \bullet {\epsilon}\,{}^{{ \mathfrak{S}}_{B}}_{{ \mathfrak{l}}_B}(\sigma'_{j,B})\,).\;\;\;\;\;\;\;\;\;\;\;\;
\end{eqnarray}
The congruence class associated with $U\subseteq { \mathfrak{S}}_{A} \times { \mathfrak{S}}_{B}$ will be denoted $U_\approx$.
\end{definition}

\begin{definition}
We introduce the following injective Inf semi-lattice homomorphism
\begin{eqnarray}
&&\hspace{-1cm}\begin{array}{rcrcl}
\Omega & : & { \mathcal{P}}({ \mathfrak{S}}_{A} \times { \mathfrak{S}}_{B})/_\approx & \hookrightarrow & {{\widecheck{S}}}_{AB}\\
& & \{\,(\sigma_{i,A},\sigma_{i,B})\;\vert\; i\in I\,\}_\approx & \mapsto & \Omega( \{\,(\sigma_{i,A},\sigma_{i,B})\;\vert\; i\in I\,\}_\approx)\;\vert\; \forall {\mathfrak{l}}_A\in { \mathfrak{E}}_{A},\forall {\mathfrak{l}}_B\in { \mathfrak{E}}_{B},
\end{array}\nonumber \\
&&\;\;\;\;\;\;\Omega( \{\,(\sigma_{i,A},\sigma_{i,B})\;\vert\; i\in I\,\}_\approx)({\mathfrak{l}}_A,{\mathfrak{l}}_B):=\bigwedge{}_{i\in I}\; {\epsilon}\,{}^{{ \mathfrak{S}}_{A}}_{{ \mathfrak{l}}_A}(\sigma_{i,A}) \bullet {\epsilon}\,{}^{{ \mathfrak{S}}_{B}}_{{ \mathfrak{l}}_B}(\sigma_{i,B}).
\;\;\;\;\;\;\;\;\;\;\;\;\;\;\;\;\;\;\;\;\;\;\;\;
\end{eqnarray}
We note that ${\widetilde{S}_{AB}}$ is the image of $\Omega$ in ${\widecheck{S}_{AB}}$.  
\end{definition}

If we adopt the following notation
\begin{eqnarray}
\forall \widetilde{\sigma}\in {{\widetilde{S}}}_{AB},\;\;\;\; \langle \widetilde{\sigma} \rangle & := & Max \{\,U\in { \mathcal{P}}({ \mathfrak{S}}_{A} \times { \mathfrak{S}}_{B})\;\vert\; \Omega(U_\approx)  \sqsupseteq_{{}_{{{\widetilde{S}}}_{AB}}} \widetilde{\sigma} \,\}\nonumber\\
& = & \{\,(\sigma_A,\sigma_B) \;\vert\; \Omega(\{(\sigma_A,\sigma_B)\}_\approx)  \sqsupseteq_{{}_{{{\widetilde{S}}}_{AB}}}\widetilde{\sigma} \,\},
\end{eqnarray}
we note the following Galois connection
\begin{eqnarray}
\forall \widetilde{\sigma}\in {{\widetilde{S}}}_{AB},\forall U\in { \mathcal{P}}({ \mathfrak{S}}_{A} \times { \mathfrak{S}}_{B}),&&  \langle \widetilde{\sigma} \rangle \supseteq U \;\;\;\Leftrightarrow\;\;\; \widetilde{\sigma} \sqsubseteq_{{}_{{{\widetilde{S}}}_{AB}}} \Omega (U_\approx).
\end{eqnarray}
Indeed, let us fix $U:=\{\,(\sigma_{i,A},\sigma_{i,B})\;\vert\; i\in I\,\}$ and $\widetilde{\sigma}\in {{\widetilde{S}}}_{AB}$.  We have
\begin{eqnarray}
\langle \widetilde{\sigma} \rangle \supseteq U & \Leftrightarrow & 
\forall i\in I, \Omega(\{(\sigma_{i,A},\sigma_{i,B})\}_\approx)  \sqsupseteq_{{}_{{{\widetilde{S}}}_{AB}}}\widetilde{\sigma}\nonumber\\
& \Leftrightarrow & \forall i\in I,  \forall {\mathfrak{l}}_A\in { \mathfrak{E}}_{A},\forall {\mathfrak{l}}_B\in { \mathfrak{E}}_{B}, \;\; \widetilde{\sigma}({\mathfrak{l}}_A,{\mathfrak{l}}_B) \leq
{\epsilon}\,{}^{{ \mathfrak{S}}_{A}}_{{\mathfrak{l}}_A} (\sigma_{i,A}) \bullet {\epsilon}\,{}^{{ \mathfrak{S}}_{B}}_{{\mathfrak{l}}_B}(\sigma_{i,B})\nonumber\\
 & \Leftrightarrow & 
\forall {\mathfrak{l}}_A\in { \mathfrak{E}}_{A},\forall {\mathfrak{l}}_B\in { \mathfrak{E}}_{B}, \;\;\widetilde{\sigma}({\mathfrak{l}}_A,{\mathfrak{l}}_B) \leq \bigwedge{}_{i\in I}\,
{\epsilon}\,{}^{{ \mathfrak{S}}_{A}}_{{\mathfrak{l}}_A} (\sigma_{i,A}) \bullet {\epsilon}\,{}^{{ \mathfrak{S}}_{B}}_{{\mathfrak{l}}_B}(\sigma_{i,B})\nonumber\\
& \Leftrightarrow & \widetilde{\sigma} \sqsubseteq_{{}_{{{\widetilde{S}}}_{AB}}} \Omega(U_\approx).
\end{eqnarray}  
As a consequence of this Galois relation, we obtain that 
\begin{theorem}
${{\widetilde{S}}}_{AB}$ is a down-complete Inf semi-lattice with
\begin{eqnarray}
\forall \{\,U_i\;\vert\; i\in I\,\}\subseteq { \mathcal{P}}({ \mathfrak{S}}_{A} \times { \mathfrak{S}}_{B}),&&\bigsqcap{}^{{}^{{{\widetilde{S}}}_{AB}}}_{i\in I} \; \Omega((U_i)_\approx)  = \Omega((\bigcup{}_{i\in I}\; U_i)_\approx).\;\;\;\;\;\;\;\;\;\;\;\;\;\;\;\;\label{tensorinfSAB}
\end{eqnarray}
We will adopt the notation $\bigsqcap{}^{{}^{{ \widetilde{S}}_{AB}}}_{i\in I} \sigma_{i,A}\widetilde{\otimes} \sigma_{i,B}:=\Omega(\{\,(\sigma_{i,A},\sigma_{i,B})\;\vert\; i\in I\,\}_\approx)$ for any $\{\,(\sigma_{i,A},\sigma_{i,B})\;\vert\; i\in I\,\}$ in ${ \mathcal{P}}({ \mathfrak{S}}_{A} \times { \mathfrak{S}}_{B})$.
\end{theorem}

\begin{theorem}\label{bifilterPtilde}
We have the following relations
\begin{eqnarray}
\forall \{\sigma_{i,A}\;\vert\; i\in I\}\subseteq { \mathfrak{S}}_{A},\forall \sigma_B\in { \mathfrak{S}}_{B},&& (\bigsqcap{}^{{}^{{ \mathfrak{S}}_{A}}}_{i\in I}\,\sigma_{i,A})\widetilde{\otimes} \sigma_B =  \bigsqcap{}^{{}^{{ \widetilde{S}}_{AB}}}_{i\in I} (\sigma_{i,A}\widetilde{\otimes} \sigma_B),\;\;\;\;\;\;\;\;\;\;\;\;\;\;\;\;\label{pitensor=tensorpi1}\\
\forall \{\sigma_{i,B}\;\vert\; i\in I\}\subseteq { \mathfrak{S}}_{B},\forall \sigma_A\in { \mathfrak{S}}_{A},&& \sigma_A \widetilde{\otimes}  (\bigsqcap{}^{{}^{{ \mathfrak{S}}_{B}}}_{i\in I}\,\sigma_{i,B}) =  \bigsqcap{}^{{}^{{ \widetilde{S}}_{AB}}}_{i\in I} (\sigma_A \widetilde{\otimes} \sigma_{i,B}).\;\;\;\;\;\;\;\;\;\;\;\;\;\;\;\;\label{pitensor=tensorpi2}
\end{eqnarray}
\end{theorem}
\begin{proof}
Rewriting of Theorem \ref{bifilterPcheck}.
\end{proof}

\begin{theorem} The following maps are homomorphisms
\begin{eqnarray}
\hspace{-1cm}\begin{array}{rcrclccrcrcl}
\eta &: & { \widetilde{S}}_{AB} & \longrightarrow & {{ \mathfrak{S}}_{A}}  &&  & \lambda &: & { \widetilde{S}}_{AB} & \longrightarrow & {{ \mathfrak{S}}_{B}}  \\
& & \bigsqcap{}^{{}^{{ \widetilde{S}}_{AB}}}_{i\in I} \sigma_{i,A} \widetilde{\otimes} \sigma_{i,B} & \mapsto & \bigsqcap{}^{{}^{{{ \mathfrak{S}}_{A}}}}_{i\in I} \sigma_{i,A}  & & & & & \bigsqcap{}^{{}^{{ \widetilde{S}}_{AB}}}_{i\in I} \sigma_{i,A} \widetilde{\otimes} \sigma_{i,B} & \mapsto & \bigsqcap{}^{{}^{{{ \mathfrak{S}}_{B}}}}_{i\in I}  \sigma_{i,B}
\end{array}
\end{eqnarray}
\end{theorem}

We can obviously clarify the poset structure on ${ \widetilde{S}}_{AB}$.

\begin{lemma} 
\begin{eqnarray}
&&\hspace{-1.5cm}(\,\bigsqcap{}^{{}^{{ \widetilde{S}}_{AB}}}_{i\in I} \sigma_{i,A}\widetilde{\otimes} \sigma_{i,B} \;\;\sqsubseteq_{{}_{\widetilde{S}_{AB}}} \;\;
\bigsqcap{}^{{}^{{ \widetilde{S}}_{AB}}}_{j\in J} \sigma'_{j,A}\widetilde{\otimes} \sigma'_{j,B} \,) \;\; \Leftrightarrow \;\;\nonumber\\
&&\hspace{-1.5cm}(\,\forall {\mathfrak{l}}_A\in { \mathfrak{E}}_{A},\forall {\mathfrak{l}}_B\in { \mathfrak{E}}_{B},\;\;
\bigwedge{}_{i\in I}\; {\epsilon}\,{}^{{ \mathfrak{S}}_{A}}_{{ \mathfrak{l}}_A}(\sigma_{i,A}) \bullet {\epsilon}\,{}^{{ \mathfrak{S}}_{B}}_{{ \mathfrak{l}}_B}(\sigma_{i,B})
\leq
\bigwedge{}_{j\in J}\; {\epsilon}\,{}^{{ \mathfrak{S}}_{A}}_{{ \mathfrak{l}}_A}(\sigma'_{j,A}) \bullet {\epsilon}\,{}^{{ \mathfrak{S}}_{B}}_{{ \mathfrak{l}}_B}(\sigma'_{j,B})\,).\;\;\;\;\;\;\;\;\;\;\;\;\;\;\;
\end{eqnarray}
\end{lemma}

This poset structure can be "explicited" according to following lemma addressing the word problem in ${{\widetilde{S}}}_{AB}$.

\begin{lemma}\label{Lemmadevelopetildeleqetilde}
\begin{eqnarray}
&& \hspace{-2cm}\left( \bigsqcap{}^{{}^{{ \widetilde{S}}_{AB}}}_{i\in I} \sigma_{i,A}\widetilde{\otimes} \sigma_{i,B} \;\;\sqsubseteq_{{}_{{{\widetilde{S}}}_{AB}}} \;\sigma_{A}\widetilde{\otimes} \sigma_{B} \right)
 \nonumber\\
&\Leftrightarrow &\left( (\bigsqcap{}^{{}_{{ \mathfrak{S}}_{A}}}_{k\in I}\; \sigma_{k,A}) \;\sqsubseteq_{{}_{{ \mathfrak{S}}_{A}}}\sigma_{A}
\;\;\textit{\rm and}\;\;
(\bigsqcap{}^{{}_{{ \mathfrak{S}}_{B}}}_{m\in I} \;\sigma_{m,B})\; \sqsubseteq_{{}_{{ \mathfrak{S}}_{B}}} \sigma_{B}
\;\;\textit{\rm and}\;\;\right.\nonumber\\
&&\left.\left(\forall \varnothing \varsubsetneq K  \varsubsetneq I, \;\; (\bigsqcap{}^{{}_{{ \mathfrak{S}}_{A}}}_{k\in K}\; \sigma_{k,A}) \;\sqsubseteq_{{}_{{ \mathfrak{S}}_{A}}}\sigma_{A}\;\;\;\textit{\rm or}\;\;\; (\bigsqcap{}^{{}_{{ \mathfrak{S}}_{B}}}_{m\in I-K} \;\sigma_{m,B})\; \sqsubseteq_{{}_{{ \mathfrak{S}}_{B}}} \sigma_{B}\right) \right)\;\;\;\;\;\;\;\;\;\;\;\;\;\;\;\;\;\;\;\;\;\;\;
\label{developmentetildeordersimplify}\\
&\Leftrightarrow & \left( 
\exists { \mathcal{K}},{ \mathcal{K}}'\subseteq 2^I\;\;\textit{\rm with}\;\; { \mathcal{K}}\cup { \mathcal{K}}'=2^I, { \mathcal{K}}\cap { \mathcal{K}}'=\varnothing, \{\varnothing\}\in { \mathcal{K}}', I\in { \mathcal{K}}\;\vert \right. \;\;\nonumber\\
&& \left. (\bigsqcup{}^{{}_{{ \mathfrak{S}}_{A}}}_{K\in { \mathcal{K}}}\bigsqcap{}^{{}_{{ \mathfrak{S}}_{A}}}_{k\in K}\; \sigma_{k,A}) \;\sqsubseteq_{{}_{{ \mathfrak{S}}_{A}}}\sigma_{A}\;\;\;\textit{\rm and}\;\;\; (\bigsqcup{}^{{}_{{ \mathfrak{S}}_{A}}}_{K'\in { \mathcal{K}}'}\bigsqcap{}^{{}_{{ \mathfrak{S}}_{B}}}_{m\in I-K'} \;\sigma_{m,B})\; \sqsubseteq_{{}_{{ \mathfrak{S}}_{B}}} \sigma_{B}\right).\label{developmentetildeordersimplifypost}
\end{eqnarray}
\end{lemma}
\begin{proof}
The inequality $\bigsqcap{}^{{}^{{ \widetilde{S}}_{AB}}}_{i\in I} \sigma_{i,A}\widetilde{\otimes} \sigma_{i,B} \sqsubseteq_{{}_{{{\widetilde{S}}}_{AB}}} \sigma_{A}\widetilde{\otimes} \sigma_{B}$ 
 is equivalent to
\begin{eqnarray}
\forall {\mathfrak{l}}_A\in { \mathfrak{E}}_{A}, \forall {\mathfrak{l}}_B\in { \mathfrak{E}}_{B},&&
\left( \bigwedge{}_{i\in I}\;
{\epsilon}\,{}^{{ \mathfrak{S}}_{A}}_{{\mathfrak{l}}_A} ( \sigma_{i,A})\bullet {\epsilon}\,{}^{{ \mathfrak{S}}_{B}}_{{\mathfrak{l}}_B} ( \sigma_{i,B})\right)
\leq
{\epsilon}\,{}^{{ \mathfrak{S}}_{A}}_{{\mathfrak{l}}_A} (\sigma_{A}) \bullet {\epsilon}\,{}^{{ \mathfrak{S}}_{B}}_{{\mathfrak{l}}_B}(\sigma_{B}).\;\;\;\;\;\;\;\;\;\;\;\;\;\;\;\;\;\;\label{assumptionorder2}
\end{eqnarray}
We intent to choose a pertinent set of effects ${\mathfrak{l}}_A\in { \mathfrak{E}}_{A}$
 and ${\mathfrak{l}}_B\in { \mathfrak{E}}_{B}$ to reformulate this inequality. \\
Let us firstly choose ${\mathfrak{l}}_B={ \mathfrak{Y}}_{{}_{{ \mathfrak{E}}_{B}}}$. Using  (\ref{expressionbullet}), we obtain
\begin{eqnarray}
{\epsilon}\,{}^{{ \mathfrak{S}}_{A}}_{{ \mathfrak{l}}_A}(\bigsqcap{}^{{}^{{ \mathfrak{S}}_{A}}}_{i\in I}\; \sigma_{i,A})\leq {\epsilon}\,{}^{{ \mathfrak{S}}_{A}}_{{ \mathfrak{l}}_A}(\sigma_{A})
,\forall { \mathfrak{l}}_A \in { \mathfrak{E}}_A,
\end{eqnarray} 
which leads immediately (using (\ref{Chuseparated}))
\begin{eqnarray}
\bigsqcap{}^{{}^{{ \mathfrak{S}}_{A}}}_{i\in I}\; \sigma_{i,A} \;\sqsubseteq_{{}_{{ \mathfrak{S}}_{A}}} \sigma_{A}.\label{ineqBY}
\end{eqnarray}
Choosing ${\mathfrak{l}}_A={ \mathfrak{Y}}_{{}_{{ \mathfrak{E}}_{A}}}$, we obtain along the same line 
\begin{eqnarray}
\bigsqcap{}^{{}^{{ \mathfrak{S}}_{B}}}_{i\in I} \;\sigma_{i,B}\; \sqsubseteq_{{}_{{ \mathfrak{S}}_{B}}} \sigma_{B}.\label{ineqAY}
\end{eqnarray}
Let us now consider $\varnothing \varsubsetneq K  \varsubsetneq I$ and let us choose ${ \mathfrak{l}}_A$ and ${ \mathfrak{l}}_B$ according to
\begin{eqnarray}
{\epsilon}\,{}^{{ \mathfrak{S}}_{A}}_{{ \mathfrak{l}}_A}(\sigma):=\textit{\bf N},\forall \sigma \sqsupseteq_{{}_{{ \mathfrak{S}}_{A}}} \bigsqcap{}^{{}^{{ \mathfrak{S}}_{A}}}_{k\in K}\; \sigma_{k,A}&\textit{ \rm and}&{\epsilon}\,{}^{{ \mathfrak{S}}_{A}}_{{ \mathfrak{l}}_A}(\sigma):=\bot, \;\;\textit{ \rm elsewhere}, \;\;\;\;\;\;\;\;\;\;\;\;\;\;\;\;\;\;\\
{\epsilon}\,{}^{{ \mathfrak{S}}_{B}}_{{ \mathfrak{l}}_B}(\sigma):=\textit{\bf N},\forall \sigma \sqsupseteq_{{}_{{ \mathfrak{S}}_{B}}} \bigsqcap{}^{{}^{{ \mathfrak{S}}_{B}}}_{m\in I-K}\; \sigma_{m,B}&\textit{ \rm and}& {\epsilon}\,{}^{{ \mathfrak{S}}_{B}}_{{ \mathfrak{l}}_B}(\sigma):=\bot, \;\;\textit{ \rm elsewhere}.
\end{eqnarray}
We deduce, from the assumption (\ref{assumptionorder2}), that, for this $\varnothing \varsubsetneq K  \varsubsetneq I$, we have
\begin{eqnarray}
 &&(\bigsqcap{}^{{}^{{ \mathfrak{S}}_{A}}}_{k\in K}\; \sigma_{k,A} \;\sqsubseteq_{{}_{{ \mathfrak{S}}_{A}}} \sigma_{A}) \;\;\textit{\rm or}\;\; (\bigsqcap{}^{{}^{{ \mathfrak{S}}_{B}}}_{m\in I-K} \;\sigma_{m,B}\; \sqsubseteq_{{}_{{ \mathfrak{S}}_{B}}} \sigma_{B}).\label{ineqN}
\end{eqnarray}
We let the reader check that we have obtained the whole set of independent inequalities reformulating the property (\ref{assumptionorder2}).
\end{proof}

\begin{theorem}
If ${ \mathfrak{S}}_{A}$ and ${ \mathfrak{S}}_{B}$ admit $\bot_{{}_{{\mathfrak{S}}_A}}$ and $\bot_{{}_{{\mathfrak{S}}_B}}$ respectively as bottom elements, then
${{\widetilde{S}}}_{AB}$ admits a bottom element explicitly given by $\bot_{{}_{{{\widetilde{S}}}_{AB}}}=\bot_{{}_{{\mathfrak{S}}_A}}\widetilde{\otimes}\bot_{{}_{{\mathfrak{S}}_B}}$.
\end{theorem}
\begin{proof}
Trivial using the expansion (\ref{developmentetildeordersimplify}).
\end{proof}

\subsection{Canonical vs. minimal tensor product}\label{subsectioncanonical}

\begin{definition}
We denote by ${\widetilde{S}}{}^{fin}_{AB}$ the sub-poset of ${\widetilde{S}}_{AB}$ defined as follows :
\begin{eqnarray}
{\widetilde{S}}{}^{fin}_{AB} &:=& \{ \Omega(U_\approx) \;\vert\; U \;\textit{\rm finite subset of}\; { \mathfrak{S}}_A\times { \mathfrak{S}}_B\,\}.
\end{eqnarray}
It is also a sub- Inf semi-lattice of ${\widetilde{S}}_{AB}$.
\end{definition}

\begin{theorem} \label{theoremsqsubseteqPABgeqSAB} 
We have the following obvious property relating the partial orders of ${{\widetilde{S}}}{}^{fin}_{AB}$ and ${S}_{AB}$. For any $\{(\sigma_{i,A},\sigma_{i,B})\;\vert\; i\in I\}\subseteq_{fin} { \mathfrak{S}}_{A}\times { \mathfrak{S}}_{B}$,
\begin{eqnarray}
(\bigsqcap{}^{{}^{{S}_{AB}}}_{i\in I} \sigma_{i,A}\otimes \sigma_{i,B}) \sqsubseteq_{{}_{{S}_{AB}}}  \sigma'_{A}\otimes \sigma'_{B}  & \Rightarrow & (\bigsqcap{}^{{}^{{\widetilde{S}}_{AB}}}_{i\in I} \sigma_{i,A}\widetilde{\otimes} \sigma_{i,B}) \sqsubseteq_{{}_{{\widetilde{S}}_{AB}}}   \sigma'_{A}\widetilde{\otimes} \sigma'_{B}.\;\;\;\;\;\;\;\;\;\;\;
\end{eqnarray}
We denote 
\begin{eqnarray}
\widetilde{ \mathfrak{F}}\{(\sigma_{i,A},\sigma_{i,B})\;\vert\; i\in I\}&:=&\{\,(\sigma'_A,\sigma'_B)\;\vert\; (\bigsqcap{}^{{}^{{\widetilde{S}}_{AB}}}_{i\in I} \sigma_{i,A}\widetilde{\otimes} \sigma_{i,B}) \sqsubseteq_{{}_{{\widetilde{S}}_{AB}}}   \sigma'_{A}\widetilde{\otimes} \sigma'_{B}\}\nonumber\\ &=&\langle \bigsqcap{}^{{}^{{\widetilde{S}}_{AB}}}_{i\in I} \sigma_{i,A}\widetilde{\otimes} \sigma_{i,B} \rangle. \;\;\;\;\;\;\;\;\;\;
\end{eqnarray}
\end{theorem}
\begin{proof}
First of all, it is clear that $\widetilde{ \mathfrak{F}}\{(\sigma_{i,A},\sigma_{i,B})\;\vert\; i\in I\}$ is a bi-filter.\\ 
Secondly, it is easy to check that $(\sigma_{k,A},\sigma_{k,B})\in \widetilde{ \mathfrak{F}}\{(\sigma_{i,A},\sigma_{i,B})\;\vert\; i\in I\}$ for any $k\in I$. Indeed, for any $K\subseteq I$, if $k\in K$ we have $(\bigsqcap{}^{{}_{{ \mathfrak{S}}_{A}}}_{l\in K}\; \sigma_{l,A}) \;\sqsubseteq_{{}_{{ \mathfrak{S}}_{A}}}\sigma_{k,A} $ and if $k\notin K$ we have $(\bigsqcap{}^{{}_{{ \mathfrak{S}}_{B}}}_{m\in I-K} \;\sigma_{m,B})\; \sqsubseteq_{{}_{{ \mathfrak{S}}_{B}}} \sigma_{k,B}$. \\
As a conclusion,  and by definition of ${ \mathfrak{F}}\{(\sigma_{i,A},\sigma_{i,B})\;\vert\; i\in I\}$ as the intersection of all bi-filters containing $(\sigma_{i,A},\sigma_{i,B})$ for any $i\in I$, we have then 
$\widetilde{ \mathfrak{F}}\{(\sigma_{i,A},\sigma_{i,B})\;\vert\; i\in I\} \supseteq  { \mathfrak{F}}\{(\sigma_{i,A},\sigma_{i,B})\;\vert\; i\in I\}$.\\
We now use Lemma \ref{sigmainFinequality} to obtain the announced result.
\end{proof}

\begin{theorem} \label{theoremsqsubseteqPAB=SAB}
If ${ \mathfrak{S}}_{A}$ \underline{or} ${ \mathfrak{S}}_{B}$ are distributive (i.e. simplex), then ${{\widetilde{S}}}{}^{fin}_{AB}$ and ${S}_{AB}$ are isomorphic posets.  
\end{theorem}
\begin{proof}
We now suppose that ${ \mathfrak{S}}_{A}$ or ${ \mathfrak{S}}_{B}$ is distributive and we intent to prove that ${ \mathfrak{F}}\{(\sigma_{i,A},\sigma_{i,B})\;\vert\; i\in I\}=\widetilde{ \mathfrak{F}}\{(\sigma_{i,A},\sigma_{i,B})\;\vert\; i\in I\}$ for any $\{(\sigma_{i,A},\sigma_{i,B})\;\vert\; i\in I\}\subseteq_{fin} { \mathfrak{S}}_{A}\times { \mathfrak{S}}_{B}$.\\
Let us prove the following fact : every bi-filter $F$ which contains $(\sigma_{k,A},\sigma_{k,B})$ for any $k\in I$ contains also $\widetilde{ \mathfrak{F}}\{(\sigma_{i,A},\sigma_{i,B})\;\vert\; i\in I\}$.  In fact, we can show that, for any bi-filter $F$ we have
\begin{eqnarray}
&&\hspace{-1cm}(\forall k\in I,\;(\sigma_{k,A},\sigma_{k,B})\in F) \Rightarrow \nonumber \\
&& (\bigsqcup{}^{{}_{{ \mathfrak{S}}_{A}}}_{K\in { \mathcal{K}}}\bigsqcap{}^{{}_{{ \mathfrak{S}}_{A}}}_{k\in K}\; \sigma_{k,A}, \bigsqcup{}^{{}_{{ \mathfrak{S}}_{A}}}_{K'\in { \mathcal{K}}'}\bigsqcap{}^{{}_{{ \mathfrak{S}}_{B}}}_{m\in I-K'} \;\sigma_{m,B})\in F, \nonumber\\
&& \forall { \mathcal{K}},{ \mathcal{K}}'\subseteq 2^I,{ \mathcal{K}}\cup { \mathcal{K}}'=2^I, { \mathcal{K}}\cap { \mathcal{K}}'=\varnothing, \{\varnothing\}\in { \mathcal{K}}', I\in { \mathcal{K}}.\;\;\;\;\;\;\;\;\;\;\;\;\;\label{intermediatedistrib}
\end{eqnarray}

The first step towards (\ref{intermediatedistrib}) is obtained by checking that $\forall { \mathcal{K}},{ \mathcal{K}}'\subseteq 2^I,{ \mathcal{K}}\cup { \mathcal{K}}'=2^I, { \mathcal{K}}\cap { \mathcal{K}}'=\varnothing, \{\varnothing\}\in { \mathcal{K}}', I\in { \mathcal{K}}$,
\begin{eqnarray}
 && (\bigsqcup{}^{{}^{{ \mathfrak{S}}}}_{K'\in { \mathcal{K}}'}\bigsqcap{}^{{}^{{ \mathfrak{S}}}}_{m\in I-K'} \;\sigma_{m}) \sqsupseteq_{{}_{ \mathfrak{S}}}
(\bigsqcap{}^{{}^{{ \mathfrak{S}}}}_{K\in { \mathcal{K}}}\bigsqcup{}^{{}^{{ \mathfrak{S}}}}_{k\in K}\; \sigma_{k})\label{firststep}
\end{eqnarray}
for any distributive ${ \mathfrak{S}}$ and any collection of elements of ${ \mathfrak{S}}$  denoted $\sigma_k$ for $k\in I$ for which these two sides of inequality exist.  To check this fact, we have to note that,  using \cite[Lemma 8 p. 50]{Balbes1975}, we have first of all
\begin{eqnarray}
(\bigsqcap{}^{{}^{{ \mathfrak{S}}}}_{K\in { \mathcal{K}}}\bigsqcup{}^{{}^{{ \mathfrak{S}}}}_{k\in K}\; \sigma_{k})= \bigsqcup{}^{{}^{{ \mathfrak{S}}}}\!\!\! \left\{ \bigsqcap{}^{{}^{{ \mathfrak{S}}}}_{K\in { \mathcal{K}}} \pi_K(A) \;\vert\; A\in \prod_{K\in { \mathcal{K}}} K \right\},
\end{eqnarray}
where $\pi_K$ denotes the projection of the component indexed by K in the cardinal product $\prod_{K\in { \mathcal{K}}} K$. Moreover,  for any $A\in \prod_{K\in { \mathcal{K}}} K$,  there exists $L\in { \mathcal{K}}'$ such that $\bigcup \{ \pi_K(A) \;\vert\; K\in { \mathcal{K}}\}\;\supseteq\; (I\smallsetminus L)$ and then $(\bigsqcap{}^{{}^{{ \mathfrak{S}}}}_{K\in { \mathcal{K}}} \pi_K(A)) \sqsubseteq_{{}_{ \mathfrak{S}}} (\bigsqcap{}^{{}^{{ \mathfrak{S}}}}_{m\in I-L} \;\sigma_{m}) \sqsubseteq_{{}_{ \mathfrak{S}}} (\bigsqcup{}^{{}^{{ \mathfrak{S}}}}_{K'\in { \mathcal{K}}'}\bigsqcap{}^{{}^{{ \mathfrak{S}}}}_{m\in I-K'} \;\sigma_{m})$. As a result, we obtain the property (\ref{firststep}).\\

The second step towards (\ref{intermediatedistrib}) consists in showing that 
\begin{eqnarray}
(\forall k\in I,\;(\sigma_{k,A},\sigma_{k,B})\in F) & \Rightarrow & (\bigsqcup{}^{{}_{{ \mathfrak{S}}_{A}}}_{K\in { \mathcal{K}}}\bigsqcap{}^{{}_{{ \mathfrak{S}}_{A}}}_{k\in K}\; \sigma_{k,A},  \bigsqcap{}^{{}_{{ \mathfrak{S}}_{B}}}_{K\in { \mathcal{K}}}\bigsqcup{}^{{}_{{ \mathfrak{S}}_{B}}}_{k\in K}\; \sigma_{k,B})\;\in F\;\;\;\;\;\;\;\;\;\;\;\;\;
\end{eqnarray}
for any ${ \mathcal{K}}\subseteq 2^I$. This intermediary result is obtained by induction on the complexity of the polynomial $(\bigsqcup{}^{{}_{{ \mathfrak{S}}_{A}}}_{K\in { \mathcal{K}}}\bigsqcap{}^{{}_{{ \mathfrak{S}}_{A}}}_{k\in K}\; \sigma_{k,A})$ by using the following elementary result
\begin{eqnarray}
&&\forall \sigma_A,\sigma'_A\in { \mathfrak{S}}_A,\sigma_B,\sigma'_B\in { \mathfrak{S}}_B,\;\;\;\; \left( (\sigma_A,\sigma_B),(\sigma'_A,\sigma'_B)\in F \right) \Rightarrow \nonumber\\
&&\hspace{2cm} \left\{
\begin{array}{l}
(\sigma_A \sqcup_{{}_{{ \mathfrak{S}}_A}} \sigma'_A,\sigma_B \sqcap_{{}_{{ \mathfrak{S}}_B}} \sigma'_B)\in F\nonumber\\
(\sigma_A \sqcap_{{}_{{ \mathfrak{S}}_A}} \sigma'_A,\sigma_B \sqcup_{{}_{{ \mathfrak{S}}_B}} \sigma'_B)\in F
\end{array}\right.
\end{eqnarray}
trivially deduced using the bi-filter character of $F$, i.e. properties (\ref{defbifilter1})(\ref{defbifilter2})(\ref{defbifilter3}). \\

As a final conclusion,  using the explicit definition of ${ \mathfrak{F}}\{(\sigma_{i,A},\sigma_{i,B})\;\vert\; i\in I\}$ as the intersection of all bi-ideals containing $(\sigma_{k,A},\sigma_{k,B})$ for any $k\in I$,  we obtain $\widetilde{ \mathfrak{F}}\{(\sigma_{i,A},\sigma_{i,B})\;\vert\; i\in I\}={ \mathfrak{F}}\{(\sigma_{i,A},\sigma_{i,B})\;\vert\; i\in I\}$.\\

${{\widetilde{S}}}{}^{fin}_{AB}$ and ${S}_{AB}$ are then isomorphic posets.
\end{proof}

\begin{remark}\label{remarkdistributivity}
We note that the distributivity property is a key condition to obtain previous isomorphism between  ${{\widetilde{S}}}{}^{fin}_{AB}$ and ${S}_{AB}$. Indeed, let us consider that ${ \mathfrak{S}}_A$ and ${ \mathfrak{S}}_B$ are both defined as the Inf semi-lattice associated to the following Hasse diagram ${ \mathcal{S}}^{(1)}_3$. 
According to (\ref{developmentetildeordersimplify}), we have $(\bot_{{}_{{ \mathfrak{S}}_A}}, \bot_{{}_{{ \mathfrak{S}}_B}})\in \widetilde{ \mathfrak{F}}\{(\sigma_{1},\sigma_{1}),(\sigma_{2},\sigma_{2}),(\sigma_{3},\sigma_{3})\}.$ However, we have obviously $(\bot_{{}_{{ \mathfrak{S}}_A}}, \bot_{{}_{{ \mathfrak{S}}_B}})\notin { \mathfrak{F}}\{(\sigma_{1},\sigma_{1}),(\sigma_{2},\sigma_{2}),(\sigma_{3},\sigma_{3})\}.$
\end{remark}

\subsection{Properties of the minimal tensor product}\label{subsectionpropertiesminimal}

Let us now consider that ${ \mathfrak{S}}_{A}$ and ${ \mathfrak{S}}_{B}$ have a description in terms of pure states. We intend to prove that ${{\widetilde{S}}}{}_{AB}$ inherits a description in terms of pure states.

\begin{theorem}\label{theorempuretilde} 
\begin{eqnarray}
{{\widetilde{S}}}^{\;{}^{pure}}_{AB} & = & \{\, \sigma_A \widetilde{\otimes} \sigma_B\;\vert\; \sigma_A \in { \mathfrak{S}}^{\;{}^{pure}}_{A}, \sigma_B \in { \mathfrak{S}}^{\;{}^{pure}}_{B}\,\} = Max({{\widetilde{S}}}_{AB})
\end{eqnarray}
\end{theorem}
\begin{proof}
First of all, it is a trivial fact that the completely meet-irreducible elements of  ${{\widetilde{S}}}_{AB}$ are necessarily pure tensors of  ${{\widetilde{S}}}_{AB}$, i.e.  elements of the form $\sigma_A \widetilde{\otimes} \sigma_B$. \\
Let us then consider $\sigma_A \widetilde{\otimes} \sigma_B$ a completely meet-irreducible element of ${{\widetilde{S}}}_{AB}$ and let us assume that $\sigma_A  = \bigsqcap{}^{{}^{{ \mathfrak{S}}_{A}}}_{i\in I} \sigma_{i,A}$ for $ \sigma_{i,A}\in { \mathfrak{S}}_{A}$ for any $i\in I$. We have then $(\sigma_A \widetilde{\otimes} \sigma_B) = ((\bigsqcap{}^{{}^{{ \mathfrak{S}}_{A}}}_{i\in I} \sigma_{i,A})\widetilde{\otimes}  \sigma_B)=\bigsqcap{}^{{}^{{{\widetilde{S}}}_{AB}}}_{i\in I} ( \sigma_{i,A}\widetilde{\otimes}  \sigma_B)$.  On another part, $\sigma_A \widetilde{\otimes} \sigma_B$ being completely meet-irreducible in ${{\widetilde{S}}}_{AB}$, there exists $k\in I$ such that $\sigma_A \widetilde{\otimes} \sigma_B=\sigma_{k,A} \widetilde{\otimes} \sigma_B$, i.e, $\sigma_A=\sigma_{k,A}$. As a conclusion, $\sigma_A$ is completely meet-irreducible. In the same way, $\sigma_B$ is completely meet-irreducible.  As a first result, pure states of ${{\widetilde{S}}}_{AB}$ are necessarily of the form $\sigma_A \widetilde{\otimes} \sigma_B$ with $\sigma_A \in { \mathfrak{S}}^{\;{}^{pure}}_{A}, \sigma_B \in { \mathfrak{S}}^{\;{}^{pure}}_{B}$.\\
Conversely, let us consider $\sigma_A$ a pure state of ${ \mathfrak{S}}_{A}$ and $\sigma_B$ a pure state of ${ \mathfrak{S}}_{B}$, and let us suppose that $ (\bigsqcap{}^{{}^{{{\widetilde{S}}}_{AB}}}_{i\in I} \sigma_{i,A}\widetilde{\otimes}  \sigma_{i,B})  =  (\sigma_A \widetilde{\otimes} \sigma_B)$ with $\sigma_{i,A}\in { \mathfrak{S}}_{A}$ and $\sigma_{i,B}\in { \mathfrak{S}}_{B}$ for any $i\in I$. We now exploit the two conditions $(\bigsqcap{}^{{}_{{ \mathfrak{S}}_{A}}}_{k\in I}\; \sigma_{k,A}) = \sigma_{A}$ and $(\bigsqcap{}^{{}_{{ \mathfrak{S}}_{B}}}_{m\in I} \;\sigma_{m,B}) = \sigma_{B}$ derived from the expansion (\ref{developmentetildeordersimplify}).  From $\sigma_A\in Max({ \mathfrak{S}}_{A})$ and $\sigma_B\in Max({ \mathfrak{S}}_{B})$, we deduce that $\sigma_{i,A} = \sigma_A$ and $\sigma_{j,B} = \sigma_B$ for any $i,j\in I$. As a second result, we have then obtained that the state $(\sigma_A \widetilde{\otimes} \sigma_B)$, with $\sigma_A$ a pure state of ${ \mathfrak{S}}_{A}$ and $\sigma_B$ a pure state of ${ \mathfrak{S}}_{B}$, is completely meet-irreducible.\\
From the expansion (\ref{developmentetildeordersimplify}), we deduce also immediately that $(\sigma_A \widetilde{\otimes} \sigma_B)\in Max({{\widetilde{S}}}_{AB})$ as long as $\sigma_A\in Max({ \mathfrak{S}}_{A})$ and $\sigma_B\in Max({ \mathfrak{S}}_{B})$.
\end{proof}

\begin{theorem} \label{theoremA4Ptilde}
\begin{eqnarray}
&&\forall \sigma \in {{\widetilde{S}}}_{AB}, \;\; \sigma= \bigsqcap{}^{{}^{{{\widetilde{S}}}_{AB}}}  \underline{\sigma}_{{}_{ {{\widetilde{S}}}_{AB}}},\;\;\textit{\rm where}\;\;
\underline{\sigma}_{{}_{ {{\widetilde{S}}}_{AB}}}=
({{\widetilde{S}}}{}_{AB}^{\,{}^{pure}} \cap (\uparrow^{{}^{{{\widetilde{S}}}_{AB}}}\!\!\!\! \sigma) ).
\end{eqnarray}
\end{theorem}
\begin{proof}
Let us fix $\sigma \in {{\widetilde{S}}}_{AB}$.\\ 
We note that $\sigma \sqsubseteq_{{}_{{{\widetilde{S}}}_{AB}}} \sigma'$ for any $\sigma'\in ({{\widetilde{S}}}{}_{AB}^{\,{}^{pure}} \cap (\uparrow^{{}^{{{\widetilde{S}}}_{AB}}}\!\!\!\! \sigma)) $ and then $\sigma \sqsubseteq_{{}_{{{\widetilde{S}}}_{AB}}} \bigsqcap{}^{{}^{{{\widetilde{S}}}_{AB}}}  \underline{\sigma}_{{}_{ {{\widetilde{S}}}_{AB}}}$. \\
Secondly,  denoting  
$\sigma:=(\bigsqcap{}^{{}^{{{\widetilde{S}}}_{AB}}}_{i\in I} \sigma_{i,A}\widetilde{\otimes}  \sigma_{i,B})$, we note immediately that, for any 
$\sigma_{A} \in { \mathfrak{S}}_{A}^{pure}$ and $\sigma_{B} \in { \mathfrak{S}}_{B}^{pure}$, if 
$\sigma_{A}\sqsupseteq_{{}_{{ \mathfrak{S}}_{A}}} \sigma_{i,A}$ and $ \sigma_{B}\sqsupseteq_{{}_{{ \mathfrak{S}}_{B}}} \sigma_{i,B}$, then $(\sigma_{A}\widetilde{\otimes}  \sigma_{B}) \sqsupseteq_{{}_{{ \mathfrak{S}}_{AB}}} \sigma$, i.e. $(\sigma_{A}\widetilde{\otimes}  \sigma_{B}) \in \underline{\sigma}_{{}_{ {{\widetilde{S}}}_{AB}}}$. As a consequence, we have
\begin{eqnarray}
(\bigsqcap{}^{{}^{{{\widetilde{S}}}_{AB}}}_{i\in I}\bigsqcap{}^{{}^{{{\widetilde{S}}}_{AB}}}_{\sigma_{A} \in { \mathfrak{S}}_{A}^{pure}\;\vert\; \sigma_{A}\sqsupseteq_{{}_{{ \mathfrak{S}}_{A}}} \sigma_{i,A}}\bigsqcap{}^{{}^{{{\widetilde{S}}}_{AB}}}_{\sigma_{B} \in { \mathfrak{S}}_{B}^{pure}\;\vert\; \sigma_{B}\sqsupseteq_{{}_{{ \mathfrak{S}}_{B}}} \sigma_{i,B}} \sigma_{A}\widetilde{\otimes}  \sigma_{B}) \sqsupseteq_{{}_{{{\widetilde{S}}}_{AB}}} \bigsqcap{}^{{}^{{{\widetilde{S}}}_{AB}}}  \underline{\sigma}_{{}_{ {{\widetilde{S}}}_{AB}}}.\;\;\;\;\;\;\;
\end{eqnarray}
Endly, using Theorem \ref{bifilterPtilde},we have
\begin{eqnarray}
\sigma &=&\bigsqcap{}^{{}^{{{\widetilde{S}}}_{AB}}}_{i\in I} \sigma_{i,A}\widetilde{\otimes}  \sigma_{i,B} =  
\bigsqcap{}^{{}^{{{\widetilde{S}}}_{AB}}}_{i\in I} (\bigsqcap{}^{{}^{{{\widetilde{S}}}_{AB}}}_{\sigma_{A} \in { \mathfrak{S}}_{A}^{pure}\;\vert\; \sigma_{A}\sqsupseteq_{{}_{{ \mathfrak{S}}_{A}}} \sigma_{i,A}}\sigma_{A})\widetilde{\otimes}  (\bigsqcap{}^{{}^{{{\widetilde{S}}}_{AB}}}_{\sigma_{B} \in { \mathfrak{S}}_{B}^{pure}\;\vert\; \sigma_{B}\sqsupseteq_{{}_{{ \mathfrak{S}}_{B}}} \sigma_{i,B}}\sigma_{B})\nonumber\\
&= & 
\bigsqcap{}^{{}^{{{\widetilde{S}}}_{AB}}}_{i\in I}\bigsqcap{}^{{}^{{{\widetilde{S}}}_{AB}}}_{\sigma_{A} \in { \mathfrak{S}}_{A}^{pure}\;\vert\; \sigma_{A}\sqsupseteq_{{}_{{ \mathfrak{S}}_{A}}} \sigma_{i,A}}\bigsqcap{}^{{}^{{{\widetilde{S}}}_{AB}}}_{\sigma_{B} \in { \mathfrak{S}}_{B}^{pure}\;\vert\; \sigma_{B}\sqsupseteq_{{}_{{ \mathfrak{S}}_{B}}} \sigma_{i,B}} \sigma_{A}\widetilde{\otimes}  \sigma_{B}.
\end{eqnarray}
As a final conclusion, we obtain
\begin{eqnarray}
\sigma = (\bigsqcap{}^{{}^{{{\widetilde{S}}}_{AB}}}_{i\in I}\bigsqcap{}^{{}^{{{\widetilde{S}}}_{AB}}}_{\sigma_{A} \in { \mathfrak{S}}_{A}^{pure}\;\vert\; \sigma_{A}\sqsupseteq_{{}_{{ \mathfrak{S}}_{A}}} \sigma_{i,A}}\bigsqcap{}^{{}^{{{\widetilde{S}}}_{AB}}}_{\sigma_{B} \in { \mathfrak{S}}_{B}^{pure}\;\vert\; \sigma_{B}\sqsupseteq_{{}_{{ \mathfrak{S}}_{B}}} \sigma_{i,B}} \sigma_{A}\widetilde{\otimes}  \sigma_{B}) = \bigsqcap{}^{{}^{{{\widetilde{S}}}_{AB}}}  \underline{\sigma}_{{}_{ {{\widetilde{S}}}_{AB}}}.\;\;\;\;\;\;\;\;\;\;\;\;\;\;
\end{eqnarray}
\end{proof}

\begin{theorem}\label{formulacupStilde}
Let $\widetilde{\sigma}_{AB}$ and $\widetilde{\sigma}'_{AB}$ be two elements of ${{\widetilde{S}}}_{AB}$ having a common upper-bound. Then the supremum of $\{\widetilde{\sigma}_{AB},\widetilde{\sigma}'_{AB}\}$ exists in ${{\widetilde{S}}}_{AB}$ and its expression is given by
\begin{eqnarray}
\widetilde{\sigma}_{AB} \sqcup_{{}_{{{\widetilde{S}}}_{AB}}} \widetilde{\sigma}'_{AB} = \bigsqcap{}^{{}^{{{\widetilde{S}}}_{AB}}}_{\widetilde{\sigma} \in (\underline{\widetilde{\sigma}_{AB}}_{{}_{ {{\widetilde{S}}}_{AB}}}\!\!\!\!\!\!\!\!\!\cap\; \underline{\widetilde{\sigma}'_{AB}}_{{}_{ {{\widetilde{S}}}_{AB}}}\!\!\!\!\!\!\!\!\!)} \; \widetilde{\sigma}\label{formulacupPtilde2}
\end{eqnarray}
\end{theorem}
\begin{proof} As long as $\widetilde{\sigma}_{AB}$ and $\widetilde{\sigma}'_{AB}$ have a common upper-bound, $\underline{\widetilde{\sigma}_{AB}}\cap \underline{\widetilde{\sigma}'_{AB}}$ is not empty. Secondly, it is clear that $\widetilde{\sigma}_{AB}= (\bigsqcap{}^{{}^{{{\widetilde{S}}}_{AB}}}_{\widetilde{\sigma} \in \underline{\widetilde{\sigma}_{AB}}} \; \widetilde{\sigma})\; \sqsubseteq_{{}_{{{\widetilde{S}}}_{AB}}} \bigsqcap{}^{{}^{{{\widetilde{S}}}_{AB}}}_{\widetilde{\sigma} \in \underline{\widetilde{\sigma}_{AB}}\cap \underline{\widetilde{\sigma}'_{AB}}} \; \widetilde{\sigma}$ and $\widetilde{\sigma}'_{AB}= (\bigsqcap{}^{{}^{{{\widetilde{S}}}_{AB}}}_{\widetilde{\sigma} \in \underline{\widetilde{\sigma}'_{AB}}} \; \widetilde{\sigma})\; \sqsubseteq_{{}_{{{\widetilde{S}}}_{AB}}} \bigsqcap{}^{{}^{{{\widetilde{S}}}_{AB}}}_{\widetilde{\sigma} \in \underline{\widetilde{\sigma}_{AB}}\cap \underline{\widetilde{\sigma}'_{AB}}} \; \widetilde{\sigma}$. Then, if we suppose there exists $\widetilde{\sigma}''_{AB}$ such that $\widetilde{\sigma}_{AB}, \widetilde{\sigma}'_{AB} \sqsubseteq_{{}_{{{\widetilde{S}}}_{AB}}} \widetilde{\sigma}''_{AB}$ we can use Theorem \ref{theoremA4Ptilde} to obtain the decomposition $\widetilde{\sigma}''_{AB}=(\bigsqcap{}^{{}^{{{\widetilde{S}}}_{AB}}}_{\widetilde{\sigma} \in \underline{\widetilde{\sigma}''_{AB}}} \; \widetilde{\sigma})$ with necessarily $\forall \widetilde{\sigma}\in \underline{\widetilde{\sigma}''_{AB}},$ $\widetilde{\sigma}_{AB} \sqsubseteq_{{}_{{{\widetilde{S}}}_{AB}}}\widetilde{\sigma}$ and $\widetilde{\sigma}'_{AB} \sqsubseteq_{{}_{{{\widetilde{S}}}_{AB}}}\widetilde{\sigma}$, i.e. $\widetilde{\sigma}\in  \underline{\widetilde{\sigma}_{AB}}\cap \underline{\widetilde{\sigma}'_{AB}}$, and then $(\bigsqcap{}^{{}^{{{\widetilde{S}}}_{AB}}}_{\widetilde{\sigma} \in \underline{\widetilde{\sigma}_{AB}}\cap \underline{\widetilde{\sigma}'_{AB}}} \; \widetilde{\sigma}) \sqsubseteq_{{}_{{{\widetilde{S}}}_{AB}}} \widetilde{\sigma}''_{AB}$.  \end{proof}

\begin{theorem}\label{SASBditribandcup}
If ${ \mathfrak{S}}_A$ \underline{and} ${ \mathfrak{S}}_B$ are distributive (i.e. simplex), then ${{\widetilde{S}}}_{AB}$ is also distributive (i.e. a simplex).  \\
Note, using Theorem \ref{theoremsqsubseteqPAB=SAB},  that, in this situation, we have also ${{\widetilde{S}}}{}^{fin}_{AB}={S}_{AB}$.\\
In that case, the explicit expression for the supremum of two elements in ${{\widetilde{S}}}{}^{fin}_{AB}$ is given by
\begin{eqnarray}
&&\hspace{-1cm}(\bigsqcap{}^{{}^{{{\widetilde{S}}}_{AB}}}_{i\in I} \sigma_{i,A}\widetilde{\otimes} \sigma_{i,B}) \sqcup{}_{{}_{{{\widetilde{S}}}_{AB}}} (\bigsqcap{}^{{}^{{{\widetilde{S}}}_{AB}}}_{j\in J} \sigma'_{j,A}\widetilde{\otimes} \sigma'_{j,B}) = \nonumber\\
&&\hspace{1cm} =\bigsqcap{}^{{}^{{{\widetilde{S}}}_{AB}}}_{i\in I,\;j\in J}\; (\sigma_{i,A}\sqcup_{{}_{{ \mathfrak{S}}_A}}\sigma'_{j,A}) \widetilde{\otimes} (\sigma_{i,B} \sqcup_{{}_{{ \mathfrak{S}}_B}} \sigma'_{j,B}). \;\;\;\;\;\;\;\;\;\label{formulacupSAB}
\end{eqnarray}
\end{theorem}
\begin{proof}
Using Theorem \ref{theoremsqsubseteqPAB=SAB}, we note that, as soon as ${ \mathfrak{S}}_A$ or ${ \mathfrak{S}}_B$ is distributive, we have ${{\widetilde{S}}}_{AB}={S}_{AB}$ as Inf semi-lattices.  We are then reduced to prove the distributivity of ${S}_{AB}$.  This theorem is then a direct consequence of \cite[Theorem 3]{Fraser1978}.
\end{proof}

\begin{theorem}\label{Stildereducible}
Let us consider $\sigma_{1,A}$ and $\sigma_{2,A}$ two distinct elements of ${ \mathfrak{S}}_{A}$ ,  and $\sigma_{1,B}$ and $\sigma_{2,B}$ two distinct elements of ${ \mathfrak{S}}_{B}$. 
We have then
\begin{eqnarray}
&&\hspace{-2cm} 
(\, \Phi \sqsupset_{{}_{{\widetilde{S}}_{AB}}} (\sigma_{1,A}\widetilde{\otimes} \sigma_{1,B} \sqcap_{{}_{{\widetilde{S}}_{AB}}} \sigma_{2,A}\widetilde{\otimes} \sigma_{2,B}) \;\;\;\textit{\rm and}\;\;\;
\Phi\in Max({\widetilde{S}}_{AB})\,) \Rightarrow \nonumber\\
&& 
  \Phi\in \{\,\sigma_{1,A}\widetilde{\otimes} \sigma_{1,B} \,,\, \sigma_{2,A}\widetilde{\otimes} \sigma_{2,B}\,\}.\;\;\;\;\;\;\;\;\;\;\;\;\;\;\;\;\;\;
\end{eqnarray} 
\end{theorem}
\begin{proof}
Direct consequence of the expansion (\ref{developmentetildeordersimplify}) with Theorem \ref{theorempuretilde}.
\end{proof}

\subsection{From maximal to regular tensor product}\label{subsectionremarks}

In the present subsection we will assume the following supplementary condition satisfied by an element $\Phi\in \widecheck{S}_{AB}$ :
\begin{eqnarray}
{\Phi}( \overline{{ \mathfrak{Y}}_{{ \mathfrak{E}}_{A}}},\overline{{ \mathfrak{Y}}_{{ \mathfrak{E}}_{B}}}) = \textit{\bf N}. \label{furtherquotientNN}
\end{eqnarray}
We note that an element $\Phi\in \widetilde{S}_{AB}$ satisfies necessarily this condition.

\begin{lemma}
For any $\Phi\in \widecheck{S}_{AB}$, we have
\begin{eqnarray}
\forall { \mathfrak{l}}_B\in { \mathfrak{E}}_{B},&& \Phi( \overline{{ \mathfrak{Y}}_{{ \mathfrak{E}}_{A}}},{ \mathfrak{l}}_B) =\textit{\bf N}\label{phinl}\\
\forall { \mathfrak{l}}_A\in { \mathfrak{E}}_{A},&& \Phi({ \mathfrak{l}}_A,\overline{{ \mathfrak{Y}}_{{ \mathfrak{E}}_{B}}}) =\textit{\bf N}.\label{philn}
\end{eqnarray}
\end{lemma}
\begin{proof}
Using (\ref{furtherquotient1}) (\ref{furtherquotient2}) and (\ref{furtherquotientYY}), we deduce for any $\Phi\in \widecheck{S}_{AB}$ the following equations
\begin{eqnarray}
&&\Phi( \overline{{ \mathfrak{Y}}_{{ \mathfrak{E}}_{A}}},{ \mathfrak{Y}}_{{ \mathfrak{E}}_{B}}) = \textit{\bf N}\label{phiny}\\
&&\Phi( { \mathfrak{Y}}_{{ \mathfrak{E}}_{A}},\overline{{ \mathfrak{Y}}_{{ \mathfrak{E}}_{B}}}) = \textit{\bf N}\label{phiyn}
\end{eqnarray}

Using (\ref{phiny})(\ref{phiyn}) (\ref{furtherquotientYY})(\ref{furtherquotientNN}) and (\ref{bilinear1})(\ref{bilinear2}), we deduce for any $\Phi\in \widecheck{S}_{AB}$ the following equations
\begin{eqnarray}
&&\Phi( \overline{{ \mathfrak{Y}}_{{ \mathfrak{E}}_{A}}},\bot_{{ \mathfrak{E}}_{B}}) = \textit{\bf N}\label{phinbot}\\
&&\Phi( \bot_{{ \mathfrak{E}}_{B}},\overline{{ \mathfrak{Y}}_{{ \mathfrak{E}}_{A}}}) = \textit{\bf N}\label{phibotn}\\
&&\Phi( { \mathfrak{Y}}_{{ \mathfrak{E}}_{A}},\bot_{{ \mathfrak{E}}_{B}}) = \bot \label{phiybot}\\
&&\Phi( \bot_{{ \mathfrak{E}}_{B}},{ \mathfrak{Y}}_{{ \mathfrak{E}}_{A}}) = \bot. \label{phiboty}
\end{eqnarray}

We have also, using (\ref{phinbot}) (\ref{phiybot}) and (\ref{bilinear1}), for any $\Phi\in \widecheck{S}_{AB}$ the following equation
\begin{eqnarray}
&&\Phi( \bot_{{ \mathfrak{E}}_{A}},\bot_{{ \mathfrak{E}}_{B}}) = \bot.\label{phibotbot}
\end{eqnarray}

From (\ref{phinbot}) and (\ref{bilinear2}), we deduce that, for any ${ \mathfrak{l}}_B\in { \mathfrak{E}}_{B}$ we have $\Phi( \overline{{ \mathfrak{Y}}_{{ \mathfrak{E}}_{A}}},{ \mathfrak{l}}_B) \wedge  \Phi( \overline{{ \mathfrak{Y}}_{{ \mathfrak{E}}_{A}}},\overline{{ \mathfrak{l}}_B})=\textit{\bf N}$ and then $\Phi( \overline{{ \mathfrak{Y}}_{{ \mathfrak{E}}_{A}}},{ \mathfrak{l}}_B) =\textit{\bf N}$. Here we have used the obvious property ${ \mathfrak{l}}_B \sqcap_{{}_{{{ \mathfrak{E}}_{B}}}} \overline{{ \mathfrak{l}}_B}  = \bot_{{ \mathfrak{E}}_{B}}$ satisfied by any ${ \mathfrak{l}}_B\in { \mathfrak{E}}_{B}$. In the same way, using (\ref{phibotn}) and (\ref{bilinear1}),  we obtain the symmetric property.  As a result of our investigations of the consequences of (\ref{phinbot}) and (\ref{phibotn}), we have obtained for any $\Phi\in \widecheck{S}_{AB}$ the equations (\ref{phinl}) and (\ref{philn}).
\end{proof}

Let us now investigate the consequences of (\ref{phinl}) and (\ref{philn}). 

\begin{lemma}\label{threecasesScheck}
Let us consider any $\Phi$ in $\widecheck{S}_{AB}$ and any ${ \mathfrak{l}}_A\in { \mathfrak{E}}_{A}$. We are necessarily in one of the following three cases
\begin{enumerate}
\item 
\begin{eqnarray}
\left\{
\begin{array}{l}
 \Phi({ \mathfrak{l}}_A,{{ \mathfrak{Y}}_{{ \mathfrak{E}}_{B}}})=\textit{\bf N},\;\;\;\;\;\;\;\; \Phi(\overline{{ \mathfrak{l}}_A},{{ \mathfrak{Y}}_{{ \mathfrak{E}}_{B}}})=\textit{\bf Y},\\
 \forall { \mathfrak{l}}_B\in { \mathfrak{E}}_{B},\;\; \Phi({ \mathfrak{l}}_A,{ \mathfrak{l}}_B)=\textit{\bf N},\\
 \forall { \mathfrak{l}}_B,{ \mathfrak{l}}'_B\in { \mathfrak{E}}_{B}\;\vert\;{ \mathfrak{l}}_B \sqcap_{{}_{{{ \mathfrak{E}}_{B}}}} { \mathfrak{l}}'_B=\bot_{{ \mathfrak{E}}_{B}}, \;\; (\Phi(\overline{{ \mathfrak{l}}_A},{ \mathfrak{l}}_B)\,,\, \Phi(\overline{{ \mathfrak{l}}_A},{ \mathfrak{l}}'_B))\notin \{(\textit{\bf N},\textit{\bf N}), (\textit{\bf Y},\textit{\bf Y})\}.
\end{array} \right.
\label{condition1}
\end{eqnarray}
\item 
\begin{eqnarray}
\left\{
\begin{array}{l}
\Phi({ \mathfrak{l}}_A,{{ \mathfrak{Y}}_{{ \mathfrak{E}}_{B}}})=\textit{\bf Y},\;\;\;\;\;\;\;\; \Phi(\overline{{ \mathfrak{l}}_A},{{ \mathfrak{Y}}_{{ \mathfrak{E}}_{B}}})=\textit{\bf N},\\
\forall { \mathfrak{l}}_B\in { \mathfrak{E}}_{B},\;\; \Phi(\overline{{ \mathfrak{l}}_A},{ \mathfrak{l}}_B)=\textit{\bf N}, \\
\forall { \mathfrak{l}}_B,{ \mathfrak{l}}'_B\in { \mathfrak{E}}_{B}\;\vert\;{ \mathfrak{l}}_B \sqcap_{{}_{{{ \mathfrak{E}}_{B}}}} { \mathfrak{l}}'_B=\bot_{{ \mathfrak{E}}_{B}}, \;\; (\Phi({ \mathfrak{l}}_A,{ \mathfrak{l}}_B)\,,\, \Phi({ \mathfrak{l}}_A,{ \mathfrak{l}}'_B))\notin \{(\textit{\bf N},\textit{\bf N}), (\textit{\bf Y},\textit{\bf Y})\}.
\end{array} \right.
\label{condition2}
\end{eqnarray}
\item
\begin{eqnarray}
\left\{
\begin{array}{l}
\Phi({ \mathfrak{l}}_A,{{ \mathfrak{Y}}_{{ \mathfrak{E}}_{B}}})=\bot, \;\;\;\;\;\;\;\; \Phi(\overline{{ \mathfrak{l}}_A},{{ \mathfrak{Y}}_{{ \mathfrak{E}}_{B}}})=\bot,\\
\forall { \mathfrak{l}}_B,{ \mathfrak{l}}'_B\in { \mathfrak{E}}_{B}\;\vert\;{ \mathfrak{l}}_B \sqcap_{{}_{{{ \mathfrak{E}}_{B}}}} { \mathfrak{l}}'_B=\bot_{{ \mathfrak{E}}_{B}},  \;\; (\Phi({{ \mathfrak{l}}_A},{ \mathfrak{l}}_B)\,,\, \Phi({{ \mathfrak{l}}_A},{ \mathfrak{l}}'_B))\notin \{(\textit{\bf N},\textit{\bf N}), (\textit{\bf Y},\textit{\bf Y})\},\\
\forall { \mathfrak{l}}_B,{ \mathfrak{l}}'_B\in { \mathfrak{E}}_{B}\;\vert\;{ \mathfrak{l}}_B \sqcap_{{}_{{{ \mathfrak{E}}_{B}}}} { \mathfrak{l}}'_B=\bot_{{ \mathfrak{E}}_{B}}, \;\; (\Phi(\overline{{ \mathfrak{l}}_A},{ \mathfrak{l}}_B)\,,\, \Phi(\overline{{ \mathfrak{l}}_A},{ \mathfrak{l}}'_B))\notin \{(\textit{\bf N},\textit{\bf N}), (\textit{\bf Y},\textit{\bf Y})\}.
\end{array} \right.
\label{condition3}
\end{eqnarray}
\end{enumerate}
\end{lemma}
\begin{proof}
The distinction between the three cases is directly inherited from (\ref{furtherquotient1}).\\

Let us consider the first case : $\Phi({ \mathfrak{l}}_A,{{ \mathfrak{Y}}_{{ \mathfrak{E}}_{B}}})=\textit{\bf N}$ and $\Phi(\overline{{ \mathfrak{l}}_A},{{ \mathfrak{Y}}_{{ \mathfrak{E}}_{B}}})=\textit{\bf Y}$. Using $\Phi({ \mathfrak{l}}_A,{{ \mathfrak{Y}}_{{ \mathfrak{E}}_{B}}})=\textit{\bf N}$ and (\ref{philn}) and (\ref{bilinear2}), we obtain $\Phi({ \mathfrak{l}}_A,\bot_{{ \mathfrak{E}}_{B}})=\textit{\bf N}$ and then, 
$\forall { \mathfrak{l}}_B\in { \mathfrak{E}}_{B}, \Phi({ \mathfrak{l}}_A,{ \mathfrak{l}}_B)=\textit{\bf N}$. Secondly, using $\Phi(\overline{{ \mathfrak{l}}_A},{{ \mathfrak{Y}}_{{ \mathfrak{E}}_{B}}})=\textit{\bf Y}$ and (\ref{philn}) and (\ref{bilinear2}), we obtain $\Phi(\overline{{ \mathfrak{l}}_A},\bot_{{ \mathfrak{E}}_{B}})=\bot$, which means that for any ${ \mathfrak{l}}_B,{ \mathfrak{l}}'_B\in { \mathfrak{E}}_{B}$ such that ${ \mathfrak{l}}_B \sqcap_{{}_{{{ \mathfrak{E}}_{B}}}} { \mathfrak{l}}'_B=\bot_{{ \mathfrak{E}}_{B}}$ we have  $\Phi(\overline{{ \mathfrak{l}}_A},{ \mathfrak{l}}_B)\wedge \Phi(\overline{{ \mathfrak{l}}_A},{ \mathfrak{l}}'_B) = \bot$.\\

The second case ($\Phi({ \mathfrak{l}}_A,{{ \mathfrak{Y}}_{{ \mathfrak{E}}_{B}}})=\textit{\bf Y}$ and $\Phi(\overline{{ \mathfrak{l}}_A},{{ \mathfrak{Y}}_{{ \mathfrak{E}}_{B}}})=\textit{\bf N}$) is treated exactly in the same way as the first case.\\

Let us conclude with the third case : $\Phi({ \mathfrak{l}}_A,{{ \mathfrak{Y}}_{{ \mathfrak{E}}_{B}}})=\bot$ and $\Phi(\overline{{ \mathfrak{l}}_A},{{ \mathfrak{Y}}_{{ \mathfrak{E}}_{B}}})=\bot$. Using (\ref{philn}) and (\ref{bilinear2}), we obtain $\Phi({ \mathfrak{l}}_A,\bot_{{ \mathfrak{E}}_{B}})=\bot$ and $\Phi(\overline{{ \mathfrak{l}}_A},\bot_{{ \mathfrak{E}}_{B}})=\bot$, which means that, for any ${ \mathfrak{l}}_B,{ \mathfrak{l}}'_B\in { \mathfrak{E}}_{B}$ such that ${ \mathfrak{l}}_B \sqcap_{{}_{{{ \mathfrak{E}}_{B}}}} { \mathfrak{l}}'_B=\bot_{{ \mathfrak{E}}_{B}}$ we have  $\Phi(\overline{{ \mathfrak{l}}_A},{ \mathfrak{l}}_B)\wedge \Phi(\overline{{ \mathfrak{l}}_A},{ \mathfrak{l}}'_B) = \bot$ and $\Phi({ \mathfrak{l}}_A,{ \mathfrak{l}}_B)\wedge \Phi({ \mathfrak{l}}_A,{ \mathfrak{l}}'_B) = \bot$. This concludes the proof.
\end{proof}

If we restrict ourselves to the elements of ${\widetilde{S}}_{AB}$, the conditions are in fact more severe.

\begin{lemma}\label{restrictedconditions1}
Let us now fix $\Phi\in \widetilde{S}_{AB}$ and ${ \mathfrak{l}}_A\in { \mathfrak{E}}_{A}$, and let us suppose that $\Phi({ \mathfrak{l}}_A,{{ \mathfrak{Y}}_{{ \mathfrak{E}}_{B}}})=\textit{\bf Y}$. Then, for any ${ \mathfrak{l}}_B\in { \mathfrak{E}}_{B}$, we have 
\begin{eqnarray}
(\Phi({ \mathfrak{l}}_A,{ \mathfrak{l}}_B)\,,\, \Phi({ \mathfrak{l}}_A,\overline{{ \mathfrak{l}}_B}))\in \{(\textit{\bf Y},\textit{\bf N}), (\textit{\bf N},\textit{\bf Y}),(\bot,\bot)\}.
\end{eqnarray}
\end{lemma}

\begin{proof}
Let us consider that $\Phi=\bigsqcap{}^{{}^{\widecheck{S}_{AB}}}_{i\in I}\, \iota^{\widecheck{S}_{AB}}(\sigma_{i,A},\sigma_{i,B})=\bigsqcap{}^{{}^{{ \widetilde{S}}_{AB}}}_{i\in I} \sigma_{i,A}\widetilde{\otimes} \sigma_{i,B}$. \\ 
We have then $\textit{\bf Y}=\Phi({ \mathfrak{l}}_A,{{ \mathfrak{Y}}_{{ \mathfrak{E}}_{B}}})=\bigwedge{}_{i\in I}\epsilon^{{ \mathfrak{S}}_A}_{{ \mathfrak{l}}_A}(\sigma_{i,A})\bullet \epsilon^{{ \mathfrak{S}}_B}_{{{ \mathfrak{Y}}_{{ \mathfrak{E}}_{B}}}}(\sigma_{i,B})=\bigwedge{}_{i\in I}\epsilon^{{ \mathfrak{S}}_A}_{{ \mathfrak{l}}_A}(\sigma_{i,A})\bullet \textit{\bf Y}=\bigwedge{}_{i\in I}\epsilon^{{ \mathfrak{S}}_A}_{{ \mathfrak{l}}_A}(\sigma_{i,A})=\epsilon^{{ \mathfrak{S}}_A}_{{ \mathfrak{l}}_A}(\bigsqcap{}^{{}^{{ \mathfrak{S}}_A}}_{i\in I}\sigma_{i,A})$. As a consequence, we obtain $\epsilon^{{ \mathfrak{S}}_A}_{{ \mathfrak{l}}_A}(\sigma_{i,A})=\textit{\bf Y}$ for any $i\in I$. As a result, we obtain $\Phi({ \mathfrak{l}}_A,{ \mathfrak{l}}_B)=\bigwedge{}_{i\in I}\epsilon^{{ \mathfrak{S}}_A}_{{ \mathfrak{l}}_A}(\sigma_{i,A})\bullet \epsilon^{{ \mathfrak{S}}_B}_{{ \mathfrak{l}}_B}(\sigma_{i,B})=\bigwedge{}_{i\in I}\textit{\bf Y}\bullet \epsilon^{{ \mathfrak{S}}_B}_{{ \mathfrak{l}}_B}(\sigma_{i,B})=\bigwedge{}_{i\in I} \epsilon^{{ \mathfrak{S}}_B}_{{ \mathfrak{l}}_B}(\sigma_{i,B})=\epsilon^{{ \mathfrak{S}}_B}_{{ \mathfrak{l}}_B}(\bigsqcap{}^{{}^{{ \mathfrak{S}}_B}}_{i\in I}\sigma_{i,B})$. We now observe that $\Phi({ \mathfrak{l}}_A,\overline{{ \mathfrak{l}}_B})=\epsilon^{{ \mathfrak{S}}_B}_{\overline{{ \mathfrak{l}}_B}}(\bigsqcap{}^{{}^{{ \mathfrak{S}}_B}}_{i\in I}\sigma_{i,B})=\overline{\epsilon^{{ \mathfrak{S}}_B}_{{ \mathfrak{l}}_B}(\bigsqcap{}^{{}^{{ \mathfrak{S}}_B}}_{i\in I}\sigma_{i,B})}=\overline{\Phi({ \mathfrak{l}}_A,{ \mathfrak{l}}_B)}$. This concludes the proof.
\end{proof}

\begin{lemma}\label{restrictedconditions2}
Let us now fix $\Phi\in \widetilde{S}_{AB}$ and ${ \mathfrak{l}}_A\in { \mathfrak{E}}_{A}$, and let us suppose that $\Phi({ \mathfrak{l}}_A,{{ \mathfrak{Y}}_{{ \mathfrak{E}}_{B}}})=\bot$. Then, for any ${ \mathfrak{l}}_B,{ \mathfrak{l}}'_B\in { \mathfrak{E}}_{B}$ such that ${ \mathfrak{l}}_B \sqcap_{{}_{{{ \mathfrak{E}}_{B}}}} { \mathfrak{l}}'_B=\bot_{{ \mathfrak{E}}_{B}}$, we have 
\begin{eqnarray}
(\Phi({ \mathfrak{l}}_A,{ \mathfrak{l}}_B)\,,\, \Phi({ \mathfrak{l}}_A,{{ \mathfrak{l}}'_B}))\in \{(\bot,\textit{\bf N}), (\textit{\bf N},\bot),(\bot,\bot)\}.
\end{eqnarray}
\end{lemma}

\begin{proof}
Let us consider that $\Phi=\bigsqcap{}^{{}^{\widecheck{S}_{AB}}}_{i\in I}\, \iota^{\widecheck{S}_{AB}}(\sigma_{i,A},\sigma_{i,B})=\bigsqcap{}^{{}^{{ \widetilde{S}}_{AB}}}_{i\in I} \sigma_{i,A}\widetilde{\otimes} \sigma_{i,B}$. \\ 
As it has been clarified in the third case of Lemma \ref{threecasesScheck}, we have then necessarily 
$ (\Phi({{ \mathfrak{l}}_A},{ \mathfrak{l}}_B)\,,\, \Phi({{ \mathfrak{l}}_A},{{ \mathfrak{l}}'_B}))\notin \{(\textit{\bf N},\textit{\bf N}), (\textit{\bf Y},\textit{\bf Y})\}$. Let us suppose that $\Phi({{ \mathfrak{l}}_A},{ \mathfrak{l}}_B)=\textit{\bf Y}$.  Due to the expression (\ref{expressionbullet}), we have then necessarily $\epsilon^{{ \mathfrak{S}}_A}_{{ \mathfrak{l}}_A}(\sigma_{i,A})=\textit{\bf Y}$ and $\epsilon^{{ \mathfrak{S}}_B}_{{ \mathfrak{l}}_B}(\sigma_{i,B})=\textit{\bf Y}$ for any $i\in I$ and then, in particular, $\textit{\bf Y}=\bigwedge{}_{i\in I}\epsilon^{{ \mathfrak{S}}_A}_{{ \mathfrak{l}}_A}(\sigma_{i,A})=\bigwedge{}_{i\in I}\epsilon^{{ \mathfrak{S}}_A}_{{ \mathfrak{l}}_A}(\sigma_{i,A})\bullet \textit{\bf Y}=\bigwedge{}_{i\in I}\epsilon^{{ \mathfrak{S}}_A}_{{ \mathfrak{l}}_A}(\sigma_{i,A})\bullet \epsilon^{{ \mathfrak{S}}_B}_{{{ \mathfrak{Y}}_{{ \mathfrak{E}}_{B}}}}(\sigma_{i,B})=\Phi({ \mathfrak{l}}_A,{{ \mathfrak{Y}}_{{ \mathfrak{E}}_{B}}})=\bot$, which is contradictory. As a conclusion, we cannot have $\Phi({{ \mathfrak{l}}_A},{ \mathfrak{l}}_B)=\textit{\bf Y}$.  In the same way, we cannot have $\Phi({{ \mathfrak{l}}_A},{{ \mathfrak{l}}'_B})=\textit{\bf Y}$.  This concludes the proof.
\end{proof}

\begin{definition}
We define {\em the regular tensor product} ${ \mathfrak{S}}_A \widehat{\otimes} { \mathfrak{S}}_B$ as follows
\begin{eqnarray}
{ \mathfrak{S}}_A \widehat{\otimes} { \mathfrak{S}}_B & \subseteq & { \mathfrak{S}}_A \widecheck{\otimes} { \mathfrak{S}}_B\\
\forall \Phi \in { \mathfrak{S}}_A \widehat{\otimes} { \mathfrak{S}}_B && \textit{\rm we require}\nonumber\\
\forall { \mathfrak{l}}_A\in { \mathfrak{E}}_{A}, \forall { \mathfrak{l}}_B\in { \mathfrak{E}}_{B},&&{\Phi}( { \mathfrak{l}}_A,\overline{{ \mathfrak{Y}}_{{ \mathfrak{E}}_{B}}}) ={\Phi}( \overline{{ \mathfrak{Y}}_{{ \mathfrak{E}}_{A}}},{ \mathfrak{l}}_B) = \textit{\bf N},\label{regularcase0}\\
\forall { \mathfrak{l}}_A\in { \mathfrak{E}}_{A},\;\;\;\;\;\;\Phi({ \mathfrak{l}}_A,{{ \mathfrak{Y}}_{{ \mathfrak{E}}_{B}}})=\textit{\bf Y} & \Rightarrow & \forall { \mathfrak{l}}_B\in { \mathfrak{E}}_{B},\nonumber\\
&& (\Phi({ \mathfrak{l}}_A,{ \mathfrak{l}}_B)\,,\, \Phi({ \mathfrak{l}}_A,\overline{{ \mathfrak{l}}_B}))\in \{(\textit{\bf Y},\textit{\bf N}), (\textit{\bf N},\textit{\bf Y}),(\bot,\bot)\},\;\;\;\;\;\;\;\;\;\;\;\;\;\label{regularcase1a}\\
\forall { \mathfrak{l}}_B\in { \mathfrak{E}}_{B},\;\;\;\;\;\;\Phi({{ \mathfrak{Y}}_{{ \mathfrak{E}}_{A}}},{ \mathfrak{l}}_B)=\textit{\bf Y} & \Rightarrow & \forall { \mathfrak{l}}_A\in { \mathfrak{E}}_{A},\nonumber\\
&& (\Phi({ \mathfrak{l}}_A,{ \mathfrak{l}}_B)\,,\, \Phi(\overline{{ \mathfrak{l}}_A},{{ \mathfrak{l}}_B}))\in \{(\textit{\bf Y},\textit{\bf N}), (\textit{\bf N},\textit{\bf Y}),(\bot,\bot)\},\label{regularcase1b}\\
 \forall { \mathfrak{l}}_A\in { \mathfrak{E}}_{A},\;\;\;\;\;\;\Phi({ \mathfrak{l}}_A,{{ \mathfrak{Y}}_{{ \mathfrak{E}}_{B}}})=\bot & \Rightarrow &\forall { \mathfrak{l}}_B,{ \mathfrak{l}}'_B\in { \mathfrak{E}}_{B}\,\vert\,{ \mathfrak{l}}_B\sqcap_{{}_{{ \mathfrak{E}}_B}}{ \mathfrak{l}}'_B=\bot_{{}_{{ \mathfrak{E}}_B}}, \nonumber\\
 &&\;\;(\Phi({ \mathfrak{l}}_A,{ \mathfrak{l}}_B)\,,\, \Phi({ \mathfrak{l}}_A,{{ \mathfrak{l}}'_B}))\in \{(\bot,\textit{\bf N}), (\textit{\bf N},\bot),(\bot,\bot)\},\label{regularcase2a}\\
\forall { \mathfrak{l}}_B\in { \mathfrak{E}}_{B},\;\;\;\;\;\;\Phi({{ \mathfrak{Y}}_{{ \mathfrak{E}}_{A}}},{ \mathfrak{l}}_B)=\bot & \Rightarrow &\forall { \mathfrak{l}}_A,{ \mathfrak{l}}'_A\in { \mathfrak{E}}_{A}\,\vert\,{ \mathfrak{l}}_A\sqcap_{{}_{{ \mathfrak{E}}_A}}{ \mathfrak{l}}'_A=\bot_{{}_{{ \mathfrak{E}}_A}}, \nonumber\\
&&\;\; (\Phi({ \mathfrak{l}}_A,{ \mathfrak{l}}_B)\,,\, \Phi({ \mathfrak{l}}'_A,{{ \mathfrak{l}}_B}))\in \{(\bot,\textit{\bf N}), (\textit{\bf N},\bot),(\bot,\bot)\}.\label{regularcase2b}
\end{eqnarray}
\end{definition}

\begin{lemma}
${ \mathfrak{S}}_A \widehat{\otimes} { \mathfrak{S}}_B$ is a sub Inf semi-lattice of the maximal tensor product ${ \mathfrak{S}}_A \widecheck{\otimes} { \mathfrak{S}}_B$ satisfying ${ \mathfrak{S}}_A \widetilde{\otimes} { \mathfrak{S}}_B \subseteq { \mathfrak{S}}_A \widehat{\otimes} { \mathfrak{S}}_B$.
\end{lemma}
\begin{proof}
Trivial.
\end{proof}

\subsection{Regular vs. minimal tensor product}\label{subsectionregularvsminimal}

After having defined the regular tensor product, we intent to show that ${ \mathfrak{S}}_A \widehat{\otimes} { \mathfrak{S}}_B$ is free from the spurious states present in ${ \mathfrak{S}}_A \widecheck{\otimes} { \mathfrak{S}}_B$.  To be precise, we intent to show that ${ \mathfrak{S}}_A \widehat{\otimes} { \mathfrak{S}}_B$ is identical to ${ \mathfrak{S}}_A \widetilde{\otimes} { \mathfrak{S}}_B$ and then to ${ \mathfrak{S}}_A {\otimes} { \mathfrak{S}}_B$ (due to Theorem \ref{theoremsqsubseteqPAB=SAB}) as soon as ${ \mathfrak{S}}_A$ or ${ \mathfrak{S}}_B$ is a simplex (i.e. is distributive).

\begin{theorem}
\begin{eqnarray}
\textit{\rm  (${ \mathfrak{S}}_A$ or $ { \mathfrak{S}}_B$ simplex)}&\Rightarrow & { \mathfrak{S}}_A\widehat{\otimes}{ \mathfrak{S}}_B  =  { \mathfrak{S}}_A\widetilde{\otimes}{ \mathfrak{S}}_B = { \mathfrak{S}}_A {\otimes}{ \mathfrak{S}}_B.
\end{eqnarray}
\end{theorem}
\begin{proof}
In the following, ${ \mathfrak{S}}_A$ will be chosen to be equal to ${ \mathfrak{B}}$. The general case, where ${ \mathfrak{S}}_A$ is given as a generic simplex space of states, follows the same line of proof.\\
We recall that ${ \mathfrak{E}}_{ \mathfrak{B}}^{{}^{pure}}$ is simply given by $\{\, { \mathfrak{Y}}_{{}_{{ \mathfrak{E}}_{ \mathfrak{B}}}},\overline{{ \mathfrak{Y}}_{{}_{{ \mathfrak{E}}_{ \mathfrak{B}}}}},{ \mathfrak{l}}_{{}_{(\textit{\bf Y},\textit{\bf N})}},\overline{{ \mathfrak{l}}_{{}_{(\textit{\bf Y},\textit{\bf N})}}}\,\}$ and we will denote by ${ \mathfrak{u}}$ the pure effect ${ \mathfrak{l}}_{{}_{(\textit{\bf Y},\textit{\bf N})}}$.\\
We will consider an element $\Phi$ in ${ \mathfrak{S}}_A\widehat{\otimes}{ \mathfrak{S}}_B$ and we will assume that this element is maximal in ${ \mathfrak{S}}_A\widehat{\otimes}{ \mathfrak{S}}_B$.  \\
To begin, we note that, due to this maximality condition, we must necessarily have $\Phi({ \mathfrak{u}},{{ \mathfrak{Y}}_{{ \mathfrak{E}}_{B}}})\in \{\textit{\bf Y},\textit{\bf N}\}$. Indeed, let us suppose that $\Phi({ \mathfrak{u}},{{ \mathfrak{Y}}_{{ \mathfrak{E}}_{B}}})=\bot$ and let us exhibit a contradiction. \\
 
Let us define the element of ${ \mathfrak{S}}_A\widehat{\otimes}{ \mathfrak{S}}_B$ denoted $\Psi$ in terms of the following decomposition by pure effects (see Theorem \ref{decomppurestatesreducedeffectsspace}) 
\begin{eqnarray}
\Psi({ \mathfrak{l}}_A, { \mathfrak{l}}_B) &:= & \bigwedge{}_{{ \mathfrak{l}}'_A\in \underline{\;{ \mathfrak{l}}_A\;}_{{ \mathfrak{E}}_A}}\bigwedge{}_{{ \mathfrak{l}}'_B\in \underline{\;{ \mathfrak{l}}_B\;}_{{ \mathfrak{E}}_B}\;}\; \Psi({ \mathfrak{l}}'_A, { \mathfrak{l}}'_B)
\end{eqnarray}
with
\begin{eqnarray}
&&\Psi ({ \mathfrak{Y}}_{{}_{{ \mathfrak{E}}_{ \mathfrak{B}}}},{ \mathfrak{Y}}_{{}_{{ \mathfrak{E}}_{{B}}}}) := \textit{\bf Y}\\
&&\Psi (\overline{{ \mathfrak{Y}}_{{}_{{ \mathfrak{E}}_{ \mathfrak{B}}}}},{ \mathfrak{l}}_B) := \textit{\bf N}\\
\forall { \mathfrak{l}}_A\in { \mathfrak{S}}_A^{{}^{pure}},&&\Psi ({ \mathfrak{l}}_A,\overline{{ \mathfrak{Y}}_{{}_{{ \mathfrak{E}}_{{B}}}}}) := \textit{\bf N}\\
&&\Psi ({ \mathfrak{u}},{ \mathfrak{Y}}_{{}_{{ \mathfrak{E}}_{{B}}}}) := \textit{\bf Y}\\
&&\Psi (\overline{ \mathfrak{u}},{ \mathfrak{Y}}_{{}_{{ \mathfrak{E}}_{{B}}}}) := \textit{\bf N}\\
\forall { \mathfrak{l}}_B\in { \mathfrak{S}}_B^{{}^{pure}}, 
&& \Psi ({ \mathfrak{u}},{ \mathfrak{l}}_B):=\textit{\bf Y}\;\;\textit{\rm if}\;\; \Phi ({ \mathfrak{u}},\overline{{ \mathfrak{l}}_B})=\textit{\bf N}\;\;\;\;\;\;\;\;\;\;\;\;\;\;\;\;\;\;\;\;\;\;\;\;\;\;\;\;\;\;\\
\forall { \mathfrak{l}}_B\in { \mathfrak{S}}_B^{{}^{pure}}, && \Psi ({ \mathfrak{u}},{ \mathfrak{l}}_B):=\textit{\bf N}\;\;\textit{\rm if}\;\; \Phi ({ \mathfrak{u}},{ \mathfrak{l}}_B)=\textit{\bf N}\;\;\;\;\;\;\;\;\;\;\;\;\;\;\;\;\;\;\;\;\\
\forall { \mathfrak{l}}_B\in { \mathfrak{S}}_B^{{}^{pure}},  && \Psi ({ \mathfrak{u}},{ \mathfrak{l}}_B):=\bot \;\;\textit{\rm if}\;\; \Phi ({ \mathfrak{u}},{ \mathfrak{l}}_B)=\Phi ({ \mathfrak{u}},\overline{{ \mathfrak{l}}_B})=\bot\;\;\;\;\;\;\;\;\;\;\;\;\;\;\;\;\;\;\;\;\\
\forall { \mathfrak{l}}_B\in { \mathfrak{S}}_B^{{}^{pure}}, &&\Psi (\overline{ \mathfrak{u}},{ \mathfrak{l}}_B) := \textit{\bf N}\\
\forall { \mathfrak{l}}_B\in { \mathfrak{S}}_B^{{}^{pure}}, &&\Psi ({ \mathfrak{Y}}_{{}_{{ \mathfrak{E}}_{ \mathfrak{B}}}},{ \mathfrak{l}}_B) := \textit{\bf Y}\;\;\textit{\rm if}\;\; \Phi ({ \mathfrak{u}},\overline{{ \mathfrak{l}}_B})= \textit{\bf N}\\
\forall { \mathfrak{l}}_B\in { \mathfrak{S}}_B^{{}^{pure}}, &&\Psi ({ \mathfrak{Y}}_{{}_{{ \mathfrak{E}}_{ \mathfrak{B}}}},\overline{{ \mathfrak{l}}_B}) := \textit{\bf N}\;\;\textit{\rm if}\;\; \Phi ({ \mathfrak{u}},\overline{{ \mathfrak{l}}_B})= \textit{\bf N}\\
\forall { \mathfrak{l}}_B\in { \mathfrak{S}}_B^{{}^{pure}}, &&\Psi ({ \mathfrak{Y}}_{{}_{{ \mathfrak{E}}_{ \mathfrak{B}}}},{ \mathfrak{l}}_B) :=\Psi ({ \mathfrak{Y}}_{{}_{{ \mathfrak{E}}_{ \mathfrak{B}}}},\overline{{ \mathfrak{l}}_B}) := \bot\;\;\textit{\rm if}\;\; \Phi ({ \mathfrak{u}},{{ \mathfrak{l}}_B})=\Phi ({ \mathfrak{u}},\overline{{ \mathfrak{l}}_B})= \bot.\;\;\;\;\;\;\;\;\;\;\;\;\;\;\;
\end{eqnarray}
We remark immediately that 
\begin{eqnarray}
\Psi & \sqsupset_{{}_{{ \mathfrak{S}}_A\widehat{\otimes}{ \mathfrak{S}}_B}} & \Phi.
\end{eqnarray}
We have then obtained a contradiction as long as $\Phi$ was assumed to be maximal. As a conclusion, we have then obtained $\Phi({ \mathfrak{u}},{{ \mathfrak{Y}}_{{ \mathfrak{E}}_{B}}})\in \{\textit{\bf Y},\textit{\bf N}\}$. \\

Let us then consider that 
\begin{eqnarray}
\Phi({ \mathfrak{u}},{{ \mathfrak{Y}}_{{ \mathfrak{E}}_{B}}})=\textit{\bf Y}. \label{firstcasea}
\end{eqnarray}
From (\ref{firstcasea}) and using property (\ref{furtherquotient1}), we have then immediately 
\begin{eqnarray}
\Phi(\overline{ \mathfrak{u}},{{ \mathfrak{Y}}_{{ \mathfrak{E}}_{B}}})=\textit{\bf N}.
\end{eqnarray}
and then, using (\ref{regularcase0}) and  (\ref{bilinear2})
\begin{eqnarray}
\Phi(\overline{ \mathfrak{u}},{\bot_{{ \mathfrak{E}}_{B}}})=\textit{\bf N}.\label{secondconclusion}
\end{eqnarray}
and then, using once again (\ref{bilinear2})
\begin{eqnarray}
 \forall { \mathfrak{l}}_B\in { \mathfrak{E}}_{B}, &&\Phi(\overline{ \mathfrak{u}},{{ \mathfrak{l}}_{{B}}})=\Phi(\overline{ \mathfrak{u}},\overline{{ \mathfrak{l}}_{{B}}})=\textit{\bf N}.\label{firstcaseconclusionubar}
\end{eqnarray}
From requirement (\ref{regularcase1a}), we obtain 
\begin{eqnarray}
\forall { \mathfrak{l}}_B\in { \mathfrak{E}}_{B},\;\; (\Phi({ \mathfrak{u}},{ \mathfrak{l}}_B)\,,\, \Phi({ \mathfrak{u}},\overline{{ \mathfrak{l}}_B}))\in \{(\textit{\bf Y},\textit{\bf N}), (\textit{\bf N},\textit{\bf Y}),(\bot,\bot)\}. \label{firstcaserequirement}
\end{eqnarray}
From this result, we note that the map ${ \mathfrak{l}}_B \mapsto \Phi({ \mathfrak{u}},{ \mathfrak{l}}_B)$ satisfies the requirements of 
Theorem \ref{blepsilonsigma}, i.e.  
\begin{eqnarray}
\forall \{{ \mathfrak{l}}_{i,B}\;\vert\; i\in I\}\subseteq { \mathfrak{E}}_B,&&  \Phi({ \mathfrak{u}},{{\bigsqcap{}^{{}^{ \mathfrak{E}_B}}_{{}_{i\in i}}{ \mathfrak{l}}_{i,B}}}) =\bigwedge{}_{{i\in I}} \;\Phi({ \mathfrak{u}}, {{{ \mathfrak{l}}_{i,B}}}),\\
\forall { \mathfrak{l}}_B\in { \mathfrak{E}}_B,&& \Phi({ \mathfrak{u}}, {\overline{\; \mathfrak{l}_B\;}})=\overline{\Phi({ \mathfrak{u}}, {{ \mathfrak{l}_B}})},\\
& & \Phi({ \mathfrak{u}}, {{ \mathfrak{Y}}_{ \mathfrak{E}}})=\textit{\bf Y}.
\end{eqnarray}
Then, we have
\begin{eqnarray}
\existunique \; \sigma_{ \mathfrak{u}}\in { \mathfrak{S}}_B & \vert & \forall { \mathfrak{l}}_B\in { \mathfrak{E}_B}, \; { \epsilon}^{ \mathfrak{S}_B}_{{ \mathfrak{l}_B}}(\sigma_{ \mathfrak{u}})=\Phi({ \mathfrak{u}}, {{{ \mathfrak{l}}_{B}}}).
\end{eqnarray}
By the way, using the same theorem, we note also obviously that 
\begin{eqnarray}
\existunique \; \sigma_{{}_{{{ \mathfrak{Y}}_{{ \mathfrak{E}}_{\mathfrak{B}}}}}}\in { \mathfrak{S}}_B & \vert & \forall { \mathfrak{l}}_B\in { \mathfrak{E}_B}, \; { \epsilon}^{ \mathfrak{S}_B}_{{ \mathfrak{l}_B}}(\sigma_{{}_{{{ \mathfrak{Y}}_{{ \mathfrak{E}}_{\mathfrak{B}}}}}})=\Phi({{ \mathfrak{Y}}_{{ \mathfrak{E}}_{\mathfrak{B}}}}, {{{ \mathfrak{l}}_{B}}}).
\end{eqnarray}
Let us now prove that $\sigma_{{}_{{{ \mathfrak{Y}}_{{ \mathfrak{E}}_{\mathfrak{B}}}}}}=\sigma_{ \mathfrak{u}}$.\\
In (\ref{firstcaserequirement}), we distinguish three cases. In the case where
\begin{eqnarray}
(\Phi({ \mathfrak{u}},{ \mathfrak{l}}_B)\,,\, \Phi({ \mathfrak{u}},\overline{{ \mathfrak{l}}_B}))=(\textit{\bf Y},\textit{\bf N})
\end{eqnarray}
 we deduce, using (\ref{firstcaseconclusionubar}), that 
\begin{eqnarray}
(\Phi(\bot,{ \mathfrak{l}}_B)\,,\, \Phi(\bot,\overline{{ \mathfrak{l}}_B}))=(\bot,\textit{\bf N}).
\end{eqnarray}
 Using (\ref{bilinear1}), we first deduce that $(\Phi({{ \mathfrak{Y}}_{{ \mathfrak{E}}_{\mathfrak{B}}}},\overline{{ \mathfrak{l}}_B})\,,\, \Phi({\overline{ \mathfrak{Y}}_{{ \mathfrak{E}}_{\mathfrak{B}}}},\overline{{ \mathfrak{l}}_B}))=(\textit{\bf N},\textit{\bf N})$.  Now, using (\ref{furtherquotient2}),we conclude that 
 \begin{eqnarray}
 (\Phi({{ \mathfrak{Y}}_{{ \mathfrak{E}}_{\mathfrak{B}}}},{ \mathfrak{l}}_B)\,,\, \Phi({{ \mathfrak{Y}}_{{ \mathfrak{E}}_{\mathfrak{B}}}},\overline{{ \mathfrak{l}}_B}))=(\textit{\bf Y},\textit{\bf N}).
 \end{eqnarray}
Reciprocally, if $(\Phi({{ \mathfrak{Y}}_{{ \mathfrak{E}}_{\mathfrak{B}}}},{ \mathfrak{l}}_B)\,,\, \Phi({{ \mathfrak{Y}}_{{ \mathfrak{E}}_{\mathfrak{B}}}},\overline{{ \mathfrak{l}}_B}))=(\textit{\bf Y},\textit{\bf N})$,  we deduce, using (\ref{regularcase0}) and (\ref{bilinear1}) that 
\begin{eqnarray}
(\Phi(\bot,{ \mathfrak{l}}_B)\,,\, \Phi(\bot,\overline{{ \mathfrak{l}}_B}))=(\bot,\textit{\bf N}).
\end{eqnarray}
Using (\ref{bilinear1}), we first deduce that $(\Phi({ \mathfrak{u}},\overline{{ \mathfrak{l}}_B})\,,\, \Phi(\overline{ \mathfrak{u}},\overline{{ \mathfrak{l}}_B}))=(\textit{\bf N},\textit{\bf N})$.  Now, using $\Phi({{ \mathfrak{Y}}_{{ \mathfrak{E}}_{\mathfrak{B}}}},{ \mathfrak{l}}_B)=\textit{\bf Y}$ and requirement (\ref{regularcase1b}), we deduce that 
\begin{eqnarray}
 (\Phi({ \mathfrak{u}},{ \mathfrak{l}}_B)\,,\, \Phi(\overline{{ \mathfrak{u}}},{{ \mathfrak{l}}_B}))\in \{(\textit{\bf Y},\textit{\bf N}), (\textit{\bf N},\textit{\bf Y}),(\bot,\bot)\}.\label{requirementagain}
\end{eqnarray}
However, from (\ref{secondconclusion}) and $\Phi(\overline{ \mathfrak{u}},\overline{{ \mathfrak{l}}_B})=\textit{\bf N}$, we deduce, using (\ref{bilinear2}), that  $\Phi(\overline{{ \mathfrak{u}}},{{ \mathfrak{l}}_B})=\textit{\bf N}$, and then, from (\ref{requirementagain}), necessarily $\Phi({ \mathfrak{u}},{ \mathfrak{l}}_B)=\textit{\bf Y}$. As a conclusion, we have obtained the equivalence between $(\Phi({ \mathfrak{u}},{ \mathfrak{l}}_B)\,,\, \Phi({ \mathfrak{u}},\overline{{ \mathfrak{l}}_B}))=(\textit{\bf Y},\textit{\bf N})$ and $(\Phi({{ \mathfrak{Y}}_{{ \mathfrak{E}}_{\mathfrak{B}}}},{ \mathfrak{l}}_B)\,,\, \Phi({{ \mathfrak{Y}}_{{ \mathfrak{E}}_{\mathfrak{B}}}},\overline{{ \mathfrak{l}}_B}))=(\textit{\bf Y},\textit{\bf N})$.\\
In the same way, we obtain that $(\Phi({{ \mathfrak{Y}}_{{ \mathfrak{E}}_{\mathfrak{B}}}},{ \mathfrak{l}}_B)\,,\, \Phi({{ \mathfrak{Y}}_{{ \mathfrak{E}}_{\mathfrak{B}}}},\overline{{ \mathfrak{l}}_B}))=(\textit{\bf N},\textit{\bf Y})$ is equivalent to $(\Phi({ \mathfrak{u}},{ \mathfrak{l}}_B)\,,\, \Phi({ \mathfrak{u}},\overline{{ \mathfrak{l}}_B}))=(\textit{\bf N},\textit{\bf Y})$. And obviously the only remaining case is obtained, i.e. $(\Phi({{ \mathfrak{Y}}_{{ \mathfrak{E}}_{\mathfrak{B}}}},{ \mathfrak{l}}_B)\,,\, \Phi({{ \mathfrak{Y}}_{{ \mathfrak{E}}_{\mathfrak{B}}}},\overline{{ \mathfrak{l}}_B}))=(\bot,\bot)$ is equivalent to $(\Phi({ \mathfrak{u}},{ \mathfrak{l}}_B)\,,\, \Phi({ \mathfrak{u}},\overline{{ \mathfrak{l}}_B}))=(\bot,\bot)$. \\
As a conclusion, we have obtained $(\Phi({{ \mathfrak{Y}}_{{ \mathfrak{E}}_{\mathfrak{B}}}},{ \mathfrak{l}}_B)\,,\, \Phi({{ \mathfrak{Y}}_{{ \mathfrak{E}}_{\mathfrak{B}}}},\overline{{ \mathfrak{l}}_B}))=(\Phi({ \mathfrak{u}},{ \mathfrak{l}}_B)\,,\, \Phi({ \mathfrak{u}},\overline{{ \mathfrak{l}}_B}))$ for any ${ \mathfrak{l}}_B\in { \mathfrak{E}}_B$, i.e. 
\begin{eqnarray}
\forall { \mathfrak{l}}_B\in { \mathfrak{E}_B}, && { \epsilon}^{ \mathfrak{S}_B}_{{ \mathfrak{l}_B}}(\sigma_{{}_{{{ \mathfrak{Y}}_{{ \mathfrak{E}}_{\mathfrak{B}}}}}})={ \epsilon}^{ \mathfrak{S}_B}_{{ \mathfrak{l}_B}}(\sigma_{ \mathfrak{u}}).
\end{eqnarray}
In other words, 
\begin{eqnarray}
\sigma_{{}_{{{ \mathfrak{Y}}_{{ \mathfrak{E}}_{\mathfrak{B}}}}}}=\sigma_{ \mathfrak{u}}.
\end{eqnarray}
It is clear that, as long as $\Phi$ is chosen to be maximal, $\sigma_{ \mathfrak{u}}$ must be a maximal state.\\
We can then summarize previous results :
\begin{eqnarray}
\Phi({ \mathfrak{u}},{{ \mathfrak{Y}}_{{ \mathfrak{E}}_{B}}})=\textit{\bf Y} & \Rightarrow &\existunique \sigma\in { \mathfrak{S}}^{{}^{pure}}_B\;\vert\; \forall { \mathfrak{l}}_B\in { \mathfrak{E}_B}, \;\; \left\{
\begin{array}{l} \Phi({{ \mathfrak{Y}}_{{ \mathfrak{E}}_{\mathfrak{B}}}},{ \mathfrak{l}}_B)= \epsilon_{{ \mathfrak{Y}}_{{ \mathfrak{E}}_{\mathfrak{B}}}}^{\mathfrak{B}}(\textit{\bf Y}) \bullet { \epsilon}^{ \mathfrak{S}_B}_{{ \mathfrak{l}_B}}(\sigma)\\
\Phi({ \mathfrak{u}},{ \mathfrak{l}}_B)= \epsilon_{{ \mathfrak{u}}}^{\mathfrak{B}}(\textit{\bf Y}) \bullet { \epsilon}^{ \mathfrak{S}_B}_{{ \mathfrak{l}_B}}(\sigma)\\
\Phi(\overline{ \mathfrak{u}},{ \mathfrak{l}}_B)= \epsilon_{\overline{ \mathfrak{u}}}^{\mathfrak{B}}(\textit{\bf Y}) \bullet { \epsilon}^{ \mathfrak{S}_B}_{{ \mathfrak{l}_B}}(\sigma)\\
\Phi(\overline{{{ \mathfrak{Y}}_{{ \mathfrak{E}}_{\mathfrak{B}}}}},{ \mathfrak{l}}_B)=  \epsilon_{\overline{{ \mathfrak{Y}}_{{ \mathfrak{E}}_{\mathfrak{B}}}}}^{\mathfrak{B}}(\textit{\bf Y}) \bullet { \epsilon}^{ \mathfrak{S}_B}_{{ \mathfrak{l}_B}}(\sigma)
\end{array}\right.\;\;\;\;\;\;\;\;\;
\end{eqnarray}
Following an analogous way, we obtain 
\begin{eqnarray}
\Phi({ \mathfrak{u}},{{ \mathfrak{Y}}_{{ \mathfrak{E}}_{B}}})=\textit{\bf N} & \Rightarrow &\existunique \sigma\in { \mathfrak{S}}^{{}^{pure}}_B\;\vert\; \forall { \mathfrak{l}}_B\in { \mathfrak{E}_B}, \;\; \left\{
\begin{array}{l} \Phi({{ \mathfrak{Y}}_{{ \mathfrak{E}}_{\mathfrak{B}}}},{ \mathfrak{l}}_B)= \epsilon_{{ \mathfrak{Y}}_{{ \mathfrak{E}}_{\mathfrak{B}}}}^{\mathfrak{B}}(\textit{\bf N}) \bullet { \epsilon}^{ \mathfrak{S}_B}_{{ \mathfrak{l}_B}}(\sigma)\\
\Phi({ \mathfrak{u}},{ \mathfrak{l}}_B)= \epsilon_{{ \mathfrak{u}}}^{\mathfrak{B}}(\textit{\bf N}) \bullet { \epsilon}^{ \mathfrak{S}_B}_{{ \mathfrak{l}_B}}(\sigma)\\
\Phi(\overline{ \mathfrak{u}},{ \mathfrak{l}}_B)= \epsilon_{\overline{ \mathfrak{u}}}^{\mathfrak{B}}(\textit{\bf N}) \bullet { \epsilon}^{ \mathfrak{S}_B}_{{ \mathfrak{l}_B}}(\sigma)\\
\Phi(\overline{{{ \mathfrak{Y}}_{{ \mathfrak{E}}_{\mathfrak{B}}}}},{ \mathfrak{l}}_B)=  \epsilon_{\overline{{ \mathfrak{Y}}_{{ \mathfrak{E}}_{\mathfrak{B}}}}}^{\mathfrak{B}}(\textit{\bf N}) \bullet { \epsilon}^{ \mathfrak{S}_B}_{{ \mathfrak{l}_B}}(\sigma)
\end{array}\right.\;\;\;\;\;\;\;\;\;
\end{eqnarray}
As a final conclusion, as long as $\Phi$ is maximal, $\Phi$ is necessarily a maximal element in ${ \mathfrak{B}}\widetilde{\otimes} { \mathfrak{S}}_B$ (see Theorem \ref{theorempuretilde}). Then,  we obtain the final conclusion :
\begin{eqnarray}
{ \mathfrak{B}}\widehat{\otimes} { \mathfrak{S}}_B & = & { \mathfrak{B}}\widetilde{\otimes} { \mathfrak{S}}_B.
\end{eqnarray}
The equality ${ \mathfrak{B}}{\otimes} { \mathfrak{S}}_B  =  { \mathfrak{B}}\widetilde{\otimes} { \mathfrak{S}}_B$ had already been obtained in Theorem \ref{theoremsqsubseteqPAB=SAB}.
\end{proof}

\begin{theorem}
$({ \mathfrak{S}}_{A},{ \mathfrak{E}}_{A},\epsilon^{{ \mathfrak{S}}_{A}})$ and $({ \mathfrak{S}}_{B},{ \mathfrak{E}}_{B},\epsilon^{{ \mathfrak{S}}_{B}})$ are states-effects Chu spaces where ${ \mathfrak{S}}_A$ and ${ \mathfrak{S}}_B$ are chosen to be non-simplex orthocomplemented spaces of states and where ${ \mathfrak{E}}_{A}$ and ${ \mathfrak{E}}_{B}$ are chosen respectively to be the reduced effects spaces $\overline{ \mathfrak{E}}_{{ \mathfrak{S}}_A}$ and $\overline{ \mathfrak{E}}_{{ \mathfrak{S}}_B}$. The tensor products ${ \mathfrak{S}}_{A} \widecheck{\otimes} { \mathfrak{S}}_{B}$ and ${ \mathfrak{S}}_{A} \widetilde{\otimes} { \mathfrak{S}}_{B}$ are defined respectively according to subsection \ref{subsectionmaximaltensorproduct} and subsection \ref{subsectionminimaltensorproduct} and the regular tensor product ${ \mathfrak{S}}_{A} \widehat{\otimes} { \mathfrak{S}}_{B}$ is defined according to subsection \ref{subsectionremarks} with these choices.  As a result, we have
\begin{eqnarray}
{ \mathfrak{S}}_A\widetilde{\otimes}{ \mathfrak{S}}_B & \varsubsetneq & { \mathfrak{S}}_A\widehat{\otimes}{ \mathfrak{S}}_B.
\end{eqnarray}
\end{theorem}
\begin{proof}
${ \mathfrak{S}}_A$ being a non-simplex orthocomplemented space of states, there exists a pair of states $\alpha_1,\alpha_2$ in ${ \mathfrak{S}}_A$ such that $\alpha_1^\star\not\sqsubseteq_{{}_{{ \mathfrak{S}}_A}} \alpha_2$ and $\alpha_1\not\sqsubseteq_{{}_{{ \mathfrak{S}}_A}} \alpha_2$. Indeed, if it was wrong that such a pair exist, we could, for any $\alpha$, define a morphism $\psi_\alpha$ defined by $\forall \sigma\in { \mathfrak{S}}_A,\;\psi_\alpha(\sigma):=\epsilon^{{ \mathfrak{S}}_A}_{{ \mathfrak{l}}_{(\alpha,\alpha^\star)}}(\sigma)$, which would satisfy $\psi_\alpha(\alpha)=\textit{\bf Y}$ and $\psi_\alpha({ \mathfrak{S}}^{{}^{pure}}_A)\subseteq \{\textit{\bf Y},\textit{\bf N}\}$ In other words, this would impose ${ \mathfrak{S}}_A$ to be a simplex (this result is a consequence of Lemma \ref{particularpsi}), which is false by assumption. Let us then fix such a pair $(\alpha_1,\alpha_2)$. Let us consider the subset $I_{\alpha_1}$ of $\underline{\alpha_1}_{{}_{{ \mathfrak{S}}_A}}$ defined by $I_{\alpha_1}:=\{\,\alpha\in \underline{\alpha_1}_{{}_{{ \mathfrak{S}}_A}}\;\vert\; \alpha^\star \sqsubseteq_{{}_{{ \mathfrak{S}}_A}} \alpha_2\}$. We  have necessarily $I_{\alpha_1} \varsubsetneq \underline{\alpha_1}_{{}_{{ \mathfrak{S}}_A}}$, because $I_{\alpha_1} =\underline{\alpha_1}_{{}_{{ \mathfrak{S}}_A}}$ would imply $\alpha_1^\star=(\bigsqcap{}^{{}^{{ \mathfrak{S}}_A}}I_{\alpha_1} )^\star \sqsubseteq_{{}_{{ \mathfrak{S}}_A}} \alpha_2$ which is false by assumption. Then, we can always choose $\sigma_1\in (\underline{\alpha_1}_{{}_{{ \mathfrak{S}}_A}}\!\!\!\!\!\!\!\smallsetminus I_{\alpha_1}) \subseteq { \mathfrak{S}}^{{}^{pure}}_A$ and $\sigma_2 \in \underline{\alpha_2}_{{}_{{ \mathfrak{S}}_A}}\subseteq { \mathfrak{S}}^{{}^{pure}}_A$ such that $\sigma_1^\star\not\sqsubseteq_{{}_{{ \mathfrak{S}}_A}} \sigma_2$ and $\sigma_1\not= \sigma_2$.  We note that $\sigma_1^\star$ is an atom of ${ \mathfrak{S}}_A$. 
Analogously, we can choose a pair of states $\tau_1,\tau_2$ in ${ \mathfrak{S}}^{{}^{pure}}_B$ such that $\tau_1^\star\not\sqsubseteq_{{}_{{ \mathfrak{S}}_B}} \tau_2$ and $\tau_1\not= \tau_2$.  We note also that $\tau_1^\star$ is an atom of ${ \mathfrak{S}}_B$.\\

It is now rather easy to check that the following supremum exists in $\widecheck{S}_{AB}$
\begin{eqnarray}
\Sigma & := & (\sigma_1 \widetilde{\otimes} \tau_1 \sqcap_{{}_{\widetilde{S}_{AB}}} \sigma_2 \widetilde{\otimes} \tau_2) \sqcup_{{}_{\widecheck{S}_{AB}}} (\sigma_1^\star  \widetilde{\otimes} \bot_{{}_{ \mathfrak{S}_B}} \sqcap_{{}_{\widetilde{S}_{AB}}} \bot_{{}_{ \mathfrak{S}_A}} \widetilde{\otimes} \tau_1^\star)\label{expressformalSigma}
\end{eqnarray}
the explicit expression is indeed simply given in terms of the following decomposition by pure effects (see Theorem \ref{decomppurestatesreducedeffectsspace}) :
\begin{eqnarray}
&&\Sigma({ \mathfrak{l}}_A, { \mathfrak{l}}_B) :=  \bigwedge{}_{{ \mathfrak{l}}'_A\in \underline{\;{ \mathfrak{l}}_A\;}_{{ \mathfrak{E}}_A}}\bigwedge{}_{{ \mathfrak{l}}'_B\in \underline{\;{ \mathfrak{l}}_B\;}_{{ \mathfrak{E}}_B}\;}\; \Sigma({ \mathfrak{l}}'_A, { \mathfrak{l}}'_B)\label{expressexplicitSigma1}
\\
&&\left\{
\begin{array}{ll}
\Sigma ({ \mathfrak{Y}}_{{ \mathfrak{E}}_{A}},{ \mathfrak{Y}}_{{ \mathfrak{E}}_{B}}) = \textit{\bf Y}&\\
 \Sigma (\overline{{ \mathfrak{Y}}_{{ \mathfrak{E}}_{A}}},{ \mathfrak{l}}_{{B}}) =  \Sigma ({ \mathfrak{l}}_{{A}},\overline{{ \mathfrak{Y}}_{{ \mathfrak{E}}_{B}}}) = \textit{\bf N}, & \forall { \mathfrak{l}}_{{A}}\in { \mathfrak{E}}^{{}^{pure}}_{A},\forall { \mathfrak{l}}_{{B}}\in { \mathfrak{E}}^{{}^{pure}}_{B},\\
 \Sigma ({{ \mathfrak{Y}}_{{ \mathfrak{E}}_{A}}},{ \mathfrak{l}}_{{B}}) = \Sigma ({ \mathfrak{l}}_{{A}},{{ \mathfrak{Y}}_{{ \mathfrak{E}}_{B}}}) = \bot, & \forall { \mathfrak{l}}_{{A}}\in { \mathfrak{E}}^{{}^{pure}}_{A},\forall { \mathfrak{l}}_{{B}}\in { \mathfrak{E}}^{{}^{pure}}_{B},\\
\Sigma ({ \mathfrak{l}}_{(\sigma_1,\sigma_1^\star)},{ \mathfrak{l}}_{(\tau_1,\tau_1^\star)})= \textit{\bf N}\\
 \Sigma ({ \mathfrak{l}}_{(\sigma_1^\star,\sigma_1)},{ \mathfrak{l}}_{(\tau_2^\star,\tau_2)})= \textit{\bf N}\\
\Sigma ({ \mathfrak{l}}_{(\sigma_2^\star,\sigma_2)},{ \mathfrak{l}}_{(\tau_1^\star,\tau_1)})= \textit{\bf N}\\
 \Sigma ({ \mathfrak{l}}_{{A}},{ \mathfrak{l}}_{{B}}) = \bot & \textit{\rm for any other pair $({ \mathfrak{l}}_{{A}},{ \mathfrak{l}}_{{B}})$}\in { \mathfrak{E}}^{{}^{pure}}_{A}\times { \mathfrak{E}}^{{}^{pure}}_{B}.
\end{array}
\right.\label{expressexplicitSigma2}
\end{eqnarray}
Using (\ref{developmentetildeordersimplify}),  $\sigma_1^\star\not\sqsubseteq_{{}_{{ \mathfrak{S}}_A}} \sigma_1$, $\sigma_1^\star\not\sqsubseteq_{{}_{{ \mathfrak{S}}_A}} \sigma_2$,  $\sigma_1\not=\sigma_2$,  $\tau_1^\star\not\sqsubseteq_{{}_{{ \mathfrak{S}}_B}} \tau_1$, $\tau_1^\star\not\sqsubseteq_{{}_{{ \mathfrak{S}}_B}} \tau_2$ and $\tau_1\not= \tau_2$, we deduce 
\begin{eqnarray}
(\sigma_1^\star  \widetilde{\otimes} \bot_{{}_{ \mathfrak{S}_B}} \sqcap_{{}_{\widetilde{S}_{AB}}} \bot_{{}_{ \mathfrak{S}_A}} \widetilde{\otimes} \tau_1^\star) & \not\sqsubseteq_{{}_{\widetilde{S}_{AB}}} & \sigma_1 \widetilde{\otimes} \tau_1, \\
(\sigma_1^\star  \widetilde{\otimes} \bot_{{}_{ \mathfrak{S}_B}} \sqcap_{{}_{\widetilde{S}_{AB}}} \bot_{{}_{ \mathfrak{S}_A}} \widetilde{\otimes} \tau_1^\star) & \not\sqsubseteq_{{}_{\widetilde{S}_{AB}}} & \sigma_2 \widetilde{\otimes} \tau_2,
\end{eqnarray}
and $\underline{(\sigma_1 \widetilde{\otimes} \tau_1 \sqcap_{{}_{\widetilde{S}_{AB}}} \sigma_2 \widetilde{\otimes} \tau_2)}_{{}_{\widetilde{S}_{AB}}}  =  \{\, \sigma_1 \widetilde{\otimes} \tau_1\,,\, \sigma_2 \widetilde{\otimes} \tau_2\,\}$ (because of Theorem \ref{Stildereducible}).\\ 
As a consequence, we conclude that 
\begin{eqnarray}
\Sigma & \in & { \mathfrak{S}}_A\widehat{\otimes}{ \mathfrak{S}}_B \smallsetminus { \mathfrak{S}}_A\widetilde{\otimes}{ \mathfrak{S}}_B.
\end{eqnarray}
\end{proof}
\begin{remark}
For the same reasons, we note that $\Sigma$ is not an element of the canonical tensor product ${ \mathfrak{S}}_A{\otimes}{ \mathfrak{S}}_B$ either. This is a strong argument in favor of the adoption of ${ \mathfrak{S}}_A\widehat{\otimes}{ \mathfrak{S}}_B$ as the fundamental tensor product of generic ${ \mathfrak{S}}_A$ and ${ \mathfrak{S}}_B$.
\end{remark}

\subsection{Channels of the bipartite experiments}\label{subsectionsymmetriesbipartite}

Let us consider a channel $(f, f^{\ast})$ from a States/Effects Chu space $({ \mathfrak{S}}_{A_1},{ \mathfrak{E}}_{{A_1}},\epsilon^{{ \mathfrak{S}}_{A_1}})$ to another States/Effects Chu space $({ \mathfrak{S}}_{A_2},{ \mathfrak{E}}_{{A_2}},\epsilon^{{ \mathfrak{S}}_{A_2}})$. Let us also consider a channel $(g, g^{\ast})$ from the Chu space $({ \mathfrak{S}}_{B_1},{ \mathfrak{E}}_{{B_1}},\epsilon^{{ \mathfrak{S}}_{B_1}})$ to the Chu space $({ \mathfrak{S}}_{B_2},{ \mathfrak{E}}_{{B_2}},\epsilon^{{ \mathfrak{S}}_{B_2}})$. \\ 

We define the channel $((f\widehat{\otimes} g), (f\widehat{\otimes} g)^{\ast})$ from the States/Effects Chu space $({\widehat{S}}_{A_1B_1},{\widehat{E}}_{{A_1B_1}},\epsilon^{{\widehat{S}}_{A_1B_1}})$ to the States/Effects Chu space $({\widehat{S}}_{A_2B_2},{\widehat{E}}_{{A_2B_2}},\epsilon^{{\widehat{S}}_{A_2B_2}})$ by
\begin{eqnarray}
((f\widehat{\otimes} g)({\Sigma}))({ \mathfrak{l}}_A,{ \mathfrak{l}}_B) & := & {\Sigma}(f^{\ast}({ \mathfrak{l}}_{A}),g^{\ast}({ \mathfrak{l}}_{B}))
\end{eqnarray}

We define the channel $((f\widetilde{\otimes} g), (f\widetilde{\otimes} g)^{\ast})$ from the States/Effects Chu space $({\widetilde{S}}_{A_1B_1},{\widetilde{E}}_{{A_1B_1}},\epsilon^{{\widetilde{S}}_{A_1B_1}})$ to the States/Effects Chu space $({\widetilde{S}}_{A_2B_2},{\widetilde{E}}_{{A_2B_2}},\epsilon^{{\widetilde{S}}_{A_2B_2}})$ by
\begin{eqnarray}
(f\widetilde{\otimes} g)(\bigsqcap{}^{{}^{{ \widetilde{S}}_{A_1B_1}}}_{i\in I} \sigma_{i,A_1}\widetilde{\otimes} \sigma_{i,B_1}) & := & \bigsqcap{}^{{}^{{ \widetilde{S}}_{A_2B_2}}}_{i\in I} f(\sigma_{i,A_1})\widetilde{\otimes} g(\sigma_{i,B_1}).
\end{eqnarray}

\section{Conclusion}

The study of the tensor product of semi-lattices is an old story which has debuted by G.A.Fraser's proposal \cite{Fraser1976}\cite{Fraser1978}. Although this construction is particularly interesting for its universal nature, it is disappointing in many respects. \\
In the present paper, we are interested by Inf semi-lattices with a bottom element. These spaces are called "spaces of states". We plug any of these spaces of states as the carrier space of a bi-extensional Chu space. This Chu space is associated to a target space given by the three elements boolean Inf semi-lattice denoted ${ \mathfrak{B}}$ and defined in our preamble. The co-carrier of this Chu space is called "space of effects".  After a general introduction of these elements (section \ref{sectiongeneral}), we proceed to the description of our tensor products constructions and to the analysis of their properties (Section \ref{sectionnewperspective}). 
Hence, we naturally define the tensor product of two spaces of states as the set of bimorphic maps defined from the cartesian product of the associated spaces of effects to the target space ${ \mathfrak{B}}$. By defining on the target space ${ \mathfrak{B}}$ a commutative monoid law, which is distributive with respect to the conjunction law,  we show that the pure tensors can be embedded into this tensor product space. This construction leads to the definition of two extreme tensor products called minimal and maximal tensor products (subsections \ref{subsectionminimaltensorproduct} and \ref{subsectionmaximaltensorproduct}). After having explored the properties of the minimal tensor product (subsection \ref{subsectionpropertiesminimal}), and in particular analized its relation with canonical tensor product, we proceed to a careful analysis of the maximal tensor product and we are led to define a regularized tensor product in subsection \ref{subsectionremarks}.  The subsection \ref{subsectionregularvsminimal} is dedicated to the study of this regular tensor product. In particular, it is proved that the regular tensor product of two spaces of states is equal to their canonical (and minimal) tensor product as soon as one of these spaces of states is distributive.   It is also proved that, when both spaces of states are non-distributive, the regular tensor product is strictly larger than the minimal tensor product and is different from the canonical tensor product. This is a strong argument for the adoption of our regular tensor product as the fundamental tensor product of a pair of Inf semi-lattices with bottom element.

\section*{Statements and Declarations}
The author did not receive any support from any organization for the present work. The author declares that there is no conflict of interest. There is no associated data to this article.



\end{document}